\documentclass[10pt]{amsart}
\usepackage{amscd, amsfonts, amsthm, amsgen, amsmath, amssymb ,verbatim, enumerate, textcomp}
\usepackage{hyperref} 
\usepackage[cmtip, all]{xy}
\usepackage{graphicx}
\usepackage[margin=1in]{geometry}

\usepackage{amscd}
\usepackage{xypic}
\usepackage{amsmath, amssymb}

\newtheorem{theorem}{Theorem}[section]

\newtheorem{lemma}[theorem]{Lemma}

\newtheorem{proposition}[theorem]{Proposition}
\theoremstyle{definition}
\newtheorem{definition}[theorem]{Definition}
\theoremstyle{notation}

\newtheorem{notation}[theorem]{Notation}
\newtheorem{example}[theorem]{Example}

\newtheorem{terminology}[theorem]{Terminology}
\theoremstyle{remark}
\newtheorem{remark}[theorem]{Remark}

\numberwithin{equation}{section}
\usepackage{amscd}
\usepackage{xypic}
\usepackage{amsmath, amssymb}

\numberwithin{equation}{subsection}
\newcommand{\be}%
  {\protect\setcounter{equation}{\value{subsubsection}}}
  \newcommand{\ee}%
   {\protect\setcounter{subsubsection}{\value{equation}}}

  {\protect\setcounter{subsubsection}{\value{equation}}}

\def \rma{\rm a}

\def \rmA{\rm A}

\def \cB{\mathcal B}
\def \obeta{\overset o \beta}
\def \BG{\rm BG}
\def \EG1{\rm EG}
\def \BN(T){\rm BN_{G}(T)}
\def \rmB{\rm B}
\def \BH{\rm BH}

\def \rmC{\rm C}
\def \C{\mathcal C}

\def \Cl{\mathbb C}
\def \colim{\underset \rightarrow  {\hbox {lim}}}
\def \colimm{\underset {m \rightarrow \infty}  {\hbox {lim}}}

\def \colimalpha{\underset {\alpha}  {\hbox {colim}}}

\def \colimK.{\underset {\underset K^.  \rightarrow}  {\hbox {lim}}}

\def \colimU.{\underset {\underset U_.  \rightarrow}  {\hbox {lim}}}

\def \compl{\, \, {\widehat {}}}

\def \D{\mathcal D}

\def \rmD{\rm D}
\def \rmd{\rm d}

\def \DS1X{\rm {DS^1X}}
\def \rmD{\rm D}

\def \cE{\mathcal E}

\def \EG1{E{(G \times {\mathbb C}^*)}{\underset {G\times {\mathbb C}^*} 
\times}}

\def \EZ(s)1{E{(Z(s) \times {\mathbb C}^*)}{\underset {(Z(s)\times {\mathbb
C}^*)}  \times}}

\newcommand{\eps}{ \, {\boldsymbol\varepsilon} \,}

\def \EM(u){EM(u){\underset {M(u)}  \times}}
\def \EM(us){EM(u,s){\underset {M(u, s)}  \times}}

\def \EG{\rm EG}
\def \rmE{\rm E}
\def \EH{\rm EH}

\def \et{\rm et}

\def \F{\mathcal F}

\def \rmf{\rm f}
\def \rmF{\rm F}

\def \group{\rm G}

\def \rmG{\rm G}
\def \GL{\rm {GL}}

\def \rmg{\rm g}

\def\holimD{\mathop{\textrm{holim}}\limits_{\Delta }}

\def \H{\mathbb H}
\def \rmH{\rm H}

\def \Hom{\underline {Hom}}
\def \Hom{{\mathcal H}om}

\def \rmh{\rm h}

\def \invlimn{\underset {\infty \leftarrow n}  {\hbox {lim}}}
\def \invlim1{\underset {\infty \leftarrow q}  {\hbox {lim}}^1}

\def \oI{\rm I}
\def \rmI{\rm I}

\def \oJ{\rm J}
\def \rmJ{\rm J}

\def \rmK{\rm K}
\def \k{\it k}

\def \L3{\Lambda \times \Lambda \times \Lambda}
\def \L2{\Lambda \times \Lambda}

\def \lim{\underset \leftarrow  {\hbox {lim}}}
\def \limm{\underset {\infty \leftarrow m}  {\hbox {lim}}}
\def \invlimn{\underset {\infty \leftarrow n}  {\hbox {lim}}}

\def \longright2arrow{{\overset \longrightarrow  {\overset {} 
\longrightarrow}}}

\def \L{L\times \Cl ^*}

\def \rmL{\rm L}
\def \ellh{\rm \ell.h}

\def \rmM{\rm M}

\def \cN{\mathcal N}

\def \N(T){\rm {N_{G}(T)}}
\def \rmN{\rm N}

\def \Nis{\rm {Nis}}

\def \rmp{\rm p}

\def \rmP{\rm P}

\def \Spt{\rm {Spt}}

\def \rmQ{\rm Q}
\def \rmq{\rm q}

\def \ra{\rightarrow}

\def \RG^{R(G)^{\hat {}}\ }
\def \resp{respectively}
\def \res{respectively}
\def \rmS{\rm S}

\def \RHom{{{\mathcal R}{\mathcal H}om}}

\def \R{{\mathcal R}}

\def \rmR{\rm R}

\def \Sm{\rm {Sm}}

\def \Speck{{\rm {Spec}}\, {\it k}}

\def\Spt{\rm {\bf Spt}}

\def \Sph{\rm {Sph}}

\def \SH{{\mathcal S}{\mathcal H}}
\def\Spt{\rm {\bf Spt}}
\def\Spc{\rm {\bf Spc}}

\def \mbS{\mathbb S}
\def \rmS{\rm S}
\def \rms{\rm s}

\def \topGcoh*{^{top, *} _{G}}
\def \topGho*{ _{top,*} ^{G}}

\def \T{{\mathbf T}}

\def \tr{\it {tr}}
\def \rmT{\rm T}
\def \Th{\rm Th}

\def \rmU{\rm U}
\def \U{\mathcal U}

\def \rmV{\rm V}

\def \W{\mathbf W}
\def \rmW{\rm W}

\def \rmX{\rm X}
\def \itX{\it X}

\def \X{\mathcal X}

\def \Y{\mathcal Y}
\def \rmY{\rm Y}

\def \Z(s){Z(s) \times {\mathbb C}^*}
\def \Z{\mathcal Z}
\def \bZ{\mathbb Z}

\def \rmZ{\rm Z}

\def \Zar{\rm Zar}
\def\Zl{{\mathbb Z}/\ell}

\begin{document}

\title[The Motivic and \'Etale Becker-Gottlieb transfer: the construction]{The Motivic and \'Etale Becker-Gottlieb transfer: the construction of the transfer}
\author{Gunnar Carlsson}
\address{Department of Mathematics, Stanford University, Building 380, Stanford,
California 94305}
\email{gunnar@math.stanford.edu}
\thanks{  }  
\author{Roy Joshua}
\address{Department of Mathematics, Ohio State University, Columbus, Ohio,
43210, USA}
\email{joshua@math.ohio-state.edu}
\thanks{}


\thanks{\indent The  authors were supported by  grants
from the NSF at various stages on work on this paper. The second author would also like to thank the Isaac Newton Institute for Mathematical Sciences, Cambridge, for support and hospitality during the programme {\it K-Theory, Algebraic Cycles and Motivic Homotopy Theory} where part of the work on this paper was carried out. This work was also supported by EPSRC grant no EP/R014604/1 and by the Simons Foundation.}
\begin{abstract} The main goal of the paper is the construction of a variant of the Becker-Gottlieb transfer in the motivic and \'etale frameworks. 
This needs considerable work in equivariant motivic and \'etale homotopy theory, equivariant for the action of linear algebraic groups, which is discussed
in the first half of the paper.
\end{abstract}
\maketitle

\centerline{\bf Table of contents}
\vskip .2cm 
1. Introduction
\vskip .2cm
2. Basic Model category framework for simplicial presheaves
\vskip .2cm
3. Model structures for equivariant simplicial presheaves
\vskip .2cm
4. Categories of spectra
\vskip .2cm
5. Model structures for categories of spectra
\vskip .2cm
6. Comparison of model structures
\vskip .2cm
7. Spanier-Whitehead duality in the motivic and \'etale framework
\vskip .2cm
8. Construction of the motivic and \'etale transfer
\vskip .2cm
9. Appendix: Spherical fibrations and Thom spaces
\input xypic
\vfill \eject
\section{\bf Introduction}
We begin with the following example to motivate the discussion in the present paper. One knows from the work of Atiyah (see \cite{At}) that 
for a compact manifold $\rmM$ without boundary, the suspension spectrum of the Thom-space of the normal bundle $\nu$ to imbedding $\rmM$ in a large dimensional Euclidean 
space is a Spanier-Whitehead dual to the suspension spectrum of $\rmM_+$ modulo a certain shift. Now assume that the manifold $\rmM$ is provided
with the action of a compact Lie group $\rmG$. Then one may find a $\rmG$-equivariant imbedding of $\rmM$ in a large dimensional Euclidean $\rmV= {\mathbb R}^{\rm N}$
space with a $\rmG$-action. Now the normal bundle $\nu$ and its Thom-space inherit $\rmG$-actions. The resulting Thom-Pontrjagin collapse
map from the sphere ${\rm TP}:\rmS(\rmV \oplus 1) \ra {\rm Th}(\nu)$ shows that the $\rmG$-suspension spectrum of ${\rm Th}(\nu)$ will
be a Spanier-Whitehead dual to the $\rmG$-suspension spectrum of $\rmM_+$ (again modulo certain shifts). Observe that the map ${\rm TP}$ followed by the 
diagonal map ${\rm Th}(\nu) \ra \rmM_+ \wedge {\rm Th}(\nu)$ then
provides a $\rmG$-equivariant pre-transfer map which is then fed into the Borel construction to obtain the  first step in the construction
of the classical Becker-Gottlieb transfer. (See \cite[section 3]{BG75}.)
\vskip .2cm
The point we want to emphasize here is that the Spanier-Whitehead dual used in the above construction of the transfer 
is the $\rmG$-suspension spectrum of $\Th(\nu)$, so that on forgetting the $\rmG$-action we obtain the Spanier-Whitehead dual in the
non-equivariant setting. In other words, the Spanier-Whitehead dual used in the construction of the transfer is the same Spanier-Whitehead dual
in the non-equivariant setting, but made equivariant simply by making the imbedding of $\rmM$ in a large Euclidean space $\rmG$-equivariant.
{\it One of our goals in this paper} is to set up a suitable  framework so that the Spanier -Whitehead dual of the suspension spectra of simplicial presheaves
in the motivic and \'etale setting provided with $\rmG$-action, reduce to the corresponding non-equivariant Spanier-Whitehead dual on
forgetting the $\rmG$-action. We also would like to point out that this approach seems essential, at least for now in the motivic setting, since
no variant of Gabber's refined alterations exists that is compatible with a group action.
\vskip .2cm
At the same time, since no account of equivariant (stable) motivic homotopy theory for actions of all linear algebraic groups
on smooth schemes exists yet in this generality, (analogous to the discussion in \cite{LMS} for the topological situation), we need to set up 
the basic machinery in place. One may observe that when the groups involved are finite discrete groups, it is possible to incorporate
the corresponding equivariant homotopy theory into the theory of motivic symmetric spectra: see \cite{DRO2}.
This approach fails when the groups are no longer finite. 
\vskip .2cm
Throughout the paper, we will work with {\it smooth schemes of finite type over a given field $k$, which we refer to 
 as the base field}.
 The following is {\it the basic framework} we adopt throughout the paper:
 \vskip .1cm
 {\bf Basic framework adopted throughout the paper.}
 \label{basic.framework}
 \vskip .1cm \noindent
 {\it Basic assumptions on the base field}. 
\begin{enumerate}
\label{basic.assumpts.field}
 \item A {\it standing assumption throughout} is that the base field $k$ is a {perfect field of arbitrary characteristic.} \index{perfect field}
 \item When considering actions by linear algebraic groups $\rmG$
  that are {\it not special}, we {\it will also assume the base field is  infinite} to prevent certain
unpleasant situations. 
\item On considering \'etale realizations of the transfer, it is important to assume
 that the base field \index{finite cohomological dimension}
 \be \begin{multline}
   \label{etale.finiteness.hyps}
    \begin{split}
 k \mbox{ has finite } \ell-\mbox{cohomological dimension,  for } \ell \ne char(k) \mbox{ and satisfies the finiteness conditions  that}\\
 	 \rmH^n_{et}(Spec \, {\it k}, {\mathbb Z}/\ell^{\nu}) \mbox{ is finitely generated in each degree } n \mbox{ and vanishes for all } n>> {\rm 0}, \mbox { all } \nu > 0. \\
 	 \mbox{(Such an assumption is not needed on dealing with the motivic transfer alone.)} 
   \end{split}
 \end{multline} \ee
 \vskip .1cm \noindent
One should be able to see that such an assumption is necessary to get any theory of Spanier-Whitehead duality on the \'etale site of $\Speck$.
\end{enumerate}
\vskip .2cm
{\it Basic assumptions on the linear algebraic groups} considered: 
\begin{enumerate}
\label{basic.assmpts.group}
 \item 
we allow any linear algebraic group over $k$, {\it irrespective} of whether 
 it is {\it connected or not} and 
\item
 we are { not assuming it is special in the sense of Grothendieck (see \cite{Ch}). This means, in particular,  we allow groups such
 as all orthogonal groups and finite groups, which are all known to be non-special.}
\end{enumerate}

\vskip .2cm
\begin{definition} 
\label{Zl.local}
Let $\rmM \in \SH(\k)$ ($\SH(\k_{et}$). For each prime number $\ell$, let ${\mathbb Z}_{(\ell)}$  denote the localization of the integers at the prime ideal $\ell$ and let ${\mathbb Z}\compl_{\ell} =\invlimn {\mathbb Z}/\ell^n$.
Then we say $\rmM$ is ${\mathbb Z}_{(\ell)}$-local ($\ell$-complete, $\ell$-primary torsion), if each $[{\rmS^1}^{\wedge s} \wedge \T^t \wedge {\Sigma^{\infty}_{\T}}\rmU_+, \rmM]$ is a ${\mathbb Z}_{(\ell)}$-module (${\mathbb Z}\compl_{\ell}$-module, ${\mathbb Z}\compl_{\ell}$-module which is torsion, \res) as $\rmU$ varies among the objects of the given site, and where 
$[{\rmS^1}^{\wedge s}\wedge \T^t \wedge {\Sigma^{\infty}_{\T}}\rmU_+, \rmM]$
 denotes $Hom$ in the stable homotopy category $\SH(\k)$ ($\SH(\k_{et})$, \res). \index{$\ell$-complete spectra} \index{${\mathbb Z}_{(\ell)}$-local spectra}
\end{definition}
Let $\rmM \in \SH(\k)$ ($\SH(\k_{et})$). Then one may observe that if $\ell$ is a prime number, and $\rmM$ is $\ell$-complete, then $\rmM$ is ${\mathbb Z}_{(\ell)}$-local. This follows readily by observing that the natural map ${\mathbb Z} \ra {\mathbb Z}\compl_{\ell}$ factors through ${\mathbb Z}_{(\ell)}$ since
 every prime different from $\ell$ is inverted in ${\mathbb Z}\compl_{\ell}$. One may also observe that if $\cE$ is a commutative ring spectrum which
is ${\mathbb Z}_{(\ell)}$-local ($\ell$-complete), then any module spectrum $\rmM$ over $\cE$ is also ${\mathbb Z}_{(\ell)}$-local 
($\ell$-complete, \res). $\ell$-completion in the motivic framework is discussed in some detail in \cite[Appendix]{CJ23-T2}.
\vskip .2cm
We summarize the results on dualizability in the following Theorem.
\begin{theorem}(Dualizability)
\label{duality}
 \begin{enumerate}[\rm(i)]
  \item Over a field $\k$ of characteristic $0$, the $\T$-suspension spectrum of any smooth scheme of finite type over $\k$ is dualizable
  in $\SH(\k_{\rm mot})$.
  \item Assume the base field $k$ is of positive characteristic $p$. Let $\cE$ denote a motivic ring spectrum which is ${\mathbb Z}_{(\ell)}$-local as in
  Definition  ~\ref{Zl.local}, for a prime $\ell$ different from $p$. Then $\cE\wedge X_+$ is dualizable in $\SH(\k_{\rm mot}, \cE)$, for any smooth 
  scheme $\rmX$ of finite type over $\k$.
  \item Assume the base field $\k$ satisfies the finiteness conditions in ~\eqref{etale.finiteness.hyps}. Let $\ell$ be a prime different
   from $char(\k)$ and let $\cE$ denote
  a motivic ring spectrum whose homotopy groups are all $\ell$-primary torsion. Then $ \epsilon^*(\cE \wedge \rmX_+)$ is dualizable in $\SH(\k_{et}, \epsilon^*(\cE))$
  for any smooth scheme $\rmX$ of finite type over $\k$.
 \end{enumerate}
\end{theorem}
\vskip .2cm
The following is an overview of our main results on the construction of the transfer. We start with an equivariant form of the pre-transfer which is
defined using a theory of Spanier-Whitehead duality, which can be made to be compatible with a group action as discussed below. 
The next step in the construction of the transfer is to feed an equivariant form of the pre-transfer, such as in ~\eqref{G.equiv.pretransfer} 
into a Borel construction, which is discussed in section ~\ref{Borel.construct}.
Let ${\group}$ denote a linear algebraic group. We need to carry out the construction of the transfer in two distinct contexts:
(i) when the group ${\group}$ is {\it special} in Grothendieck's terminology: see \cite{Ch}. For example, ${\group}$ could be a $\GL_n$ for some $n$ or a finite product of $\GL_n$s \mbox{ and } (ii) when ${\group}$ is not
necessarily special. In the first case, every $\rmG$-torsor is locally trivial on the Zariski (and hence the Nisnevich) topology 
while in the second case $\rmG$-torsors are locally trivial only in the \'etale topology.
\vskip .2cm
In both cases, we will let ${\BG}^{\it gm,m}$ (${\EG}^{\it gm,m}$) denote the $m$-th degree approximation to the classifying 
space of the group $\group$ (its principal $\group$-bundle, \res) as in  \cite{MV} (see also \cite{Tot}).
These are, in general, quasi-projective smooth schemes over $k$. 
\vskip .2cm
More generally,  we may start with a $\rmG$-torsor $\rmE \ra \rmB$, with both $\rmE$ and $\rmB$ smooth quasi-projective schemes over $\rmS$. 
We will further assume that $\rmB$ is {\it always connected}. A basic example of such a torsor is ${\EG}^{\it gm,m} \ra {\BG}^{\it gm,m}$.
\footnote{By default, whenever a group acts on a scheme or a simplicial presheaf, we will view it as a left action.}
\vskip .1cm
{\it The Borel construction}. 
Given a simplicial presheaf $\rmX$ with an action by $\rmG$, one forms the quotient $\rmE\times _{\rmG} \rmX$
of the product $\rmE \times \rmX$ by the diagonal action: here $\rmG$ acts on the right on $\rmE$ through the involution of $\rmG$ given by $g \mapsto g^{-1}$ and
on the left on $\rmX$ in the usual manner. The construction of such a quotient is the Borel construction and
it needs to be carried out carefully so that if $\rmX$ is a smooth scheme, one obtains the correct object. This construction is discussed
in detail in section ~\ref{Borel.construct}.
\vskip .1cm
Let $\rmX$ and $\rmY$ denote two simplicial presheaves provided with $\rmG$-actions. We will consider the following three {\it basic contexts} for the transfer: \index{torsor} \index{The three basic contexts}
 \vskip .1cm
 (a) $\rmp: \rmE \ra \rmB$ is a $\rmG$-torsor for the action of a linear algebraic group $\rmG$ with both $\rmE$ and $\rmB$ smooth quasi-projective schemes over $k$, with $\rmB$ {\it connected} and 
 \[\pi_{\rmY}: \rmE \times_{\rmG}  (\rmY \times \rmX) \ra \rmE \times_{\rmG}  \rmY\]
 the induced map, where
 $\rmG$ acts diagonally on $\rmY \times \rmX$. One may observe that, on taking $\rmY= Spec \, \k$ with the trivial action of $\rmG$, the map $\pi_{\rmY}$ becomes $\pi:\rmE\times_{\rmG} \rmX \ra \rmB$ (the induced projection), which is an
 important special case.
\vskip .1cm 
(b) ${\BG}^{\it gm,m}$ will denote the $m$-th degree approximation to the geometric classifying space of the linear algebraic group $\rmG$ as in \cite{MV} (see also \cite{Tot}),
$\rmp: {\EG}^{\it gm, m} \ra {\BG}^{\it gm, m}$ is the corresponding universal $\rmG$-torsor and 
\[\pi_{\rmY}: {\EG}^{\it gm, m}\times_{\rmG} (\rmY \times \rmX) \ra {\EG}^{\it gm, m}\times_{\rmG} \rmY\]
is the induced map.
\vskip .1cm
(c) If $ \rmp_m$ ($\pi_{\rmY, m}$) denotes the map denoted $\rmp$ ($\pi_{\rmY}$) in (b), here we let $\rmp = \colimm \rmp_m$ and let 
\[\pi_{\rmY} = \colimm \pi_{\rmY,m}: {\EG}^{\it gm}\times_{\rmG} (\rmY \times \rmX) =\colimm {\EG}^{\it gm, m}\times_{\rmG}(\rmY \times \rmX) \ra \colimm {\EG}^{\it gm, m}\times_{\rmG} \rmY = {\EG}^{\it gm}\times_{\rmG} \rmY.\]
Strictly speaking, the above definitions apply only to the case where $\rmG$ is { special}. 
 When $\rmG$ is { not special}, the above objects will in fact need to be replaced by the derived push-forward of the above objects viewed as sheaves on the
big \'etale site of $k$ to the corresponding big Nisnevich site of $k$, as discussed in ~\eqref{et.case.Borel.const.1}.  However, we will denote
these new objects also by the same notation throughout, except when it is necessary to distinguish between them. Recall that, for $\rmG$ { not special}, we {\it will assume the base field is also infinite} to prevent certain
unpleasant situations. 
\vskip .2cm
Throughout the following discussion, $\cE^{\rmG}$ will denote any one of the $\rmG$-equivariant spectra considered in ~\eqref{choice.ring.spectra}, with $\cE$ denoting the 
 corresponding non-equivariant spectrum: see Definition ~\ref{equiv.vs.nonequiv.spectra}.

\vskip .2cm

Then we obtain the following key result.
\vskip .2cm
\begin{theorem} 
\label{thm.transf.const}
 Let $\rmf: \rmX \ra \rmX$ denote
a ${\group}$-equivariant map and let $\pi_{\rmY}: \rmE\times _{\rm G}(\rmY \times \rmX) \ra \rmE\times _{\rm G}\rmY$ denote any one of the maps considered in {\rm (a)} through {\rm (c)} above. Let $\rmf_{\rmY} = id_{\rmY} \times \rmf:\rmY \times \rmX \ra \rmY \times \rmX$ denote
 the induced map. 
\vskip .2cm
Then in case {\rm (a)}, we obtain a map (called {\it the transfer })
\[\tr(\rmf_{\rmY}):{\Sigma^{\infty}_{\T}}(\rmE\times _{\rm G}\rmY)_+ \ra {\Sigma^{\infty}_{\T}} (\rmE\times _G({\rmY}\times \rmX))_+\quad  (\tr(\rmf_{\rmY}): \cE \wedge (\rmE\times _{\rm G}\rmY)_+ \ra \cE \wedge (\rmE\times _G({\rmY}\times \rmX))_+ )\]
in ${\SH}(k)$ ($\SH(k, {\cE})$, \res) if ${\Sigma^{\infty}_{\T}}\rmX_+$ is dualizable in ${\SH}(k)$ 
(if $ \cE \wedge \rmX_+$ is dualizable in $\SH(k, {\cE})$, \res)
having the following properties.
\begin{enumerate}[\rm(i)]
\item If $\tr(\rmf_{\rmY})^m:{\Sigma^{\infty}_{\T}} (\EG^{\it gm, m}\times_{\rmG}\rmY) _+ \ra {\Sigma^{\infty}_{\T}} (\EG^{\it gm,m}\times_{\rmG}(\rmY \times \rmX))_+\quad  (\tr(\rmf_{\rmY})^m: \cE \wedge (\EG^{\it gm, m}\times_{\rmG}\rmY) _+ \ra \cE \wedge (\EG^{\it gm,m}\times_{\rmG}(\rmY \times \rmX))_+ )$
denotes the corresponding transfer maps in case {\rm (b)}, the maps $\{\tr(\rmf_{\rmY})^m|m\}$ are compatible as $m$ varies. The corresponding induced map $\colimm \, \tr(\rmf_{\rmY})^m$ will be denoted $\tr(\rmf_{\rmY})$.
\item
Assume the base field $k$ satisfies the finiteness conditions in ~\eqref{etale.finiteness.hyps}. Assume
$\cE$ (which belongs to $\Spt(\k_{\rm mot})$) is  $\ell$-complete, in the sense of Definition ~\ref{Zl.local}, 
for some prime $\ell \ne char(\k)$.  Let $\epsilon^*: \Spt(\k_{\rm mot}) \ra \Spt(\k_{et})$
denote the functor induced by the obvious map sites from the \'etale site of $k$ to the Nisnevich site of $k$. 
\vskip .1cm
If $ \epsilon^*(\cE \wedge \rmX_+)$ is dualizable in $\SH(\k_{et}, \epsilon^*(\cE))$,  then there exists a transfer 
$\tr(\rmf_{\rmY})$ in $\SH(\k_{et}, \epsilon^*(\cE))$ satisfying similar properties.
\end{enumerate}
\end{theorem}
\begin{remark} Further properties of the transfer are discussed in \cite[Theorem 2.3]{CJ23-T2}.
\end{remark}
\vskip .2cm
We begin section 2, with a quick review of the basic model category framework for simplicial presheaves both in the motivic and 
\'etale settings. This is followed by a brief discussion of a model structure on the category of pointed simplicial presheaves provided with 
the action of a presheaf of groups.  The next three sections discuss the categories of spectra used in the construction of the
transfer. We let $\rmS$ denote either the base scheme or a fixed simplicial presheaf. 
Section 4 then starts with the category of equivariant spectra, denoted $\Spt^{\rmG}(\rmS)$: such equivariant spectra
will be indexed by the Thom spaces of finite dimensional representations over the given base $\rmS$. $\Spt(\rmS)$ will denote the 
category of spectra that are indexed by the non-negative integers. To relate these two categories of spectra we also introduce
 intermediate categories of spectra denoted ${\widetilde \Spt}^{\rmG}(\rmS)$ and ${\widetilde \Spt}(\rmS)$. The above categories of spectra are considered in both the 
motivic and \'etale contexts. Section 5 discusses various model structures on these categories. In section 6, we relate the 
model structures on the above categories of spectra. 
\vskip .2cm
The following is proven there.
\begin{theorem} (See Proposition ~\ref{comp.1}.)
 \label{thm.comp.spectra}
 \begin{enumerate}
  \item The categories of spectra $\Spt(\rmS)$ and ${\widetilde \Spt}(\rmS)$ are related by adjoint functors 
  $i^*:{\widetilde \Spt}(\rmS) \ra \Spt(\rmS)$ and ${\mathbb P}: {\widetilde \Spt}(\rmS) \ra \Spt(\rmS)$ which define a 
  Quillen equivalence between the corresponding projective stable model structures.
  \item The categories of spectra ${\widetilde \Spt}^{\rmG}(\rmS)$ and ${\widetilde \Spt}(\rmS)$ are related by adjoint functors 
  $j^*:{\widetilde \Spt}(\rmS) \ra {\widetilde \Spt}^{\rmG}(\rmS)$ and $\tilde{\mathbb P}: {\widetilde \Spt}^{\rmG}(\rmS) \ra {\widetilde \Spt}(\rmS)$ which define a 
  Quillen equivalence between the corresponding projective stable model structures.
  \item The functors ${\mathbb P}$ and $\tilde {\mathbb P}$ are strict monoidal functors.
 \end{enumerate}

\end{theorem}
In addition, there is an obvious forgetful functor $\tilde \rmU: \Spt^{\rmG}(\rmS) \ra {\widetilde \Spt}^{\rmG}(\rmS)$.
It is shown in Proposition ~\ref{comp.2}, that if $\X$ belongs to $\Spt^{\rmG}(\rmS)$, then a functorial fibrant or cofibrant
replacement of $\tilde \rmU(\X)$ in fact belongs to $\Spt^{\rmG}(\rmS)$. This observation, then enables us to
show that one can define Spanier-whitehead duals of spectra $\X$ in $\Spt^{\rmG}(\rmS)$ so that $\tilde \rmU(\X)$ 
are dualizable as objects in ${\widetilde \Spt}^{\rmG}(\rmS)$, and that then the Spanier-Whitehead duals in fact belong to $\Spt^{\rmG}(\rmS)$. This is similar to 
the discussion in the first two paragraphs, on the 
Spanier-Whitehead dual in the topological setting of the suspension spectra of compact manifolds provided  with group-actions and plays a key role in the
construction of the transfer in section 8. Section 7 discusses Spanier-Whitehead duality in the motivic and \'etale settings and an appendix
summarizes basic results on spherical fibrations and Thom spaces in the motivic and \'etale settings.

\label{intro}
\section{\bf Basic Model category framework for simplicial presheaves}
\label{basic.framework.2}
\vskip .2cm
We will fix a {\it perfect} field $k$ as the base scheme, and then restrict to the category of smooth
schemes of finite type over $k$.  This category will be denoted ${\rm Sm}_{\k}$. 
 This category will be provided
with either the big Zariski, big Nisnevich or big \'etale topologies and the corresponding site will be denoted ${\rm Sm}_{\k, \Zar}$, ${\rm Sm}_{\k, \Nis}$ or
${\rm Sm}_{\k, \et}$. (Observe as a result, that the objects of these categories are all smooth schemes of finite type over $\k$ and hence
have $\Speck$ as the terminal object, and the coverings of a given scheme will be either Zariski, Nisnevich or \'etale coverings.) If $\k=\Cl$ is the field of complex numbers, one also considers the local homeomorphism
 topology. Here  the coverings of an object $\rmU$ are collections $\{\rmU_i \ra \rmU(\Cl)|i\}$, with each $\rmU_i \ra \rmU(\Cl)$ 
a quasi-finite map of
topological spaces which are local-homeomorphisms when $\rmU(\Cl)$ is provided with the transcendental topology.
 ${\rm Sm}_{\k,\ellh}$ will denote the corresponding big site, where the objects are again smooth schemes over $k$.
\vskip .2cm
The goal of this section is to establish a general framework for the rest of our work: though much of our work takes place in the 
motivic setting on the Nisnevich site, the \'etale and Betti realization functors make it essential that we state our results in 
this section so that they hold on any of the above sites. Results of a technical nature on the various model categories considered 
in this section will be discussed separately in later sections.
Given the above choices for the categories of schemes, the following will be the main choice for a category of simplicial presheaves on it.
\subsection{Simplicial presheaves on ${\rm Sm}_{\k}$}
\label{model.simpl.presh}
The category of all unpointed
simplicial presheaves on ${\rm Sm}_{\k}$ will be denoted $\Spc(\k)$, while the corresponding category of all pointed 
simplicial presheaves on this category will be denoted $\Spc_*(\k)$. Observe that the latter category is a symmetric monoidal
category with the usual smash product of pointed simplicial presheaves as the product: this will be denoted $\wedge$. \index{$\Spc(\k)$} \index{$\Spc_*(\k)$}
\vskip .1cm
Next one has several possible choices of {\it model structures} on the  categories of simplicial presheaves on $\Spc(\k)$ and 
$\Spc_*(\k)$. For example, one has the {\it projective} model structure, where
the fibrations and weak-equivalences are defined section-wise, with the cofibrations defined by  the lifting property. One also
has the {\it injective} model structure (which is also often called the object-wise model structure), where weak-equivalences and cofibrations are defined section-wise, with the fibrations defined
by the lifting property. One of the main advantages of the injective model structure is that every object is cofibrant and every
injective map of simplicial presheaves is a cofibration. \index{injective model structure on simplicial presheaves} \index{object-wise model structure on simplicial presheaves}
All the model structures considered above are cofibrantly generated, and in fact combinatorial model categories: see, for example, \cite[Proposition A.2.8.2]{Lur}.
These are also {\it tractable} model structures, in the sense that the sources of the generating cofibrations and trivial cofibrations are 
cofibrant. The projective model structure is also {\it cellular} (see \cite[Definition 12.1.1]{Hirsch}) and both model structures are left proper (see \cite[Definition 13.1.1]{Hirsch}.
\vskip .1cm
Next we let $\rmS$ denote a fixed simplicial presheaf in $\Spc(\k)$, which could be either the presheaf represented by an object
of the site ${\rm Sm}_{\k}$ or any simplicial presheaf in $\Spc(\k)$. Then we let $(\Spc(\k) \downarrow \rmS)$ denote the 
category of objects over $\rmS$ in $\Spc(\k)$: an object in this category is an object $\rmP \in \Spc(\k)$ together with a map
$\rmp_{\rmP}: \rmP \ra \rmS$ and where a map from  $(\rmP, \rmp_{\rmP})$ to $(\rmQ, \rmq_{\rmQ})$ is a map $\rmf: \rmP \ra \rmQ$ 
so that $\rmp_{\\rmQ} \circ \rmf = \rmp_{\rmP}$. Clearly there is a forgetful functor $\rmU: (\Spc(\k)\downarrow \rmS) \ra \Spc(\k)$.
It is shown in \cite[Theorem 1.5]{Hirsch15} that this model category is also a cofibrantly generated
model category where a map $\rmf: (\rmP, \rmp_{\rmP}) \ra (\rmQ, \rmp_{\rmQ})$ is a cofibration (fibration, weak-equivalence, generating cofibration, generating trivial cofibration) 
if and only if $\rmU(f)$ is a cofibration (fibration, weak-equivalence, generating cofibration, generating trivial cofibration, \res) in $\Spc(\k)$. 
Moreover, it is shown in \cite[Theorem 1.7]{Hirsch15} that the model category $(\Spc(\k)\downarrow \rmS)$ is cellular and 
left proper, when the given model structure on $\Spc(\k)$ is cellular and left proper.
\vskip .2cm
\subsubsection{\bf Pointed simplicial presheaves}
\label{simpl.presheaf.0}
 \index{simplicial presheaf} 
 \vskip .1cm
(i) The main choice for the category of simplicial presheaves, (which will be important in considering fiber-wise duality), will be the following. 
In general $\rmS$ will denote a simplicial presheaf on ${\rm Sm}_{\k}$ as in the last paragraph.
We will restrict to the category of simplicial presheaves that are pointed over $\rmS$, i.e., it is the category 
 consisting of pairs $(\rmP, \rmP_{\rmP})$ in $(\Spc(\k)\downarrow \rmS)$ together with a section $\rms_{\rmP}$ to $\rmp_{\rmP}$.
A map $(\rmP, \rmp_{\rmP}, \rms_{\rmP}) \ra (\rmQ, \rmp_{\rmQ}, \rms_{\rmQ})$ will denote a map $\rmf: \rmP \ra \rmQ$ so
that $\rmp_{\rmQ} \circ \rmf = \rmp_{\rmP}$ and $\rmf \circ \rms_{\rmP} = \rms_{\rmQ}$. {\it This category will henceforth be denoted $\Spc_*(\rmS)$.} 
Therefore, $\rms_{\rmP}$ sends $\rmS$ isomorphically to
a sub-object of $\rmP$, which we denote by $\rms_{\rmP}(\rmS)$. We let the forgetful functor sending a triple $(\rmP, \rmp_{\rmP}, \rms_{\rmP})$ to
$\rmP$ be denoted by $\rmU$ again.
\vskip .1cm
(An example of the case where $\rmS$ is a simplicial presheaf
appears in ~\eqref{et.case.Borel.const.1}. (See also ~\eqref{compat.sections.V}.) In fact that is the reason for working in this generality.) 
\vskip .1cm
It is shown in \cite[Lemma 2.1.21]{Hov99} that one may define the structure of a cofibrantly generated model category 
on $\Spc_*(\rmS)$ by defining a map $\rmf: (\rmP, \rmp_{\rmP}, \rms_{\rmP}) \ra (\rmQ, \rmp_{\rmQ}, \rms_{\rmQ})$ to be a cofibration (fibration,
 weak-equivalence) if $\rmU(f)$ is a cofibration (fibration, weak-equivalence, \res) in $\Spc(\k)$. Moreover the generating cofibrations
 (generating trivial cofibrations) 
 for this model structure is given by 
 \be \begin{align}
  \label{IJ+}
\rmI_{\rmS} &=\{i\sqcup \rmS : \rmA \sqcup \rmS \ra \rmB \sqcup \rmS| i:\rmA \ra \rmB  \in \rmI, \rmA, \rmB, i \in (\Spc(\k)\downarrow \rmS)\}, \\
 (\rmJ_{\rmS} &= \{j\sqcup \rmS: \rmC\sqcup \rmS \ra \rmD \sqcup \rmS| j: \rmC \ra \rmD \in \rmJ, \rmC, \rmD, j \in (\Spc(\k)\downarrow \rmS)\}, \res) \notag
 \end{align} \ee
 if $\rmI$ ($\rmJ$) denote the set of generating cofibrations (generating trivial cofibrations) for
  the model structure on $\Spc(\k)$. It follows from \cite[Theorems 2.7 and 2.8]{Hirsch15} that the resulting model structure 
  on $\Spc_*(\rmS)$ is cellular and left proper when the model structure on $\Spc(\k)$ that one starts with is. For an object $\rmA \in
 (\Spc(\k)\downarrow \rmS)$, we will henceforth refer to $\rmA \sqcup S$ as {\it the object $\rmA$ pointed by $\rmS$} and {\it denote it by} $\rmA_+$ for convenience.  \index{$\Spc_*(\rmS)$}
\vskip .2cm
We next define a monoidal structure on
$\Spc_*(\rmS)$ as follows. Let $\rmP, \rmQ \in \Spc_*(\rmS)$. Then we let:
\be \begin{equation}
     \label{wedgeS}
\rmP \wedge^{\rmS} \rmQ = (\rmP\times_{\rmS} \rmQ)/(\rms_{\rmP}(\rmS) \times_{\rmS} \rmQ \cup \rmP \times_{\rmS} \rms_{\rmQ}(\rmS)).
    \end{equation} \ee
\vskip .2cm \noindent \index{$\wedge^{\rmS}$}
It may be important to point out that the term on the right is the quotient over $\rmS$, i.e. {\it the pushout} of:
$\rmS \leftarrow \rms_{\rmP}(\rmS) \times_{\rmS} \rmQ \cup \rmP \times_{\rmS} \rms_Q(S) \ra P \times _{\rmS}\rmQ$.
We skip the verification that $\Spc_*(\rmS)$ with above smash product $\wedge ^{\rmS}$ is a closed symmetric monoidal 
category. 
\vskip .1cm
If the base object $\rmS$ represents a {\it point} in the site, for example, is the spectrum of a field
for the Zariski and Nisnevich sites and is the spectrum of a separably closed field for the \'etale site, then
every simplicial presheaf has an obvious map to $\rmS$, so that the above monoidal structure  reduces to the familiar one.  The main difference
between the two cases is therefore, when $\rmS$ is a general scheme or a chosen simplicial presheaf. In this case, the smash product $\wedge^{\rmS}$ defines
what corresponds to {\it a fiber-wise} smash product over $\rmS$. \index{fiber-wise smash product}  The discussion of the transfer map in  section 8 (see 
~\eqref{Ch4.Thom.sp.1} through  ~\eqref{tr(f)G.3}) and Appendix A, Lemma ~\ref{AppA.fiberwise.join.2}
show that indeed the fiber-wise smash product is important for us.
\begin{terminology}
 \label{termin.0}
 It is convenient for us to work with a general simplicial presheaf $\rmS$ as the base for a considerable part of our discussion, in this introductory section. As a result we will let $\rmS$ denote
 such a general simplicial presheaf for the most part in this section, except in those special cases where we need to require this to be the base field $\k$.
\end{terminology}

\vskip .2cm
{\bf Further refinements of the above model structures.}
We need to modify these model structures, so that
 the resulting model structures satisfy { the following basic requirements}:
 \vskip .1cm
\label{model.struct.req.0}
\begin{enumerate}[\rm(i)]
 \item the pushout-product axiom and the monoidal axiom with respect to
the above monoidal structures hold.
\item Since one of the main focus is on motivic applications, we will always invert all maps of the form 
\[\{{\rm {pr}}:\rmX \times {\mathbb A}^1 \ra \rmX|\rmX\},\]
where $\rmX$ varies over all the objects in the given site.  We will perform this localization even when considering the \'etale sites,
since ${\mathbb A}^1$ is acyclic in the \'etale topology only with respect to constant sheaves like $\Zl^{\nu}$,
where $\ell $ is different from the residue characteristics.
\end{enumerate}
\subsubsection{\bf The motivic model structure on Nisnevich presheaves}
\label{Nis.presh}
We will accomplish this on the Nisnevich site as follows. One defines a presheaf $\rmP \in \Spc_*(\rmS)$
to be {\it motivically fibrant} if (i) $\rmP \in \Spc_*(\rmS)$ is fibrant, (ii) $\rmP$ sends an {\it elementary 
distinguished square} as in \cite[p. 96, Definition 1.3]{MV} whose component schemes when pointed by $\rmS$ belong to $\Spc_*(\rmS)$ to a homotopy cartesian square and (iii) the obvious pull-back 
$\Gamma (\rmU, \rmP) \ra \Gamma (\rmU \times {\mathbb A}^1, \rmP)$ is a weak-equivalence, for all $\rmU$ in the site ${\rm Sm}_{\k}$ when pointed by $\rmS$  belong to 
$\Spc_*(\rmS)$. Then a map $\rmf: \rmP \ra \rmQ$ 
 in  $\Spc_*(\rmS)$ is a motivic weak-equivalence if the induced map $Map(\rmf, \rmP)$ is
a weak-equivalence for every motivically fibrant object $\rmP$, with $Map$ denoting the simplicial
mapping space. One then localizes such weak-equivalences.  The resulting model structure will be denoted 
 $\Spc_*(\rmS_{\rm mot})$. \index{$\Spc_*(\rmS_{\rm mot})$}
\vskip .2cm
One may specify the generating trivial cofibrations for the localized model category considered above as in \cite[section 2]{DRO2}.

\vskip .3cm
{\it An alternate approach} that applies in general to any of the sites we consider is to localize by inverting hypercovers as in \cite{DHI}.
For the convenience of the reader, we will discuss a little bit of the background here. First one needs to fix 
a Grothendieck topology on ${\rm Sm}_{\k}$, which we will assume is one of the topologies considered before, namely,
the Zariski, the Nisnevich or the \'etale topologies. We will denote this site by ${\rm Sm}_{\k?}$. Next one defines the notion of
local fibrations for simplicial presheaves in $\Spc_*(\rmS)$, which in particular, implies that the induced map on the stalks are all 
fibrations: see \cite[section 3]{DHI}. A map of simplicial presheaves in $\Spc_*(\rmS)$ is a {\it local weak-equivalence} if it induces
a weak-equivalence at all the stalks. A map of simplicial presheaves in $\Spc_*(\rmS)$ is a {\it local acyclic fibration} if 
it is a local fibration and a local weak-equivalence. Let $\rmU \eps {\rm Sm}_{\k}$ so that $\rmU \sqcup \rmS$ belongs to $\Spc_*(\rmS)$. 
A {\it hypercover} $\rmU_{\bullet}\sqcup \rmS \ra \rmU \sqcup \rmS$ is a simplicial presheaf in $\Spc_*(\rmS)$ so that (i) each $\rmU_n$ is a co-product
of representable objects from the site ${\rm Sm}_{k?}$ with each $\rmU_n \sqcup \rmS$ belonging to $\Spc_*(\rmS)$, (ii) comes equipped with a map $\rmU_{\bullet}\sqcup \rmS \ra \rmU \sqcup \rmS$
in $\Spc_*(\rmS)$ that is a local acyclic fibration.
\vskip .2cm
Next it is shown in \cite[Proposition 6.4]{DHI} that the class of all hypercovers contains {\it a dense set} of hypercovers
so that any hypercover may be refined by one belonging to the above set. 
Following \cite{DHI}, a simplicial presheaf $\rmP$ has the {\it descent property}
for all hypercovers if for all $\rmU$ in ${\rm {Sm}}_{\k?}$, and all hypercovers $\rmU_{\bullet} \ra \rmU$,
the induced map $\rmP(\rmU \sqcup \rmS) \ra \holimD \{\Gamma (\rmU_n \sqcup \rmS, \rmP)|n\}$ is a weak-equivalence.
\vskip .1cm
By localizing with respect to
maps of the form $\rmU_{\bullet}\sqcup \rmS \ra \rmU \sqcup \rmS$ where $\rmU_{\bullet}$ belongs to a dense set of hypercovers of 
 of $\rmU$ and also maps of the form $\rmU \times {\mathbb A}^1 \sqcup \rmS \ra \rmU \sqcup \rmS$, it is proven in
 \cite[Theorem 8.1]{Dug} and \cite[Example A. 10]{DHI} that we obtain
a model category which is Quillen equivalent to the model category  $\Spc_*(\rmS_{\rm mot})$ considered above.
In this case one defines a generating set of trivial cofibrations as follows, where $\rmh_X$ denotes the simplicial presheaf associated to 
an object $\rmX$ in the site. Let 
\vskip .1cm
\be \begin{multline} 
     \label{tilde.Jet}
      \begin{split}
 {\oJ}' = \{ u: \rmh_{\rmU \times {\mathbb A}^1}\sqcup \rmS \ra  \rmC_u \sqcup \rmS \mid \rmU \eps {\rm {Sm}}_{\k} \} \cup \{\sqcup_{\alpha}h_{\rmU_{\alpha}}\sqcup \rmS \ra \rmh_{\rmU} \sqcup \rmS \in \Sm_{\k}| \rmU \mbox{ is the disjoint union of } \rmU_{\alpha}\}\\
\cup \{ \rmh_{\rmU_{\bullet}} \sqcup \rmS \ra Cyl(h_{\rmU_{\bullet}} \ra \rmh_{\rmU}) \sqcup \rmS \mid \rmU \eps {\rm {Sm}}_{\k} \mbox{ and }\rmU_{\bullet} \ra \rmU  \mbox{ is a given hypercover } \}
  \end{split} 
\end{multline} \ee
\noindent 
where  we factor
the obvious map  $\rmh_{\rmU \times {\mathbb A}^1} \ra \rmh_{\rmU}$ into a cofibration $u:\rmh_{\rmU \times {\mathbb A}^1} \ra 
\rmC_u$ followed by a simplicial homotopy equivalence $\rmC_u \ra \rmh_{\rmU}$. (See \cite[section 2]{Isak}.)

\vskip .2cm
 Though the resulting localized category
is cellular and left-proper (see \cite[Chapters 12 and 13]{Hirsch}) it is unlikely to be weakly finitely generated:
the main issue  is that the hypercoverings, being simplicial objects, need not be small. 
\vskip .2cm
\subsubsection{The injective model structures}
\label{inj.model}
In case one starts with the {\it injective model structure} on $\Spc_*(\rmS)$, one needs to modify the above
set-up as follows. First the generating cofibrations for the injective model structure on $\Spc(\k)$ will be some set of injective maps $\{\rmA_{\alpha} \ra \rmB_{\alpha}|\alpha\}$
and not the set $\{\delta \Delta[n]\times \rmh_{\rmU} \ra \Delta[n] \times \rmh_{\rmU} \in \Spc(\k) |\rm n \ge  {\rm 0}\}$. The resulting model structure will be a combinatorial model structure. Next in order 
to consider the left Bousfield localization for the motivic model structure, one needs to modify the set $\oJ$ of the generating trivial cofibrations as follows:
it will be the pushout product of maps of the form $\rmf \square \rmg$, where $\rmf$ denotes a map in the set ${\oJ}'$, and $\rmg: A_{\alpha} \ra B_{\alpha}$ considered above.
\vskip .1cm
Then the monoidal axiom and the pushout-product axiom may be verified readily as every object in this model structure is cofibrant. As a result it also provides
the structure of a monoidal model structure.
\vskip .2cm
\subsubsection{\bf The motivic model structure on \'etale simplicial presheaves: $\Spc_*(\rmS_{et})$}
\label{et.presh}
Localizing with respect to
maps of the form $\rmU_{\bullet}\sqcup \rmS \ra \rmU \sqcup \rmS$ where $\rmU_{\bullet} \sqcup$ belongs to a dense set of hypercovers 
 of $\rmU \sqcup \rmS$ and also maps of the form $\rmU \times {\mathbb A}^1 \sqcup \rmS\ra \rmU \sqcup \rmS$, as $\rmU$ varies among the objects of the site 
 is the only possibility in this case.
 We discuss the model structure on $\Spc(\k_{et})$ which may be then modified as in ~\eqref{IJ+} to define
 a corresponding model structure on  $\Spc_*(\rmS_{et})$. 
In this case one defines a generating set of trivial cofibrations ${\oJ}'$ just as in ~\eqref{tilde.Jet} by  considering \'etale hypercovers in the 
place of Nisnevich hypercovers.
\vskip .1cm \noindent
Let $\oJ$ be the pushout product of maps of the form $\rmf \square \rmg$, where $\rmf$ denotes a map in the set ${\oJ}'$, and $\rmg: \delta \Delta[n] \ra \Delta[n]$
denotes the obvious cofibration of simplicial sets. The set $\oJ$ will then denote 
a {\it generating set} for the {\it trivial cofibrations} in the above motivic model structure. 

\vskip .2cm
Then we obtain the following result:
\begin{theorem} \index{$\Spc_*(\rmS_{et})$}
 \label{et.model.struct}
 On starting with the projective model structure on $\Spc_*(\rmS)$,  the resulting model structure on $\Spc_*(\rmS_{et})$ is cofibrantly generated, cellular and left proper. 
 On starting with the injective model structure on $\Spc_*(\rmS)$, the resulting model structure on $\Spc_*(\rmS_{et})$ is combinatorial and left proper. Both  also
 satisfy the pushout-product axiom and the monoidal axiom so that they are monoidal model categories with respect to the smash product in ~\eqref{wedgeS}.
\end{theorem}
\begin{proof} The proof follows along the lines of the proof of the corresponding result for the Nisnevich simplicial presheaves
 as discussed in \cite[section 2]{DRO2}. The main difference is in the fact that we will be  considering the set 
 $\{ \rmh_{\rmU_{\bullet}} \sqcup \rmS \ra Cyl(\rmh_{\rmU_{\bullet}} \ra \rmh_{\rmU})\sqcup \rmS \mid \rmU \eps 
{\rm {Sm}}_{\k} \mbox{ that belongs to } \Spc_*(\rmS) \}$ in the place of the set $\{\rmq: {\rm sq} \ra  {\rm tq} \mid q \mbox { is an elementary distinguished square in } \Spc_*(\rmS) \}$.
However, one may observe that if $\phi:\rmU_{\bullet} \ra \rmU$ is a hypercover and $\rmV$ is any object in the site, then 
$\phi\times \rmV: \rmU_{\bullet} \times \rmV \ra \rmU \times \rmV$ is also a hypercover and that moreover, the mapping cylinder
${\rm Cyl}(\rmh_{\rmU_{\bullet} \times \rmV} \ra \rmh_{\rmU \times \rmV}) = Cyl(\rmh_{\rmU_{\bullet}} \ra \rmh_{\rmU}) \times \rmV$. With this modification,
the same arguments as in \cite[section 2]{DRO2} apply to complete the proof.
\end{proof}
\vskip .2cm
\begin{terminology}
 \label{generic.notation} In order to carry out our discussions in as much generality as possible, we will adopt the following convention.
 Throughout the rest of the volume, we will let $\Spc_*(\rmS)$ denote one of the following, unless further clarified: (i) $\Spc_*(\rmS_{\rm mot})$,
 (ii) $\Spc_*(\rmS_{et})$ or (iii) the category of pointed simplicial presheaves on $\Sm_{\k}$ pointed over a fixed simplicial presheaf $\rmS$ as in 
 ~\ref{simpl.presheaf.0}, i.e., without further Bousfield localizations. Moreover, when $\rmS = Spec \, \k$, $\Spc_*(\rmS_{\rm mot})$ ($\Spc_*(\rmS_{et})$)
 will be denoted $\Spc_*(\k_{\rm mot})$ or $\Spc_*(\k)$ ($\Spc_*(\k_{et})$, \res). \index{$\Spc_*(\k_{\rm mot})$} \index{$\Spc_*(\k_{et})$}
\end{terminology}

\subsubsection{\bf Pointed equivariant simplicial presheaves}
We start with a linear algebraic group $\group$ defined over $\k$ (or equivalently a smooth affine group scheme of finite type over $\k$). We will next do a base-extension to $\rmS$,
i.e., replace $\group$ by $\group_{\rmS} = \group \times_{\Speck} \rmS$. But we will continue to denote $\group_{\rmS}$ by $\group$. This way, we may assume,
without loss of generality that $\group \in \Spc_*(\rmS)$. {\it In the following discussion, we will view $\group$ as the corresponding presheaf of
groups on the given site}. 
\begin{definition}
\label{equiv.prshvs}
\begin{enumerate}[\rm(i)]
\item Then 
\be \begin{equation} \index{$\Spc_*^{\rmG}(\rmS)$}
     \label{equiv.framework.0}
\Spc_*^{\rmG}(\rmS) 
    \end{equation} \ee
\vskip .1cm \noindent
will denote the category of 
those presheaves $\rmP \in \Spc_*(\rmS)$ 
provided with an action by the presheaf represented by $\rmG$, with the morphisms being $\rmG$-equivariant maps.  (In particular, this means 
the group action preserves the base point of any pointed simplicial presheaf $\rmP \in \Spc_*^{\rmG}(\rmS)$.)
\item $\Spc_*^{\rmG}(\rmS_{\rm mot})$ ($\Spc_*^{\rmG}(\rmS_{et})$) will denote the corresponding
category of $\rmG$-equivariant presheaves associated to presheaves in $\Spc_*(\rmS_{\rm mot})$ ($\Spc_*(\rmS_{et})$, \res).
\end{enumerate}
\end{definition}
 Here it is important that the base scheme $\rmS$ has trivial action by the 
group $\group$ so that the maps $\rms:S\ra \rmP$ and $\rmp:\rmP \ra \rmS$ are $\group$-equivariant. 
Then maps between such $\group$-equivariant simplicial presheaves will be ${\group}$-equivariant
maps of simplicial presheaves, compatible with the structure maps $\rms$ and $\rmp$. \index{equivariant simplicial presheaf}
\vskip .2cm 
Let $\rmU: \Spc_*^{\rmG}(\rmS) \ra \Spc_*(\rmS)$ 
denote the forgetful functor
forgetting the group action. 
Observe that if $\rmP, \rmQ \in \Spc_*^{\rmG}(\rmS) $, then $\rmP \wedge ^{\rmS}\rmQ$ defined above (i.e., with $\rmP$ and $\rmQ$ viewed as objects in $\Spc_*(\rmS)$) has a natural induced
$\group$-action and therefore, defines an object in $\Spc_*^{\rmG}(\rmS)$. Therefore, we let the monoidal structure on
$\Spc_*^{\rmG}(\rmS)$ be defined by $\wedge ^{\rmS}$ as in ~\eqref{wedgeS}, \res).
Similarly, if $\rmP, \rmQ \in  \Spc_*^{\rmG}(\rmS)$, then the internal $\Hom(\rmP, \rmQ)$ in $\Spc_*(\rmS)$ belongs to
$\Spc_*^{\rmG}(\rmS)$. These basically prove:
\be \begin{equation}
\label{U.1}
\rmU(\rmP \wedge \rmQ) = \rmU(P) \wedge \rmU(Q) \mbox { and } \rmU(\Hom_{\rmG}(\rmP, \rmQ) = \Hom(\rmU(P), \rmU(Q)), \rmP, \rmQ \in  \Spc_*^{\rmG}(\rmS),
\end{equation} \ee
\vskip .1cm \noindent
where $\Hom_G$ denotes the internal hom in $\Spc_*^{\rmG}(\rmS)$.
\vskip .2cm
\begin{proposition}
	\label{functorial.rep.inhert.G.act}
Let $\Spc_*(\rmS)$ be provided with one of the model structures defined above. Let $\rmP \in \Spc_*^{\rmG}(\rmS)$.
\begin{enumerate}[\rm(i)]
 \item 
If $\rmP' \ra \rmU(P)$ is a functorial cofibrant replacement in $\Spc_*(\rmS)$, then $\rmP' \in \Spc_*^{\group}(\rmS)$.
\item
If $ \rmU(P) \ra \rmP''$ is a functorial fibrant replacement in $\Spc_*(\rmS)$, then $\rmP''\in \Spc_*^{\group}(\rmS)$. 
\end{enumerate}
\end{proposition}
\begin{proof} As the proof of (ii) is entirely similar to the proof of (i), we will explicitly consider only the proof of (i).
	Recall ${\group}$ acts on $\rmP$ as a presheaf, i.e., for each scheme $\rmX$ in the given site, ${\group}(\rmX)$ is given
	an action on $\rmP(\rmX)$, compatible with restrictions for maps $\rmU \ra \rmX$ in the site. The functoriality in the choice of the cofibrant replacement $\rmP'$ shows that
	each $g_s \in \rmG(\rmX)$ then has an induced action on $\rmP'(\rmX)$, that the square
	\[ \xymatrix{{\rmP'(\rmX)}\ar@<1ex>[r]^{g_s} \ar@<1ex>[d] & {\rmP'(\rmX)} \ar@<1ex>[d]\\
			      {\rmP(\rmX)} \ar@<1ex>[r]^{g_s} & {\rmP(\rmX)}} \]
commutes, and that the corresponding squares for $\rmU$ and $\rmX$, for a map $\rmU \ra \rmX$ in the given site are compatible. (See \cite[Definition 1.1.1]{Hov99} for details on functorial fibrant and cofibrant replacements.) \end{proof} 

\section{Model structures on equivariant simplicial presheaves}
\index{unstable model structures: equivariant presheaves}
\label{model.eq.psh}
Here we will be making strong use of the model structures on simplicial presheaves in the non-equivariant setting, discussed already in the last section.
One may recall from the discussion in the last section that the following are some of {\it the main choices for the category of simplicial presheaves} we consider. 
 We will let $\rmS$ denote either a scheme in ${\rm Sm}_k$ or a fixed simplicial presheaf and then restrict to the 
 category of simplicial presheaves on ${\rm Sm}_k$ that are pointed over $\rmS$, i.e., those 
 simplicial presheaves $\rmP$ on ${\rm Sm}_{\k}$ that come
 equipped with maps $\rmp: \rmP \ra \rmS$ and $\rms: \rmS \ra P$ so that $\rmp \circ \rms = id_{\rmS}$. 
 \vskip .2cm
 We will follow the {\it generic notation} adopted
 in  Terminology ~\ref{generic.notation} so that $\Spc_{*}(\rmS)$ will denote one of the three categories defined there, i.e.  
 (i) $\Spc_*(\rmS_{\rm mot})$, (ii) $\Spc_*(\rmS_{et})$  or (iii) the category of pointed simplicial presheaves on $\Sm_{\k}$ pointed over the 
 fixed simplicial presheaf $\rmS$ as in 
 ~\ref{simpl.presheaf.0}, i.e., without any further Bousfield localizations.
 \vskip .1cm
Moreover, if $\rmG$ denotes a linear algebraic group of finite type over $\k$,
we will assume that $\rmG$ acts trivially on $\rmS$ and $\Spc_{*}^{\rmG}(\rmS)$ will denote the corresponding category of $\rmG$-equivariant 
simplicial presheaves pointed over $\rmS$. Observe that $\rmG$ itself identifies with the presheaf of groups represented by $\rmG$,
and therefore an action by $\rmG$ on a simplicial presheaf has the obvious meaning. We let $\rmG_{\rmS} = \rmG\times_{\Speck}\rmS$. 
\vskip .2cm
We let $\W$ denote a family of subgroup-schemes of $\group$ so that it has the following properties 
\be  \begin{enumerate}[ \rm (i)]
        \item it is an inverse system ordered by inclusion,
        \item if $\rmH \eps \W$, $\rmH_{\group}$ =the core of $\rmH$, i.e., the largest subgroup of $\rmH$ that is normal in $\group$, belongs to $\W$, and
         \item if $\rmH \eps \W$ and $\rmH' \supseteq \rmH$ is a closed smooth subgroup-scheme of $\group$, then $\rmH' \eps \W$. 
\end{enumerate} \ee
In the case $\group$ is a finite
group, $\W$ will denote all subgroups of $\group$. When $\group$ is a smooth group-scheme, we will leave $\W$ unspecified, for now. Clearly the family of all closed smooth subgroup-schemes
of a given smooth group-scheme satisfies all of the above properties, so that we will use this as a default choice of $\W$ when nothing 
else is specified.
\vskip .3cm
For each  subgroup-scheme $\rmH \eps\W$, let $\rmP^{\rmH}$ denote the sub-presheaf
of $\rmP$ of sections fixed by $\rmH$, i.e., $\Gamma (\rmU, \rmP^{\rmH}) = \Gamma (\rmU, \rmP)^{\rmH}$. If $\rmH$ is a normal subgroup-scheme of $\group$ and 
$\bar \rmH= \group /\rmH$, then
\[\Gamma (\rmU, \rmP)^{\rmH} = \{s \eps \Gamma (\rmU, \rmP)| \mbox { the action of $\group$ on $s$ factors through } \bar \rmH\}.\]
\vskip .3cm
\subsubsection{The $\group$-equivariant sheaves $\group/H_+$, $\rmH \eps \W$}
\label{GmodH}
We let $\group/H \otimes \rmS = \sqcup_{\group/H}S$, i.e., 
\[\Gamma(\rmU, \rmG/\rmH \otimes \rmS) = \sqcup _{\Gamma(\rmU, \rmG/\rmH)} \Gamma(U, \rmS).\]
Then we let $\group/H_+ = ((\group/H) \otimes S) \sqcup S$.
This will be viewed as an object of $\Spc_*^{\rmG}(\rmS)$ where the structure map $\rmp:\group/H_+ \ra S$ sends all the summands $\rmS$ appearing in
$((\group/H) \otimes S) \sqcup S$ to $\rmS$ by the identity map of $\rmS$. The section to $\rmp$ is the map $\rms$ that sends $\rmS$ by the identity to
the outer summand $\rmS$ in $((\group/H) \otimes S) \sqcup \rmS$. We observe that we obtain the adjunction:
\[\Spc_*^{\rmG}(\rmS) \ra \Spc_*(\rmS), Q \mapsto Q^{\rmH} \mbox{ has as left adjoint the functor } \rmP \mapsto (\group/H)_+ \wedge^{\rmS} \rmP.\]

\begin{proposition} 
\label{key.props}
\begin{enumerate}[\rm(i)]
\item Let $\phi: \rmP' \ra \rmP$ denote a map of simplicial presheaves in $\Spc_*^{\rmG}(\rmS)$. Then $\phi$ induces a
 map $\phi^{\rmH} : {\rmP'}^{\rmH} \ra \rmP^{\rmH}$ for each subgroup $\rmH \eps \W$. The association 
$\phi \mapsto \phi^{\rmH}$ is functorial in $\phi$ in the sense that if $\psi: P'' \ra P'$ is another map,
then the composition $(\phi\circ \psi)^{\rmH} = \phi^{\rmH} \circ \psi^{\rmH}$.
\item Let $\{\rmQ_{\alpha}|\alpha\}$ denote a direct system  of simplicial sub-presheaves of a simplicial
presheaf $\rmQ \eps \Spc_{*}^{ \group}(\rmS)$ indexed by a small filtered category. If $\rmK$ is any subgroup of $\group$, then
$({\underset {\alpha} \colim} \, \rmQ_{\alpha})^{\rmK }= {\underset {\alpha} \colim} \, \rmQ_{\alpha}^{\rmK}$.
\end{enumerate}
\end{proposition}
\begin{proof} (i) is clear. 
(ii) Observe  that for a simplicial sub-presheaf $\rmQ'$ of $\rmQ$,
 with 
$\rmQ' \eps \Spc_{*}^{ \group}(\rmS)$, $({\rmQ'})^{\rmK} = {\rmQ'} \cap \rmQ^{\rmK}$. Now each $\rmQ_{\alpha}$ maps injectively into $\rmQ$ and the structure maps
of the direct system $\{\rmQ_{\alpha}|\alpha\}$ are all injective maps. Therefore (ii) follows readily. 
\end{proof}
\vskip .3cm 
 \subsubsection{\it Finitely presented objects} \index{finitely presented}
Recall an object $C$ in a category $\C$ is {\it finitely presented} or {\it compact} if
$Hom_{\C}(C, \quad)$ commutes with all small filtered colimits in the second
argument. Here $Hom_{\C}$ denotes the eternal hom in the category $\C$.
\vskip .2cm
Next we define the following structure of a cofibrantly generated simplicial model
category on $\Spc_*^{\group}(\rmS)$ 
starting with the projective
model structure or the injective model structure on $\Spc_*(\rmS)$: see  ~\ref{model.simpl.presh}. Let $\oI$ ($\oJ$) denote the generating cofibrations (generating trivial cofibrations) in
$\Spc_*(\rmS)$ for the corresponding model structure. Let $\wedge ^{\rmS}$ denote the smash product defined  in section 1, ~\eqref{wedgeS}.
\vskip .3cm 
\begin{definition} (The model structure)
\label{equiv.model.struct}
(i) The {\it generating cofibrations} are of the form 
\[\oI_{\group}=\{(\group/H )_+ \wedge^{\rmS} i \mid   i \eps I, \quad H \eps \W \}, \]
\vskip .3cm 
(ii) the {\it generating trivial cofibrations} are of the form 
\[\oJ_{\group} = \{(\group/H ) _+ \wedge^{\rmS}  j \mid  j \eps J, 
H \subseteq \group, \quad H \eps \W \} \, \mbox{ and }\]
\vskip .3cm 
(iii) and the  {\it weak-equivalences} ({\it fibrations}) are maps  $\rmf:\rmP' \ra \rmP$ in
$\Spc_*^{\group}(\rmS)$ so that $\rmf^{\rmH}:{\rmP'}^{\rmH} \ra \rmP^{\rmH}$ is a weak-equivalence
(fibration, \res) in $\Spc_{*}(\rmS)$ for all $\rmH \eps \W$.
\end{definition}
\begin{theorem} 
\label{model.structs.equiv.presh}
\begin{enumerate}[\rm(i)]
\item The above structure defines a cofibrantly generated
simplicial model structure on
\newline \noindent
$\Spc_*^{\group}(\rmS)$ that is proper. 
\item The above model category is combinatorial and
tractable (in the sense that it is locally presentable, and has sets of generating cofibrations and trivial cofibrations whose sources are also cofibrant.)
\item In addition, the 
smash product of pointed simplicial presheaves (defined as in ~\eqref{wedgeS}) 
makes the above category symmetric monoidal and it satisfies the
pushout-product axiom in both the injective and projective model structures. The unit for the smash product
in  $\Spc_*^{\group}(\rmS)$ is cofibrant in both the injective and projective model structures.
\end{enumerate}
\end{theorem}
\begin{proof} We first consider (i). A key observation is the following. Let $\rmH \subseteq \group$ denote a
subgroup belonging to $\W$. Then  the functor $\rmP \mapsto \rmP^{\rmH},  \, \Spc_*^{\group}(\rmS) \ra \Spc_{*}(\rmS) $ has a left-adjoint, namely, the functor
\be \begin{equation}
     \label{basic.adj}
\rmQ \mapsto (\group/\rmH)_+ \wedge^{\rmS}  \rmQ.
\end{equation} \ee
We now proceed to verify that the hypotheses of 
\cite[Theorem 2.1.19]{Hov99} are satisfied. It is obvious that the subcategory of
weak-equivalences is closed under composition and retracts and has the two-out-of-three
property.  Next we proceed to verify that the domains of $\oI_{\group}$ are small relative to
$\oI_{\group}$-cell. (Given a set $\oI$ of maps in a category containing all small colimits, a relative
$\oI$-cell complex is a transfinite composition of pushouts of elements of $\oI$. We denote this by $\oI$-cell.) Recall first that if 
$\{ \rmP _{\alpha}| \alpha \}$ denotes a small filtered direct system of objects in $\Spc_*^{\group}(\rmS)$
that are sub-objects of  a given $\rmP \eps \Spc_*^{\group}(\rmS)$, then by Proposition ~\ref{key.props}(ii), one obtains the 
identification 
\[ (\colimalpha \rmP_{\alpha})^{\rmH} \cong \colimalpha (\rmP_{\alpha})^{\rmH}. \]
\vskip .3cm
We consider this first in the projective model structure. \index{projective model structure}
One may first recall that the projective model structure on $\Spc_*(\rmS)$ has as generating cofibrations: $\{ (\delta \Delta [n]_+ \wedge \rmh_{\rmX+}) \ra   ( \Delta [n]_+ \wedge \rmh_{\rmX+})|n\}$.
\vskip 1cm
Let $ (\group/\rmH )_+ \wedge  (\delta \Delta [n]_+ \wedge^{\rmS} \rmh_{\rmX+}) \ra (\group/\rmH )_+ \wedge^{\rmS}  ( \Delta [n]_+ \wedge^{\rmS} \rmh_{\rmX+})$
denote a generating cofibration in $\oI_G$. Suppose one is given a map $\rmf:  (\group/\rmH )_+ \wedge^{\rmS} (\delta \Delta [n]_+ \wedge^{\rmS} \rmh_{\rmX+}) \ra 
\colimalpha \rmP_{\alpha}$.
By the adjunction in ~\eqref{basic.adj}, this map corresponds to a map $\rmf^{\rmH}:(\delta \Delta [n]_+ \wedge^{\rmS} \rmh_{\rmX+}) \ra (\colimalpha \rmP_{\alpha})^{\rmH}
= \colimalpha (\rmP_{\alpha})^{\rmH}$.
Since $(\delta \Delta [n]_+ \wedge^{\rmS} \rmh_{\rmX+})$ is small in $\Spc_*(\rmS)$, the map
$(\delta \Delta [n]_+ \wedge^{\rmS} \rmh_{\rmX+}) \ra \colimalpha (\rmP_{\alpha}^{\rmH})$ factors through some $\rmP_{\alpha_0}^{\rmH}$. Therefore, its adjoint
$\rmf:(\group/\rmH )_+ \wedge^{\rmS} (\delta \Delta [n]_+ \wedge^{\rmS} \rmh_{\rmX+}) \ra \colimalpha \rmP_{\alpha}$ factors through $\rmP_{\alpha_0}$.
This proves the domains
of $\oI_{\group}$ are small relative to $\oI_{\group}$-cell. 

An entirely similar argument proves that the domains of $\oJ_{\group}$ are small 
relative to $\oJ_{\group}$-cell.
Observe that these two steps make use of the fact that the model structure on $\Spc_{*}(\rmS)$ is indeed the projective
model structure.
\vskip .3cm
Next we consider the injective model structure. In fact the following arguments work in both model structures. Let 
$id \wedge i: (\group/\rmK)_+ \wedge^{\rmS}  A \ra (\group/\rmK)_+ \wedge B$ denote a generating cofibration in $\oI_{\group}$, that is,
the map $i: \rmA \ra \rmB$ is in $\oI$. Here we make use of the observation that the fixed point functor $\rmQ \ra \rmQ^{\rmH}$ is {\it cellular}, that is,
it has the following two properties for simplicial presheaves (see \cite[Proposition 3.10]{Guill}): \index{injective model structure}
\be \begin{multline}
\begin{split}
     \label{pushouts.and.colimits.fixed.points}
\mbox{ (i) The fixed point functor }  \rmQ \ra \rmQ^{\rmH} \mbox{ preserves pushouts along any map of the form }\\
id \wedge^{\rmS} i: (\group/K)_+ \wedge^{\rmS} \rmA \ra (\group/K)_+ \wedge^{\rmS}  \rmB, \mbox{ with } i \eps \oI \mbox{ and }
\end{split}
\end{multline} \ee
\be \begin{multline}
     \begin{split}
\mbox{ (ii) if } \{ \rmP _{\alpha}| \alpha \} \mbox{ denotes a small filtered direct system of objects in } \Spc_*^{\group}(\rmS) 
\mbox{ that are sub-objects of  a  } \notag \\  
\rmP \eps \Spc_*^{\group}(\rmS), \mbox{ one obtains the identification }
(\colimalpha \rmP_{\alpha})^{\rmH} \cong \colimalpha (\rmP_{\alpha})^{\rmH}. \notag
\end{split} \end{multline} \ee
\vskip .3cm \noindent
 (The first property may be verified readily for the action of any group on
presheaves of pointed sets. The last property has already been observed above in general.)
Observe that $(id \wedge^{\rmS} i)^{\rmH} = \vee _{\group/H \ra \group/K} i$ which is clearly a cofibration. The combined effect of the above two properties is that
the fixed point functor $\rmQ \ra \rmQ^{\rmH}$ sends any $\oI_G-cell$ ($\oJ_G-cell$) to an $\oI-cell$ ($\oJ-cell$, \res). Therefore since the domains 
of maps in $\oI$
 ($\oJ$) are small relative to $\oI$-cell ($\oJ$-cell, \res), it follows that the domains of the maps in $\oI_{\group}$ ($\oJ_{\group}$) are
also small relative $\oI_{\group}-cell$ ($\oJ_{\group}-cell$, \res.)
\vskip .3cm
Recall that given a collection of maps $\oJ$ in a category, $\oJ-inj$ denotes those maps that have the right lifting property 
 with respect to every map in $\oJ$. 
One may now observe using the adjunction  that the fibrations defined above identify with $\oJ_{\group}-inj$  and that the trivial 
fibrations (that is, 
the fibrations that are also weak-equivalences) identify with $\oI_{\group}-inj$. Recall from \cite[2.1.2]{Hov99} that $\oI_{\group}-cof$ (that is, the $\oI_{\group}$-cofibrations)
are the maps $(\oI_{\group}-inj)-proj$, that is,  those maps that have the left lifting property with respect to every trivial fibration. We now observe that
every map in $\oJ_{\group}-cell$ is clearly a weak-equivalence. Therefore, suppose we are given a commutative square in $\Spc_*^{\group}(\rmS)$:
\[\xymatrix{{(\group /\rmH)_+ \wedge^{\rmS} \rmA} \ar@<1ex>[r] \ar@<1ex>[d] ^{id \wedge^{\rmS} j} & {\rmX} \ar@<1ex>[d]^{\rmp}\\
            {(\group/\rmH)_+ \wedge^{\rmS} \rmB} \ar@<1ex>[r] & \rmY}
\]
with $j \eps \oJ$ and $\rmp \eps \oI_{\group}-inj$. Then, by adjunction this corresponds to the commutative square:
\[\xymatrix{{ \rmA} \ar@<1ex>[r] \ar@<1ex>[d] ^{ j} & {\rmX^{\rmH}} \ar@<1ex>[d]^{\rmp^{\rmH}}\\
            { \rmB} \ar@<1ex>[r] & {\rmY^{\rmH}}}
\]
in $\Spc_{*}(\rmS)$. Now $\rmp^{\rmH}$ is a trivial fibration and $j$ is a generating trivial cofibration, so that one obtains a 
lifting: $\rmB \ra X^{\rmH}$ making the two triangles commute. By adjunction this lift corresponds to a lift $(\group/\rmH )_+ \wedge \rmB \ra \rmX$ in the
first diagram. This proves any map in $\oJ_{\group}-cell $ is in $\oI_{\group}-cof$. 
\vskip .3cm
Since $\oI_{\group}-inj$ corresponds to trivial fibrations, every map in $\oI_{\group}-inj$ is a weak-equivalence and  it is in 
 $\oJ_{\group}-inj$ (which  denote the  fibrations). Since $\oJ_{\group}-inj$ denotes fibrations, it is clear that any map that is in
$\oJ_{\group}-inj$ and is also a weak-equivalence is also in $\oI_{\group}-inj$ (which denotes trivial fibrations.)
Therefore, we have verified all the hypotheses in \cite[Theorem 2.1.19]{Hov99} and therefore the {\it first} statement that the structures
in Definition ~\ref{equiv.model.struct} define a cofibrantly generated model category in the theorem is proved.
\vskip .3cm
The left-properness may be established using the cellularity of the fixed point functors considered above. The right properness is clear since
the fixed point functor preserves pull-backs. 
\vskip .3cm
Next observe that for any $\group$, the objects of the form $(\group/\rmH)_+ \wedge^{\rmS} (\Delta [n]_+ \wedge^{\rmS} {\rm h_{U+}})$ as $\rmU \eps \Sm_S$, $\rmH \eps \W$ and $n \ge 0$
 vary form a set of generators for $\Spc_*^{\rmG}(\rmS)$. 
Therefore,  every object in the above categories is a filtered colimits of objects $\{G_{\alpha}|\alpha\}$ 
 obtained as finite colimits of the above generators.
 In the projective model structure, it is clear from the choices of the sets $\oI_{\group}$ and $\oJ_{\group}$ that
every object $(\group/\rmH)_+ \wedge^{\rmS}  {\rm h_{X+}}$, $\rmX \eps \Sm/S$ is cofibrant. (These follow from the fact that each object ${\rm h_{X+}}$ is cofibrant in
the projective model structure on $\Spc_*(\rmS)$.) In the injective model structure, all monomorphisms are cofibrations, 
so that the
domains of the sets $\oI_{\group}$ and $\oJ_{\group}$ are cofibrant. Therefore, it follows that these model structures are also tractable. These prove (ii).
\vskip .3cm
Finally we prove that the model structures in Definition ~\ref{equiv.model.struct} define a symmetric monoidal model category structure on $\Spc_*^{\group}(\rmS)$
 with respect to the monoidal structure given by $\wedge ^{\rmS}$. Observe from 
\cite[Corollary 4.2.5]{Hov03} that in order to prove the pushout-product axiom holds in general, it suffices to prove that the pushout product of two generating cofibrations is a cofibration and
that this pushout-product is also a weak-equivalence when one of the arguments is a generating trivial cofibration.
Therefore, let $(\group/H)_+ \wedge^{\rmS} i: (\group/H)_+ \wedge ^{\rmS} \rmA \ra (\group/H)_+ \wedge ^{\rmS} \rmB$ and let 
$(\group/K)_+ \wedge ^{\rmS} j: (\group/K)_+ \wedge ^{\rmS} \rmX \ra
(\group/K)_+ \wedge ^{\rmS} \rmY$ denote two generating cofibrations in $\Spc_*(\rmS)^{\group}$. Then
a key observation is that $(\group/H)_+ \wedge ^{\rmS} (\group/K)_+ \cong \vee (\group/(H \cap K))_+$ where the $\vee$ is over the 
orbits of $\group$ for the diagonal action of $\group$ on $\group/H \times \group/K$. Therefore, it suffices to prove that for ${\rm E}$
a fibrant object in $\Spc_*^{\group}(\rmS)$, the induced map
\be \begin{multline}
     \begin{split}
  \Hom((\group/(H\cap K))_+ \wedge ^{\rmS} B \wedge ^{\rmS} Y, E) \ra\\
    \Hom((\group/(H \cap K))_+ \wedge ^{\rmS} A \wedge ^{\rmS} Y, E) {\underset {\Hom((\group/(H\cap K))_+ \wedge ^{\rmS} A \wedge ^{\rmS} X, E)} \times} \Hom((\group/(H\cap K))_+ \wedge ^{\rmS} B \wedge ^{\rmS} X, E)
     \end{split}
\end{multline} \ee
is a fibration in $\Spc_*^{\group}(\rmS)$, which is a weak-equivalence if $i$ or $j$ is also weak-equivalence and where $\Hom$
denotes the appropriate internal hom.
The above map now identifies with
\[ \Hom( {\rm B \wedge ^{\rmS} Y}, \rmE^{\rm H \cap K}) \ra \Hom( {\rm A \wedge ^{\rmS} Y}, \rmE^{\rm H\cap K}) {\underset {\Hom( {\rm A \wedge ^{\rmS} X}, \rmE^{\rm H \cap K})} \times} \Hom({\rm B \wedge ^{\rmS} X}, {\rm E}^{\rm H \cap K}) \]
where $\Hom$ now denotes the internal hom in $\Spc_*(\rmS)$. Therefore, the fact that the above map is a fibration 
and that it is a trivial fibration if $i$ or $j$ is also a weak-equivalence follows from the fact that the pushout-product axiom
holds in $\Spc_*(\rmS)$.
\vskip .3cm
 The unit for the smash-product $\wedge ^{\rm S}$ on $\Spc_*^{\group}(\rmS)$ defined in ~\eqref{wedgeS} is ${\rm S}_+$ and this is
cofibrant in both the injective and projective model structures on $\Spc_*^{\group}(\rmS)$. 
\end{proof}

\vskip .3cm

\begin{remark}
With the above model structure on $\Spc_*^{\rmG}(\rmS)$, one may readily see that the forgetful functor $\rmU: \Spc_*^{\rmG}(\rmS) \ra \Spc_*(\rmS)$
is a left Quillen functor, but an object $\rmP \in \Spc_*^{\rmG}(\rmS)$ so that $\rmU(\rmP)$ is cofibrant in $\Spc_*(\rmS)$ need not be
cofibrant in $\Spc_*^{\rmG}(\rmS)$. 
An alternate way to put a model structure on the category
$\Spc_*^{\rmG}(\rmS)$ is to transfer the model structure on $\Spc_*(\rmS)$ by means of the underlying functor $\rmU$ and a left adjoint to it. This adjoint is given by the functor sending a simplicial presheaf $\rmP$ to $\group \otimes \rmP$ (which is
defined by $(\group \otimes P)(\rmX) = {\underset {G(\rmX)} \vee } P(\rmX)$. Again an object $\rmP \in \Spc_*^{\rmG}(\rmS)$ so that
$\rmU(\rmP)$ is cofibrant as an object in $\Spc_*(\rmS)$ need not be cofibrant in $\Spc_*^{\rmG}(\rmS)$.
\end{remark}
\subsubsection{\bf A Key observation}
\label{key.observ.1}
As a result the composite functor $\RHom( \quad, \quad) \circ \rmU$ will be in general distinct from $\rmU \circ\RHom^{\rmG} (\quad, \quad)$, where
$\RHom (\quad, \quad) $( $\RHom^{\rmG}(\quad, \quad)$) denotes the internal derived Hom in $\Spc_*(\rmS)$ ($\Spc_*^{\rmG}(\rmS)$, \res).
 Recall that the notion of Spanier-Whitehead duality we will need to use involves  stable versions of the corresponding functors
in the non-equivariant framework: hence the above approach is not helpful for us. (See the introductory paragraph to section 8 for more on this.)
\vskip .1cm
Therefore, we need to obtain an analogue of Proposition ~\ref{functorial.rep.inhert.G.act}  
for spectra: one of the  goals of the discussion in the next section, is to accomplish this while at the same time 
setting up a category of spectra with group actions to be used throughout the rest of the paper.
\vskip .2cm
\section{\bf Categories of spectra} Spectra play {\it two distinct roles} in our context:
\label{spectra.types}
	\begin{enumerate} [\rm(i) ]
	\item {One may observe that the definition of the transfer is as a stable map of certain spectra, and its applications are to splitting  maps of generalized cohomology theories defined with respect to spectra. Here spectra mean either motivic or \'etale spectra which { are not necessarily equivariant}. { Moreover, the notion of Spanier-Whitehead dual that is needed for the transfer is essentially in the non-equivariant setting.}}	
    \item {In contrast, the construction of the transfer as a stable map starts with a { pre-transfer,
    	which will have to be an equivariant map of equivariant spectra, which is then fed into the Borel-construction to obtain the transfer for generalized (Borel-style) equivariant cohomology theories}. (Equivariant spectra are defined below.)}
    \item{{ Thus, the spectra that enter into the construction of a pre-transfer (which have to be equivariant)
    	  are all equivariant spectra, though the  transfer is applied to generalized cohomology theories 
    	  that are defined with respect to spectra that need not be equivariant.} This dual role of spectra, makes it necessary for us to proceed carefully and explaining how the two roles are related.
    	  }
    \item {When the group ${\group}$ is a finite group, the regular representation of ${\group}$ will contain all
		the irreducible representations (at least in characteristic $0$), so that one may define a suspension functor
		by taking the smash product with the Thom-space of the regular representation. As a result one can then
		define symmetric ${\group}$-equivariant spectra readily as one does in the non-equivariant case. Since our interest
		is mainly when the group ${\group}$ is a linear algebraic group of positive dimension, one cannot adopt this framework of symmetric spectra, 
		which is why we have defined the category $\Spt^{\rmG}(\rmS)$ as in the following discussion.}
    \item {As pointed out earlier, we consider actions by {\it all linear algebraic groups} on schemes both in the \'etale and motivic frameworks. Prior work
in this area has been restricted to actions by special classes of algebraic groups, such as those that are {\it special} in the sense of \cite{Ch} (as in \cite{Lev18})
and {\it linearly reductive groups} such as in \cite{Ho}, which in positive characteristics, are just products of tori and finite abelian groups with torsion
 prime to the characteristic. The need to consider actions by all linear algebraic groups is essential to obtain the full range of applications of the transfer,
 and this makes it necessary for us to consider equivariant unstable and stable homotopy theory in the \'etale and motivic framework.  We do this as
 concisely and briefly as possible.}
       \end{enumerate}
\subsubsection{\bf Equivariant spectra} 
\label{equiv.sph.sp.0} \index{Equivariant spectra} \index{Enriched symmetric monoidal category}
Throughout the following discussion, we will adopt the following { terminology}: ${\group}$ will denote a fixed linear algebraic group
defined over the base scheme (which we assume again is a perfect field) and $\C$ will denote the category $\Spc_*(\rmS)$, while $\C^{\group}$ will denote the 
category $\Spc_*^{\rmG}(\rmS)$. Here $\rmS$ could be either the base scheme or a fixed
simplicial presheaf, so that all the simplicial presheaves we consider will be pointed over $\rmS$. 
\vskip .2cm	 
	 The ${\group}$-spectra will be indexed not by the non-negative integers, but by the Thom-spaces of finite dimensional representations 
	 of the group-scheme ${\group}$ (i.e., affine spaces over $\k$ provided with a linear action by $\rmG$). 
	 We will fix a set of finite dimensional representations of $\rmG$, which contains all irreducible representations, and is 
	 closed under finite direct sums, just as is done in the topological framework in \cite[1.2]{LMS}. Henceforth, we will only consider representations that belong to this set. 
	 For each finite dimensional representation $\rmV$ of $\rmG$, we let $\rmT_{\rmV}^{\rmS} = \rmT_{\rmV} \times_{\Speck} \rmS$.
	 Then we let
	$\Sph^{\group}_{\rmS}$ denote the {\it subcategory} of $\C^{\group}$ whose {\it objects} are $\{\rmT_{\rmV}^{\rmS} |\rmV\}$, and
	where $\rmV$ varies over all finite dimensional representations of the group ${\group}$ and $\rmT_{\rmV}$ denotes its Thom-space. 
	We let the {\it morphisms} in this category be given by the maps $\rmT_{\rmV}^{\rmS} \ra \rmT_{\rmV \oplus \rmW}^{\rmS}$
	induced by homothety classes of $k$-linear injective and $\rmG$-equivariant maps $\rmV \ra \rmV \oplus \rmW$. One may observe that
	$\rmT_{\rmV}$ identifies with the quotient sheaf ${\rm Proj}(\rmV \oplus 1) /{\rm Proj}(\rmV)$, so that there is an
	injection  $\rmV \ra \rmT_{\rmV}$ for every $\rmG$-representation $\rmV$. 
	\vskip .1cm
	Let $\T={\mathbb P}^1$ pointed by $\infty$. \index{$\T$} We also let $\Sph_{\rmS}$ denote the category whose objects are
	 $\{\T^{\wedge n}_{\rmS} | n \ge 0 \}$, but given the structure of $\C$-enriched category as follows.
	 First, the morphisms in this category are given by the maps $\T^{\wedge n}_{\rmS} \ra \T^{\wedge n+m}_{\rmS}$
	induced by homothety classes of $k$-linear injective  maps ${\mathbb A}^n \ra {\mathbb A}^{n+m}$.
	\vskip .1cm
	We will make $\Sph^{\group}_{\rmS}$ ($\Sph_{\rmS}$) an {\it enriched monoidal category, enriched over the category $\C^{\group}$ ($\C$, \res)} as follows.
	 First let $\rmS^0 = \rmS_+ = \rmS \sqcup \rmS$. \index{$\Sph^{\group}_{\rmS}$} \index{$\Sph_{\rmS}$} Then for
	 $\rmV, \rmW$ that are ${\group}$-representations, we let the $\C^{\rmG}$-enriched internal hom in $\Sph^{\rmG}_{\rmS}$ be defined by: 
	 \be \begin{align}
	 \label{G.enrich.hom}
	 \Hom_{\C^{\group}}(\rmT_{\rmV}^{\rmS}, \rmT_{\rmV \oplus \rmW}^{\rmS}) &= (\underset {\alpha: \rmV \ra \rmV \oplus \rmW} \sqcup \rmT_{\rmW}^{\rmS}) \sqcup \rmS, \rmW \neq \{0\}\\
	 								  &= (\underset {\alpha: \rmV \ra \rmV } \bigvee \rmS^0) \sqcup \rmS, \rmW=\{0\}. \notag
	 \end{align} \ee
	\vskip .1cm \noindent
	Here  the sum varies over all {\it homothety classes of ${\group}$-equivariant and $k$-linear injective maps} $\alpha:\rmV \ra \rmV \oplus \rmW$ and the summand $\rmS$ denotes
	a {\it base point} added so that the above enriched $\Hom$s are pointed simplicial presheaves over $\rmS$.
	 The base points in
	each of the summands $\rmT_{\rmW}^{\rmS}$ correspond bijectively with the corresponding $\alpha$; similarly the 
	unique $0$-simplex other than the base point in each of the summands $\rmS^0$ corresponds bijectively with the
	corresponding $\alpha$. As a result, the $0$-simplices in  $\Hom_{\C^{\group}}(\rmT_{\rmV}^{\rmS}, \rmT_{\rmV \oplus \rmW}^{\rmS})$
	correspond bijectively with the morphisms $\rmT_{\rmV} \ra \rmT_{\rmV \oplus \rmW}$ in the { category} underlying the
	enriched category $\Sph^{\rmG}_{\rmS}$.
	 One defines
	the $\C$-enriched internal hom in $\Sph_{\rmS}$ by a similar formula as in ~\eqref{G.enrich.hom}:
	\be \begin{align}
	\label{enrich.hom}
	\Hom_{\C}(\T^{\wedge n}_{\rmS}, \T^{\wedge n+m}_{\rmS}) &= (\underset {\alpha: {\mathbb A}^n \ra {\mathbb A}^{n+m}} \sqcup \T^{\wedge m}_{\rmS}) \sqcup \rmS, {\it m} > 0\\
	&= (\underset {\alpha: {\mathbb A}^n \ra {\mathbb A}^n}  \bigvee \rmS^0 )\sqcup \rmS, {\it m}=0, \notag
	\end{align} \ee
	\vskip .1cm \noindent
	where now $\alpha$ varies over  homothety classes of  $k$-linear injective maps ${\mathbb A}^n \ra {\mathbb A}^{n+m}$. 
	In particular, when $m=0$, the general linear group ${\rm GL}_n$ acts on $\T^{\wedge n}$.
	\begin{proposition}
	   \label{SphG.symm.mon}
	   With the above definitions, the category $\Sph^{\group}_{\rmS}$ is a symmetric monoidal $\C^{\group}$-enriched category, where the
	   monoidal structure is given by $\rmT_{\rmV}^{\rmS} \wedge \rmT_{\rmW}^{\rmS} = \rmT_{\rmV \oplus \rmW}^{\rmS}$.  $\Sph$ is a symmetric monoidal $\C$-enriched category, where the
	   monoidal structure is given by $\rmT_{\rmS}^{n} \wedge \rmT_{\rmS}^m = \rmT_{\rmS}^{n+m}$. The forgetful functor
$j:\Sph^{\rmG}_{\rmS} \ra \Sph_{\rmS}$ is an enriched functor of $\C$-enriched categories.	
	\end{proposition}
	\begin{proof} We first verify that $\Sph^{\group}_{\rmS}$ is a $\C^{\group}$-enriched category. To see this, observe that if
	$\rmf:\rmU \ra \rmU\oplus \rmV$ is a ${\group}$-equivariant injective linear map and ${\rm g}: \rmV \ra \rmV\oplus \rmW$ is a ${\group}$-equivariant injective
	 linear map, the composition $(id \oplus {\rm g})\circ \rmf: \rmU \ra \rmU \oplus \rmV \oplus \rmW$ is an injective linear map that is
	 also ${\group}$-equivariant. The composition $\Hom_{\C^{\group}}(\rmT_{\rmU}^{\rmS}, \rmT_{\rmU\oplus \rmV}^{\rmS}) \times \Hom_{\C^{\group}}(\rmT_{\rmV}^{\rmS}, T_{\rmV \oplus \rmW}^{\rmS})  \ra
	 \Hom_{\C^{\group}}(\rmT_{\rmU}^{\rmS}, \rmT_{\rmU \oplus \rmV \oplus \rmW}^{\rmS})$ sends the summand $\rmT_{\rmV}^{\rmS} $ indexed by $\rmf$ and the summand $\rmT_{\rmW}^{\rmS}$ indexed by	 ${\rm g}$ to the summand $\rmT_{\rmV \oplus \rmW}^{\rmS}$ indexed by $(id \oplus {\rm g}) \circ \rmf: \rmU \ra \rmU \oplus \rmV \oplus \rmW$. 
	 One may now see readily that this pairing is associative and unital, so that $\Sph^{\group}_{\rmS}$ is a $\C^{\group}$-enriched category: see \cite[6.2]{Bor}.
	  \vskip .1cm
	  The monoidal structure sends $(\rmT_{\rmU}^{\rmS}, \rmT_{\rmV}^{\rmS}) \mapsto \rmT_{\rmU}^{\rmS} \wedge \rmT_{\rmV}^{\rmS} = \rmT_{\rmU \oplus \rmV}^{\rmS}$. One may now observe that the
	  associativity isomorphisms $(\rmU \oplus \rmV) \oplus \rmW \cong \rmU \oplus (\rmV \oplus \rmW)$ and the commutativity
	  isomorphism $\rmU \oplus \rmV \cong \rmV \oplus \rmU$ are both ${\group}$-equivariant maps. Therefore, one observes that the
	  monoidal structure defined by the smash product $(\rmT_{\rmU}^{\rmS}, \rmT_{\rmV}^{\rmS}) \mapsto \rmT_{\rmU}^{\rmS} \wedge \rmT_{\rmV}^{\rmS} = \rmT_{\rmU \oplus \rmV}^{\rmS}$ makes the
	  category $\Sph^{\group}_{\rmS}$ a  symmetric monoidal category. One may then readily verify that the same pairing is
	  a functor of $\C^{\group}$-enriched categories, which will prove $\Sph^{\group}_{\rmS}$ is a symmetric monoidal $\C^{\group}$-enriched category.
	  \vskip .1cm
	  The statements regarding the $\C$-enriched category $\Sph_{\rmS}$ may be proven
	  similarly. We skip the proof that $j$ is a simplicially enriched functor.
	\end{proof}
\vskip .2cm
Next we proceed to  define various categories of spectra as enriched functors. Given a symmetric monoidal category ${\mathbf {\nu}}$, and two 
${\mathbf {\nu}}$-categories ${\mathcal A}$ and  ${\mathcal B}$, $[{\mathcal A}, {\mathcal B}]$
will denote the ${\mathbf {\nu}}$-enriched category of ${\mathbf {\nu}}$-enriched functors ${\mathcal A} \ra {\mathcal B}$. (See \cite[6.2, 6.3]{Bor}.)
	\begin{definition} (The category $\Spt^{\rmG}(\rmS)$, Smash products and internal Hom in $\Spt^{\rmG}(\rmS)$). 
		\label{smash.prdcts}
		\begin{enumerate}[\rm(i)]
	\item	We define $\Spt^{\rmG}(\rmS)$ to denote the $\C^{\rmG}$-enriched category of $\C^{\rmG}$-enriched functors
		$\Sph^{\rmG}_{\rmS} \ra \C^{\rmG}$, i.e., $\Spt^{\rmG}(\rmS) =[\Sph^{\rmG}_{\rmS}, \Spc_*^{\rmG}(\rmS)]$. \index{$\Spt^{\rmG}(\rmS)$}
	\item   In particular, $\Spt^{\rmG}(\rmS_{\rm mot}) =[\Sph^{\rmG}_{\rmS}, \Spc_*^{\rmG}(\rmS_{\rm mot})]$ and 
		$\Spt^{\rmG}(\rmS_{et}) =[\Sph^{\rmG}_{\rmS}, \Spc_*^{\rmG}(\rmS_{et})]$.
		
	\item	Observe that the $\C^{\group}$-enriched category $\Sph^{\group}$ is {\it symmetric} monoidal with respect to
		the smash product of Thom-spaces. As a result (see \cite{Day}), if  $\X$, $\Y$  are two ${\group}$-spectra, viewed
		as enriched functors $\Sph^{\group}_{\rmS} \ra \C^{\group}$, their smash product $\X \wedge \Y$ defined as the left-Kan extension
		with respect to the monoidal product $\wedge: \Sph^{\group}_{\rmS} \times \Sph^{\group}_{\rmS} \ra \Sph^{\group}_{\rmS}$, will {\it  define}
		a smash product that is symmetric monoidal on $\Spt^{\rmG}(\rmS)$. The smash product identifies with the following co-end taking values 
		in the symmetric monoidal category $\Spc_*^{\rmG}(\rmS)$:
		\be \begin{equation}
			\X \wedge \Y= \int^{Ob(\Sph^{\group}_{\rmS} \otimes \Sph^{\group}_{\rmS})} \Hom_{\Sph^{\group}_{\rmS}}(\rmT_{\rmV} \wedge \rmT_{\rmW}, \quad) \wedge \X(\rmT_{\rmV}) \wedge \Y(\rmT_{\rmW}).
		\end{equation} \ee
	\item	The internal $\Hom(\X, \Y)$ is defined by the corresponding end:
		\be \begin{equation}
			\Hom(\X, \Y)(\rmT_{\rmV}) = \int_{\rmT_{\rmW} \in Ob(\Sph^{\group}_{\rmS})} \Hom_{\C^{\rmG}}(\X(\rmT_{\rmW}), \Y(T_{V \oplus W})).
		\end{equation} \ee
                \end{enumerate}
	\end{definition} \index{smash product of spectra}
\begin{remark}
It may be important to point out that taking $\rmG$ to be trivial does {\it not} define the familiar category $\Spt(\rmS)$ 
indexed by $\{\T^{\wedge n}|n \ge 0\}$, but a category of spectra which we show (see:  section ~\ref{comparison.cat.Spt}) is Quillen equivalent to the latter category. In fact taking $\rmG$
to be trivial will produce the category of spectra denoted ${\widetilde {\Spt}}(\rmS)$ defined below in 
Definition ~\ref{tildeSpt.1}.
\end{remark}
	
\vskip .1cm
\begin{definition} (The equivariant sphere spectrum and suspension spectra) \index{The equivariant sphere spectrum $\mbS^{\rmG}_{\rmS}$} \index{$\mbS^{\rmG}_{\rmS}$}
 \label{equiv.sph.sp}
 \begin{enumerate}[\rm(i)]
\item The equivariant sphere spectrum $\mbS^{\rmG}_{\rmS}$ will be defined to be the object in $\Spt^{\rmG}(\rmS)$ given by the functor $\Sph^{\rmG}_{\rmS} \ra \C^{\rmG}$, that is, 
 $\mbS^{\rmG}(\rmT_{\rmV}^{\rmS}) = \rmT_{\rmV}^{\rmS}$, $\rmT_{\rmV}^{\rmS} \in \Sph^{\rmG}$.  
 When $\rmG$ is trivial, this defines {\it a sphere spectrum}, which will be denoted 
 $\tilde \mbS_{\rmS}$. (Observe that this defines an object in the category ${\widetilde {\Spt}}(\rmS)$ defined in Definition ~\ref{tildeSpt.1}.)
 \item
 On the Nisnevich site (\'etale site) of $\Speck$, $\mbS^{\rmG}_{\rmS}$ will define the motivic sphere spectrum (the \'etale sphere spectrum), which will be denoted 
 $\mbS^{\rmG}_{\rmS, mot}$ or often simply $\mbS^{\rmG}_{\rmS}$ ($\mbS^{\rmG}_{\rmS, et}$, \res). 
 \item
 When $\rmS= \Speck$, we will denote $\mbS^{\rmG}_{\rmS, mot}$ ($\mbS^{\rmG}_{\rmS, et}$)
 by $\mbS^{\rmG}_{\k, mot}$ or simply $\mbS^{\rmG}_{\k}$ ($\mbS^{\rmG}_{\k, et}$, \res).
 \end{enumerate}
 \end{definition}

\begin{definition} (The equivariant motivic Eilenberg-Maclane spectrum) \index{The equivariant motivic Eilenberg-Maclane spectrum}
 \label{equiv.EM.sp}
 \begin{enumerate}[\rm(i)]
\item Let $\Spc_*^{tr}(\rmS)$ denote the category of all simplicial abelian presheaves with transfers on the Nisnevich site of $\Speck$ and 
pointed over $\rmS$. Let $\U: \Spc_*^{tr}(\rmS) \ra \Spc_*(\rmS)$ denote
 the forgetful functor sending a simplicial abelian presheaf with transfers to the underlying pointed simplicial presheaf. For each representation $\rmV$
 of $\rmG$, and a commutative Noetherian ring $\rmR$ with $1$, we let ${\rmR}^{tr}(\rmT_{\rmV}) = cokernel (({\mathbb Z}^{tr}({\rm Proj}(\rmV)) \otimes \rmR) \ra {\mathbb Z}^{tr}({\rm Proj}(\rmV \oplus 1)) \otimes \rmR)$.
 Now we let ${\mathbb H}({\rmR})^{\rmG}_{\rmS}= \{\U({\rmR}^{tr}(\rmT_{\rmV}))|\rmV\}$. The structure maps are given by:
 \be \begin{equation}
 \label{equiv.EM.sp.1}
 \rmT_{\rmW} \wedge \U({\rmR}^{tr}(\rmT_{\rmV})) \ra \U({\rmR}^{tr}(\rmT_{\rmW})) \wedge \U({\rmR}^{tr}(\rmT_{\rmV})) \ra \U({\rmR}^{tr}(\rmT_{\rmW}) \otimes^{tr} {\rmR}^{tr}(\rmT_{\rmV})) \cong \U({\rmR}^{tr}(\rmT_{\rmW \oplus \rmV})),
 \end{equation} \ee
 where $\otimes^{\rm tr}$ denotes the monoidal structure on the category of simplicial abelian presheaves with transfers.
 (We skip the verification that this defines a ring spectrum
 in $\Spt^{\rmG}(\k_{\rm mot})$.) 
 \item
 Taking $\rmG$ to be trivial defines the {\it motivic Eilenberg-Maclane spectrum}, which will be denoted $\tilde {\mathbb H}({\rmR})_{\rmS}$ (and when
 $\rmS = \Speck$ by $\tilde {\mathbb H}({\rmR})_{\k}$).
 \end{enumerate}
 \index{${\mathbb H}({\rmR})^{\rmG}_{\rmS}$}
\end{definition}

	\vskip .2cm
	{\it In case $\cE^{\rmG}$ is a commutative ring spectrum} in $\Spt^{\rmG}(\rmS)$, we will let $\Spt^{\rmG}(\rmS, {\cE^{\rmG}})$ denote the category 
	consisting of module spectra over $\cE^{\rmG}$ and their
	maps. In this case, the smash product $\wedge$ will be replaced by $\wedge _{\cE^{\rmG}}$ which is defined as
	\be \begin{equation}
		\label{wedgeE}
		\rmM \wedge _{\cE^{\rmG}} \rmN = Coeq( \rmM \wedge \cE^{\rmG} \wedge  \rmN \stackrel{\rightarrow}{\rightarrow} \rmM \wedge \rmN)
	\end{equation} \ee
	\vskip .2cm \noindent
	where the two maps above make use of the module structures on $\rmM$ and $\rmN$, \res. The corresponding internal $Hom$ will be denoted $\Hom _{\cE^{\rmG}}$.
	\index{$\wedge_{\cE^{\rmG}}$} \index{$\Hom _{\cE^{\rmG}}$}
	\vskip .2cm
	The main ${\group}$-equivariant ring spectra of interest to us,
	including the sphere spectrum $\mbS^{\group}$, will be the following: 
	\be \begin{multline}
	    \label{choice.ring.spectra}
		\begin{split}
			(i)\quad \mbS^{\rmG}, \quad (ii)\quad \mbS^{\group}_{\rmS}[{\it p}^{-1}] \mbox{ if the base scheme } \rmS \mbox{ is a field of characteristic } p, \\
			{\it (iii)} \quad \mbS^{\group}_{{\rmS},(\ell)}, \mbox{ which denotes the localization of } \mbS^{\group}_{{\rmS}} \mbox{ at the prime ideal } (\ell), \ell \mbox{ is a prime} \ne char(\k),\\
			(iv) \quad {\widehat {\mbS}}^{\group}_{{\rmS},\ell}, \mbox{ where } \ell \mbox{ is a prime } \ne char(\k) \mbox{ and } \\
			{\widehat {\mbS}}^{\group}_{{\rmS},\ell} \mbox{ denotes the }\ell-\mbox{completed }\rmG-\mbox{equivariant sphere spectrum} \mbox{ as well as } \\
			(v) \quad {\mathbb H}({\rmR})^{\rmG}_{\rmS} \mbox{ with } \rmR= {\mathbb Z}/\ell^n, \mbox{ with } \ell \mbox{ a prime }  \ne char(\k).
		\end{split}
	\end{multline} \ee
	Here the completion
	 at the prime $\ell$ is the Bousfield-Kan completion discussed in \cite[Appendix]{CJ23-T2}.
	\index{$\mbS^{\group}[{\it p}^{-1}]$} \index{$ \mbS^{\group}_{(\ell)}$} \index{${\widehat {\mbS}}^{\group}_{\ell}$}
	\vskip .2cm
 \vskip .2cm
 \begin{definition}
 \label{usual.spectra}
 \begin{enumerate}[\rm(i)]
  \item 
  Let $\Spt(\rmS_{\rm mot})$ denote the (usual) category of motivic spectra defined as follows.  Its objects are 
 $\X=\{\rmX_n \in \Spc_*(\rmS_{\rm mot}), \mbox{ along with structure maps } \T^{\wedge m} \wedge \rmX_n \ra \rmX_{n+m} |n, m \in {\mathbb N}\}$.  Morphisms between two such objects  $\X$ and $\Y$ are defined as compatible collection of maps
  $\X_n \ra \Y_n$, $n \in {\mathbb N}$ compatible with suspensions by $\T^{\wedge m}$, $m \in {\mathbb N}$. 
  When $\rmS = \Speck$, this category will be denoted $\Spt(\k_{\rm mot})$ of often simply $\Spt(\k)$.
  \item
  The unit of this category
  is the motivic sphere spectrum denoted $\mbS_{\k}$. For a simplicial presheaf $\rmP \in \Spc_*(\k_{\rm mot})$, the suspension spectrum $\mbS_{\k} \wedge \rmP$ will be denoted $\Sigma_{\T}^{\infty}\rmP$.
  \item
  $\Spt(\rmS_{et})$ will denote the corresponding category of $\T$-spectra defined by starting with the category $\Spc_*(\rmS_{et})$
  of pointed simplicial presheaves on the big \'etale site of $\Speck$ and pointed over $\rmS$. 
  \item
  The unit of this category
  is the  sphere spectrum denoted $\mbS_{\k_{et}}$.
  \end{enumerate}
\end{definition}

 \begin{remark}
 \label{tricky.aspect}
 We begin with the following remarks to motivate the constructions below.
 Here is {\it a particularly tricky aspect} of the construction of the pre-transfer.
The Spanier-Whitehead duality one needs to invoke in the construction of the pre-transfer is in the setting of non-equivariant spectra
and {\it not} in a corresponding  category of equivariant spectra, such as the ones discussed above. There are several
reasons for this choice, some of which are: 
\begin{enumerate}[\rm(i)]
\item Currently one does {not} have Spanier-Whitehead duality for algebraic varieties in the equivariant framework, since one
does not yet have equivariant versions of Gabber's refined alterations.
\item 
For the construction of the transfer in the context of Borel-style generalized equivariant cohomology theories this is all
that is needed: see, for example, \cite{BG75}. In more detail: all one needs in this context is Spanier-Whitehead duality in a {\it non-equivariant setting, but applied to spectra with group actions.} 
\item On the other hand, we still need
the Spanier-Whitehead dual of an object with a ${\group}$-action to inherit a nice ${\group}$-action
and we need to use sphere-spectra which also have non-trivial ${\group}$-actions. 
{ In fact, it is crucial that the source of the co-evaluation maps will have to be ${\group}$-equivariant (sphere) spectra: otherwise the spectra showing up as the target of the co-evaluation maps will have no ${\group}$-action}: see Definition ~\ref{equiv.sph.sp} and
~\ref{pretransfer}. In more detail:  though we only need a non-equivariant form of Spanier-Whitehead duality, one needs
to make all the constructions sufficiently equivariant so as to be able to feed them into the Borel construction.
\item
In \cite{BG75}, the way these issues are resolved is by making sure the Thom-Pontrjagin collapse map (which plays the role of the co-evaluation map) can be made equivariant. In our framework, 
the way we resolve
these problems is as follows. First we use ${\group}$-equivariant spectra to serve as the source of the co-evaluation maps. Then we observe that 
for the underlying non-equivariant spectrum, associated to an equivariant spectrum, but viewed as an object in the category
${\widetilde {\Spt}}^{\rmG}({\rmS_{\rm mot}})$ (defined below), one can find functorial fibrant and cofibrant replacements in the latter category, and the functoriality implies that these objects come equipped with
compatible ${\group}$-actions. Further, we show in section ~\ref{comparison.cat.Spt} that the model category ${\widetilde {\Spt}}^{\rmG}({\rmS_{\rm mot}})$ is Quillen equivalent to the usual category of non-equivariant spectra $\Spt$ considered in
Definition ~\ref{usual.spectra}.
Therefore, the dual we define will be making use of such functorial cofibrant and fibrant replacements of the underlying non-equivariant spectra in ${\widetilde {\Spt}}^{\rmG}({\rmS_{\rm mot}})$ 
 and therefore, though they correspond to  duals in $\Spt$, they still come equipped with nice ${\group}$-actions. It is precisely these issues that make it necessary for us to introduce and work with the categories ${\widetilde {\Spt}}^{\rmG}({\rmS_{\rm mot}})$ 
	 and ${\widetilde {\Spt}}({\rmS_{\rm mot}})$ of spectra that come in between, and relate the category of equivariant spectra $\Spt^{\rmG}(\rmS_{\rm mot})$ with the category of spectra $\Spt(\rmS_{\rm mot})$.
\end{enumerate}
\end{remark}
\vskip .2cm
We now introduce the following
intermediate categories, denoted ${\widetilde {\Spt}}^{\rmG}({\rmS})$,
 and ${\widetilde {\Spt}}({\rmS})$, intermediate between $\Spt^{\rmG}(\rmS)$ and $\Spt(\rmS)$ defined in Definition ~\ref{usual.spectra}. \index{${\widetilde {\Spt}}^{\rmG}({\rmS})$} \index{${\widetilde {\Spt}}({\rmS})$}
 \begin{definition}
 \label{tildeSptG}
 \begin{enumerate}[\rm(i)]
\item The  $\C$-enriched category ${\widetilde {\Spt}}^{\rmG}({\rmS}) =[\Sph^{\rmG}_{\rmS}, \Spc_*(\rmS)]$. 
Therefore, the objects of this category
are $\C$-enriched functors 
$\tilde \X':\Sph^{\rmG}_{\rmS} \ra \C$, where $\C=\Spc_*(\rmS)$.
One may observe that an object in this category is given by $\{\tilde \X' (\rmT_{\rmV}^{\rmS})|\rmT_{\rmV}^{\rmS} \in \Sph^{\group}\}$,
provided with a compatible family of structure maps $\rmT_{\rmW}^{\rmS, \alpha} \wedge \tilde \X'(\rmT_{\rmV}^{\rmS}) \ra \tilde \X'(T_{W\oplus V}^{\rmS})$ in $\Spc_*(\rmS)$, with $\rmT_{\rmW}^{\rmS, \alpha}= \rmT_{\rmW}^{\rmS}$ as $\alpha$ varies over all
 homothety classes of $k$-linear $\rmG$-equivariant injective maps $\rmV \ra \rmV \oplus \rmW$. However, the maps $\rmT_{\rmW}^{\rmS, \alpha} \wedge \tilde \X'(\rmT_{\rmV}^{\rmS}) \ra \tilde \X'(T_{W\oplus V}^{\rmS})$ are no longer required to be
${\group}$-equivariant. 
\item
The { smash product and the internal hom} of spectra in ${\widetilde {\Spt}}^{\rmG}({\rmS})$ 
 are defined 
exactly as in the case of $\Spt^{\rmG}(\rmS)$, but making use of the category  $\Spc_*(\rmS)$ in the place of $\Spc_*(\rmS)^{\rmG}$.
\item
 ${\widetilde {\Spt}}^{\rmG}({\rmS_{\rm mot}})$ (${\widetilde {\Spt}}^{\rmG}(\rmS_{et})$) will denote the corresponding category
defined on the Nisnevich site by starting with $\Spc_*(\rmS_{\rm mot})$ (on the \'etale site by starting with $\Spc_*(\rmS_{et})$, \res) of $\Speck$.
\item
When $\cE^{\rmG} \in \Spt^{\rmG}(\rmS)$ is a commutative ring spectrum, one defines
the category ${\widetilde {\Spt}}^{\rmG}({\rmS}, {\cE^{\rmG}})$ similarly by replacing the pairings
$\rmT_{\rmW}^{\rmS} \wedge \X'(\rmT_{\rmV}) \ra \X'(T_{W\oplus V}^{\rmS})$ with the 
pairings: $\cE^{\rmG}(\rmT_{\rmW}^{\rmS}) \wedge \X'(\rmT_{\rmV}^{\rmS}) \ra \X'(T_{W\oplus V}^{\rmS})$.
\end{enumerate}
\end{definition}
\vskip .2cm
Observe that there is a forgetful functor
\be \begin{equation}
     \label{tildeU}
\tilde \rmU: \Spt^{\rmG}(\rmS) \ra {\widetilde {\Spt}}^{\rmG}({\rmS}) 
\end{equation} \ee
given by sending a $\X \in \Spt^{\rmG}(\rmS)$ to $\rmU \circ \X$, where $\rmU: \Spc_*^{\group}(\rmS) \ra \Spc_*(\rmS)$ is the forgetful functor. 
When $\cE^{\rmG} \in \Spt^{\rmG}(\rmS)$ is a commutative ring spectrum, one also obtains a forgetful functor $\tilde \rmU:\Spt^{\rmG}(\rmS, {\cE^{\rmG}}) \ra {\widetilde {\Spt}}^{\group}({\rmS}, {\cE^{\rmG}})$.
\vskip .2cm
\begin{definition}
 \label{tildeSpt.1}
 \begin{enumerate}[\rm(i)]
\item Let $\C $ denote the category $\Spc_*(\rmS)$. The  $\C$-enriched category ${\widetilde {\Spt}}({\rmS})
 = [\Sph_{\rmS}, \C]$. Therefore, the objects of this category 
 are given by
$\C$-enriched functors 
\be \begin{equation}
\label{USpt}
\X': \Sph_{\rmS} \ra \C.
\end{equation} \ee 
\vskip .1cm
Again, paraphrasing this,
such an object is given by 
 $\{\X'(\T^{\wedge n}_{\rmS})| n \ge 0\}$,
provided with a compatible family of structure maps $\T^{\wedge m}_{\rmS, \alpha} \wedge \X'(\T^{\wedge n}_{\rmS}) \ra \X'(\T^{\wedge m+n}_{\rmS})$ in $\Spc_*(\rmS)$, with $\T^{\wedge m}_{\rmS, \alpha}= \T^{\wedge m}_{\rmS}$ associated to each homothety class $\alpha$ of $k$-linear injective maps of ${\mathbb A}^n$ in ${\mathbb A}^{n+m}$, and the group of $k$-linear automorphisms of ${\mathbb A}^n$, (i.e., ${\rm GL}_n$) acts on $\X'({\T}^{\wedge n}_{\rmS})$. (In this sense, the category ${\widetilde {\Spt}}({\rmS})$ is similar to the category of what are called {\it orthogonal spectra}.) 
\vskip .1cm
Morphisms between two such objects 
$\{\Y'(\T^{\wedge n}_{\rmS})|n \ge 0\}$ and $\{\X'(\T^{\wedge n}_{\rmS})|n \ge 0\}$ are given by
compatible collections of maps $\{\Y'(\T^{\wedge n}_{\rmS}) \ra \X'(\T^{\wedge n}_{\rmS})|n \ge 0\}$ which are 
compatible with the pairings: 
$\T^{\wedge m}_{\rmS} \wedge \Y'(\T^{\wedge n}_{\rmS}) \ra \Y'(\T^{\wedge m+n}_{\rmS})$ and $\T^{\wedge m}_{\rmS} \wedge \X'(\T^{\wedge n}_{\rmS}) \ra \X'(\T^{n+m}_{\rmS})$. 
\item
${\widetilde {\Spt}}({\rmS_{\rm mot}})$ (${\widetilde {\Spt}}({\rmS_{et}})$) will denote the corresponding category
defined on the Nisnevich site (the \'etale site, \res) of $\Speck$.
\item 
The {smash product and the internal hom} of spectra in 
${\widetilde {\Spt}}({\rmS})$ are defined again
exactly as in the case of $\Spt^{\rmG}(\rmS)$, but making use of the categories $\Sph_{\rmS}$ and $\Spc_*(\rmS)$ in the place of $\Sph^{\rmG}_{\rmS}$ and $\Spc_*^{\rmG}(\rmS)$.
\item
Let $\tilde \rmU: \Spt^{\rmG}(\rmS) \ra {\widetilde {\Spt}}^{\rmG}(\rmS)$ and  $\tilde {\mathbb P}:{\widetilde {\Spt}}^{\rmG}(\rmS) \ra 
{\widetilde {\Spt}}(\rmS)$  denote the functors considered in  ~\eqref{passage.equiv.to.nonequiv}.
When $\cE^{\rmG} \in \Spt^{\rmG}(\rmS)$ is a commutative ring spectrum, one defines
the category ${\widetilde {\Spt}}({\rmS}, \tilde {\mathbb P} (\tilde \rmU({\cE^{\rmG}})))$ similarly by replacing the pairings
$\T^{\wedge m}_{\rmS} \wedge \X'(\T^{\wedge n}_{\rmS}) \ra \X'(\T^{n+m}_{\rmS})$ with the 
pairings: $\cE^{\rmG}(\T^{\wedge m}_{\rmS}) \wedge \X'(\T^{\wedge n}_{\rmS}) \ra \X'(\T^{n+m}_{\rmS})$.
\end{enumerate}


\end{definition}
\vskip .2cm
\begin{proposition}
 \label{tildeU.1}
Let $\X, \Y \in \Spt^{\rmG}(\rmS)$. Then, 
\be \begin{align}
 \tilde \rmU( \X \wedge \Y) = \tilde \rmU( X) \wedge \tilde \rmU( Y) &\mbox{ and }  \tilde \rmU (\Hom_{\Spt^{\rmG}(\rmS)}(\X, \Y)) = \Hom_{{\widetilde \Spt}^{\rmG}_{\rmS}}(\tilde \rmU(\X), \tilde \rmU(\Y)).
\end{align} \ee
Corresponding results hold for the categories $\Spt^{\rmG}(\rmS, {\cE^{\rmG}})$, ${\widetilde {\Spt}}^{\rmG}({\rmS}, {\cE^{\rmG}})$.
\end{proposition}
\begin{proof}
 The key observation is that the forgetful functor $\rmU: \Spc^{\rmG}_*(\rmS) \ra \Spc_*(\rmS)$ is a strict monoidal functor in the 
 sense $\rmU(\rmP \wedge^{\rmS} \rmQ) = \rmU(\rmP) \wedge^{\rmS} \rmU(\rmQ)$ and $\rmU(\Hom_{\rmG}(\rmP, \rmQ)) = \Hom(\rmU(\rmP), \rmU(\rmQ))$ as already observed in 
 ~\eqref{U.1}. In addition one also observes that the same forgetful functor preserves and 
 reflects all small colimits as well as all small limits. Therefore the definition of the smash product 
 (the internal hom) in the category $\Spt^{\rmG}(\rmS)$ as a co-end (end, \res) 
 in Definition ~\ref{smash.prdcts} along with the definition of the corresponding functors in
 ${\widetilde {\Spt}}^{\rmG}({\rmS})$ completes the proof.
 \end{proof}
\vskip .2cm
 \begin{terminology}
 \label{spectra.details}
 Model structures on the above categories of spectra: starting with the model structures on simplicial presheaves discussed 
 in section ~\ref{model.simpl.presh} one may put various model structures on the above categories of spectra. This is discussed in detail in the next two sections. We do not discuss the specific details of these model structures here, as we believe that will 
 take us away from our current discussion, except to point out that we choose to work with the category ${\widetilde {\Spt}}^{\rmG}(\rmS_{\rm mot})$
  and ${\widetilde {\Spt}}^{\rmG}(\rmS_{et})$ provided with injective stable model structures
 as discussed in  Proposition ~\ref{stable.model.1}: observe that  
 every object is cofibrant in this model structure. For the rest of the discussion in this section, we will implicitly make use of 
 this model structure. 
 \end{terminology}
 \begin{definition}
 \label{monoidal.functs}
 Let $\rmF: \C \ra \D$ denote a functor between monoidal categories.  
 We say $\rmF$ is a {\it weakly monoidal functor}, if for any pair of
objects $\X$, $\Y$ in $\C$, there is given a  natural map 
\be \begin{equation}
\label{monoidal.funct}
\mu: \rmF(\X) \otimes \rmF(\Y) \ra \rmF(\X \otimes \Y)
\end{equation} \ee
satisfying an associativity and unitality axiom 
as in \cite[Definition 6.4.1]{Bor}: note that
there it is called a morphism of monoidal categories, though the terminology we use seems more standard: see \cite{ncatlab}.
The unitality axiom says there is given a map $\epsilon:e_{\D} \ra \rmF(e_{\C})$, where $e_{\C}$ ($e_{\D}$) denotes the unit of the category $\C$ ($\D$, \res).
We say such a weakly monoidal functor is a {\it strong monoidal functor} ({\it strict monoidal functor}) if the map $\mu$ for all $\X$, $\Y$ in $\C$ and the 
map $\epsilon$ are isomorphisms (are the identity morphisms, \res).
\end{definition}
 \vskip .1cm
 \begin{definition} (Passing from equivariant spectra to non-equivariant spectra)
 \label{equiv.vs.nonequiv.spectra}
 One starts with the forgetful functor 
$\tilde \rmU: \Spt^{\rmG}(\rmS) \ra {\widetilde {\Spt}}^{\rmG}({\rmS})$. Since the indexing category for ${\widetilde {\Spt}}^{\rmG}({\rmS})$ is $\Sph^{\rmG}_{\rmS}$, while
 the indexing category for ${\widetilde {\Spt}}({\rmS})$ is $\Sph_{\rmS}$, the passage from ${\widetilde {\Spt}}^{\rmG}({\rmS})$ to ${\widetilde {\Spt}}({\rmS})$ is more involved.
 As discussed in ~\eqref{functor.tildeP}, this is carried out by a functor we denote by $\tilde {\mathbb P}$. 
 \vskip .1cm
 Let $\Spt(\rmS)$ denote the (usual) category of spectra indexed by the non-negative integers as in Definition ~\ref{usual.spectra}. Since ${\widetilde {\Spt}}({\rmS})$ is indexed by the 
 category $\Sph_{\rmS}$ which denote the Thom-spaces of all affine spaces $\{{\mathbb A}^n|n \ge 0\}$,
 there is an obvious functor $i^*: {\widetilde {\Spt}}({\rmS}) \ra \Spt(\rmS)$: see  ~\eqref{functor.P}.
 Thus the passage from $\rmG$-equivariant spectra in $\Spt^{\rmG}(\rmS)$ to the non-equivariant spectra $\Spt({\rmS})$ indexed by the non-negative integers 
 is defined by the sequence of functors:
 \be \begin{equation}
    \label{passage.equiv.to.nonequiv}
   \Spt^{\rmG}(\rmS) {\overset {\tilde \rmU} \ra}  {\widetilde {\Spt}}^{\rmG}({\rmS}) {\overset {\tilde {\mathbb P}} \ra} {\widetilde {\Spt}}({\rmS}) {\overset {{\it i}^*} \ra} \Spt({\rmS}).
 \end{equation} \ee
Of these the first two functors $\tilde \rmU$ and $\tilde {\mathbb P}$ are strong monoidal functors (in fact $\tilde \rmU$ is a strict monoidal functor), while the composition $i^* \circ {\mathbb P}$ is the identity,
where ${\mathbb P}: \Spt({\rmS}) \ra {\widetilde {\Spt}}({\rmS})$ is a functor left-adjoint to $i^*$ and which is also strong monoidal. (See Proposition ~\ref{comp.1}.)
Given a commutative ring spectrum $\cE^{\rmG}_{\rmS} \in \Spt^{\rmG}(\rmS) $, we let $\cE_{\rmS}  = i^*({\tilde {\mathbb P}}\tilde \rmU(\cE^{\rmG} )$, which is 
 a commutative ring spectrum in $\Spt(\rmS) $. For example, the equivariant sphere spectrum $\mbS^{\rmG}_{\rmS} $ provides 
  $\mbS_{\rmS} = i^*({\tilde {\mathbb P}}\tilde \rmU(\mbS^{\rmG}_{\rmS} ))$, the usual sphere spectrum.
  \qed
\end{definition}
  \vskip .2cm
{ Of key importance} is the observation that the $\tilde \rmU(\mbS^{\rmG}_{\rmS})$ is the unit 
of the category ${\widetilde {\Spt}}^{\rmG}({\rmS})$ with respect to the smash product in ${\widetilde {\Spt}}^{\rmG}({\rmS})$.
Similarly $\tilde \rmU(\cE^{\rmG})$ is the unit of
${\widetilde {\Spt}}(\rmS, {\cE^{\rmG}})$,  ${\tilde {\mathbb P}}({\tilde \rmU} (\mbS^{\group}_{\rmS}))$ is the unit of the category ${\widetilde {\Spt}}({\rmS})$,
and ${\tilde {\mathbb P}}({\tilde \rmU}(\cE^{\rmG}))$ is the unit of ${\widetilde {\Spt}}(\rmS, {\cE^{\rmG}})$ with respect to the corresponding smash products.  In view of this, we will henceforth denote $\tilde \rmU(\mbS^{\rmG}_{\rmS})$
, ${\tilde {\mathbb P}}({\tilde \rmU}(\mbS^{\rmG}_{\rmS}))$ by $\mbS^{\rmG}_{\rmS}$ and $ \tilde \rmU(\cE^{\rmG})$, $ {\tilde {\mathbb P}}({\tilde \rmU}(\cE^{\rmG}))$ by $\cE^{\rmG}$. 
\begin{proposition}
 \label{comp.2}
 Let $\X \in \Spt^{\rmG}(\rmS) $ and let $\tilde \rmU(\X) \in {\widetilde {\Spt}}^{\rmG}({\rmS}) $ denote the forgetful functor
 $\tilde \rmU$ (as in ~\eqref{tildeU}) applied to $\X$. If $\tilde \alpha: \tilde \X'' \ra \tilde \rmU(\X)$ ($\tilde \beta: \tilde \rmU(\X) \ra \tilde \X'$) is
 a functorial cofibrant (fibrant) replacement in the injective or projective  stable model structure on ${\widetilde {\Spt}}^{\rmG}({\rmS}) $,
  then there exists  $\X'$ and $\X''$ in $\Spt^{\rmG}(\rmS) $, and maps $\alpha: \X'' \ra \X$, $\beta: \X \ra \X'$ in $\Spt^{\rmG}(\rmS) $
  so that $\tilde \rmU(\alpha) = \tilde \alpha$ and $\tilde \rmU(\beta) = \tilde \beta$. 
\end{proposition}
\begin{proof} Recall that the linear algebraic group $\rmG$ acts on a simplicial presheaf section-wise. Therefore, the
  functoriality of the cofibrant and fibrant replacements as in the proof of Proposition ~\ref{functorial.rep.inhert.G.act}, shows that if $\tilde \X'' \ra \tilde \rmU(\X)$ is a cofibrant replacement 
  of the spectrum $\tilde \rmU (\X)$ in ${\widetilde {\Spt}}^{\rmG}({\rmS})$, the following squares commute for all $U$ in the site, are compatible with 
  the restriction along $U' \ra U$ 
  in the site, for all $g \eps \Gamma (U, \rmG)$, all $\T_{\rmW}^{\rmS}$ and $\T_{\rmV}^{\rmS} \eps \Sph^{\rmG}_{\rmS}$:
  
  \be \begin{equation}
 \label{tower.diagm}
 \xymatrix{ {\Gamma(U, \T_{\rmW}^{\rmS}) \wedge \Gamma (U, \tilde \X''(\T_{\rmV}^{\rmS}))} \ar@<1ex>[ddd]^{} \ar@<1ex>[dr]^{g \wedge g} \ar@<1ex>[rrr]^{} &&& {\Gamma (U, \tilde \X''(\T_{\rmW\oplus \rmV}^{\rmS}))} \ar@<1ex>[dl]_g \ar@<1ex>[ddd]^{}\\
            & {\Gamma(U, \T_{\rmW}^{\rmS}) \wedge \Gamma (U, \tilde \X''(\T_{\rmV}^{\rmS}))} \ar@<1ex>[r]^(.4){} \ar@<1ex>[d]^{} & {\Gamma (U, \tilde \X''(\T_{\rmW\oplus \rmV}^{\rmS}))} \ar@<1ex>[d]\\
            & {\Gamma(U, \T_{\rmW}^{\rmS}) \wedge \Gamma (U, \X(\T_{\rmV}^{\rmS}))} \ar@<1ex>[r]^(.4){} & {\Gamma (\rmU, \X(\T_{\rmW\oplus \rmV}^{\rmS}))} \\
            {\Gamma(U, \T_{\rmW}^{\rmS}) \wedge \Gamma (U, \X(\T_{\rmV}^{\rmS}))} \ar@<1ex>[ur]^{g \wedge g} \ar@<1ex>[rrr]^{} &&& {\Gamma (U, \X(\T_{\rmW\oplus \rmV}^{\rmS}))} \ar@<1ex>[ul]_{g} .}
  \end{equation}  \ee     
 It follows that the functorial cofibrant replacement $\tilde \X''$ of $\tilde \rmU(\X)$ in ${\widetilde {\Spt}}^{\rmG}({\rmS})$
  inherits a $\rmG$-action from the $\rmG$-action on
  $\X$, making it belong to $\Spt^{\rmG}(\rmS)$. A corresponding result holds for the functorial fibrant replacement as well.
\end{proof}
\vskip .3cm
\subsubsection{\bf Derived functors of $\wedge$, the internal $\Hom$ and the dual ${\rmD}$ for equivariant spectra}
\label{equiv.nonequiv.spectra.0}
Recall that the functor $\tilde \rmU: \Spt^{\rmG}(\rmS) \ra {\widetilde {\Spt}}^{\rmG}({\rmS})$ 
is a strict monoidal functor.
Let $\rmM , \rmN \in \Spt^{\rmG}(\rmS)$. The fact that
one may find functorial cofibrant and fibrant replacements of objects in ${\widetilde {\Spt}}^{\group}({\rmS})$ shows that one may find
a functorial cofibrant replacement $ \tilde \alpha:\tilde \rmM ''\ra \tilde \rmU (\rmM)$ in ${\widetilde {\Spt}}^{\group}({\rmS})$ and a 
functorial fibrant replacement $\tilde \beta: \tilde \rmU (\rmN) \ra \tilde \rmN'$ in ${\widetilde {\Spt}}^{\group}({\rmS})$. The functoriality of the cofibrant and fibrant replacements, 
  shows as in Proposition ~\ref{comp.2}  that in fact that there exists 
$ \rmM''$, $ \rmN'$ and maps 
$ \alpha:\rmM'' \ra \rmM$, $\beta: \rmN \ra \rmN'$ in  $\Spt^{\rmG}(\rmS)$, with $\tilde \rmU( \alpha) = \tilde \alpha$ and $\tilde \rmU(\beta) = \tilde \beta$. 
Therefore, we define 
\be \begin{equation}
     \label{Dual}
\rmM {\overset L {\wedge}} \rmN =  \rmM'' \wedge \rmN, \quad \RHom(\rmM, \rmN) = \Hom( \rmM'', \rmN '), \quad {\rmD}(\rmM) = \RHom(\rmM'', (\mbS^{\group})')
\end{equation} \ee
\vskip .2cm \noindent
with $\rmM {\overset L{\wedge}} \rmN, \RHom(\rmM, \rmN), D(\rmM)  \in \Spt^{\rmG}(\rmS)$. (In fact, since we choose to work with
the injective model structures, every object is cofibrant and therefore there is no need for any cofibrant replacements.)
Similar conclusions 
will hold when $ \cE^{\rmG} \in \Spt^{\rmG}(\rmS)$ is a commutative ring
spectrum with the corresponding  smash product $\wedge _{{\cE}^{\rmG}_{\rmS}}$ and $\Hom_{\cE^{\rmG}}$ defined in ~\eqref{wedgeE}. (In this case the dual with respect to the ring spectrum $\cE_{\rmS}$ will denoted ${\rmD}_{\cE_{\rmS}}$.)
\be \begin{definition}(The associated homotopy categories).
The homotopy category associated to $\Spt^{\rmG}({\rmS_{\rm mot}})$ will be denoted $\SH^{\rmG}(\rmS_{\rm mot})$, while the homotopy category
  associated to ${\widetilde {\Spt}}^{\group}({\rmS_{\rm mot}})$ will be denoted $\widetilde \SH^{\rmG}(\rmS_{\rm mot})$ and the homotopy category associated to
  ${\widetilde {\Spt}}({\rmS_{\rm mot}})$ will be denoted $ \widetilde \SH(\rmS_{\rm mot})$. Similarly the homotopy category associated to $\Spt({\rmS_{\rm mot}})$ will be 
  denoted $\SH(\rmS_{\rm mot})$, often denoted just $\SH(\rmS)$. The corresponding \'etale variants will be denoted by the subscript {\rm et} in the place of the subscript {\rm mot}.
  \vskip .1cm
  For a commutative ring spectrum $\cE^{\rmG} \eps \Spt^{\rmG}({\rmS_{\rm mot}})$, the homotopy category associated to $\Spt^{\rmG}({\rmS_{\rm mot}}, \cE^{\rmG})$ will be denoted
  $\SH^{\rmG}(\rmS_{\rm mot}, \cE^{\rmG})$  while the homotopy category
  associated to ${\widetilde {\Spt}}^{\group}(\rmS_{\rm mot}, \cE^{\rmG})$ will be denoted $\widetilde \SH^{\rmG}(\rmS_{\rm mot}, \cE^{\rmG})$ and the homotopy category associated to
  ${\widetilde {\Spt}}(\rmS_{\rm mot}, {\tilde {\mathbb P}}\tilde \rmU(\cE^{\rmG}))$ will be denoted $\widetilde \SH^{\rmG}(\rmS_{\rm mot}, {\tilde {\mathbb P}}\tilde \rmU(\cE^{\rmG}))$. Similarly the homotopy category associated to $\Spt({\rmS_{\rm mot}}, \cE_{\rmS})$ will be 
  denoted $\SH(\rmS_{\rm mot}, \cE_{\rmS})$ (or simply $\SH(\rmS_{\rm mot}, \cE_{\rmS})$), where $\cE_{\rmS} = i^*({\tilde {\mathbb P}}\tilde \rmU(\cE^{\rmG} )$ is the associated non-equivariant spectrum.
 
\end{definition} \ee

\vskip .2cm
\section{Model structures for categories of spectra} \index{model structures: spectra}
\label{model.struct.spectra}
One starts with the categories of spectra  $\Spt^{\rmG}(\rmS)$, ${\widetilde {\Spt}}^{\rmG}({\rmS})$ and ${\widetilde {\Spt}}({\rmS})$ considered in 
  Definitions ~\ref{smash.prdcts}, \ref{tildeSptG} and ~\ref{tildeSpt.1}. In case 
$\cE^{\rmG}$ is a commutative ring spectrum in $\Spt^{\rmG}(\rmS)$ (${\widetilde {\Spt}}^{\rmG}({\rmS})$) or if $\cE$ is a commutative ring spectrum in ${\widetilde {\Spt}}({\rmS})$ 
we will also consider the corresponding category $\Spt^{\rmG}(\rmS, {\cE^{\rmG}})$ (${\widetilde {\Spt}}^{\group}({\rmS}, {\cE^{\rmG}})$) and ${\widetilde {\Spt}}({\rmS}, {\cE})$
of module spectra.
\vskip .2cm
\subsection{Level-wise model structures}
Throughout this discussion, we will assume the situation where $ \Spc_*(\rmS)$ denotes the category of pointed simplicial presheaves, 
and $\Spc_*^{\rmG}(\rmS)$ which is the category of pointed simplicial presheaves with $\rmG$-action, both
pointed over $\rmS$ on either the big \'etale or the big Nisnevich or the big Zariski site over a fixed perfect field $k$. 
\vskip .1cm
$ \Spc_*(\rmS)$ will be provided with a 
chosen model structure, namely the motivic model structure in the case of the Nisnevich site (see ~\ref{Nis.presh}) and the \'etale model structure (see  Theorem ~\ref{et.model.struct}) in the case of the \'etale site which
are both based on the projective model structures or the alternate injective model structures discussed in ~\ref{inj.model}. Observe that
every object in $ \Spc_*(\rmS)$ is cofibrant in the injective model structure. 
\vskip .1cm
$\Spc_*^{\rmG}(\rmS)$ will be provided with one of the model structures provided by Theorem ~\ref{model.structs.equiv.presh}. Since the 
generating cofibrations are defined as in Definition ~\ref{equiv.model.struct}, not every object is cofibrant even in the corresponding
 injective model structure on $\Spc_*^{\rmG}(\rmS)$.
 \index{model structures: level-wise injective}
\label{USptG}
\vskip .2cm
\subsubsection{\bf The level-wise injective model structures}
We will start with the injective model structures on $ \Spc_*(\rmS)$.
Here we make use of \cite[Proposition A.3.3.2]{Lur}. The first observation is that the model categories  
$\Spt_*(\rmS_{\rm mot})$ and $\Spt_*(\rmS_{et})$, when provided with the injective model 
structures, are {\it excellent} in the sense of \cite[A.3.2.16]{Lur}: this means they are combinatorial,
every monomorphism is a cofibration, cofibrations are stable under products, the weak-equivalences are stable under filtered colimits, the smash product
$\wedge ^{\rmS}$ is a left  Quillen functor and it satisfies the invertibility hypothesis. (The last may be deduced from the category of
simplicial sets as in \cite[Lemma A.3.2.20]{Lur} by observing that the functor sending a simplicial set to the constant simplicial presheaf
is a left Quillen functor as required.) Therefore, the required model structures follow from \cite[Proposition A.3.3.2]{Lur}, and
the discussion below should be just spelling out the details.
\vskip .1cm
Here we define a map $f:\chi' \ra \chi$ of spectra in ${\widetilde {\Spt}}^{\rmG}({\rmS})$ to be a {\it level-wise injective cofibration} (a {\it level-wise injective weak-equivalence}) if the induced map $f(\rmT_{\rmV}^{\rmS}): \chi'(\rmT_{\rmV}^{\rmS}) \ra \chi(\rmT_{\rmV}^{\rmS})$ is a  cofibration (a weak-equivalence, \res) for each $\rmT_{\rmV}^{\rmS} \in \Sph^{\rmG}_{\rmS}$. The {\it level-wise injective fibrations} are defined by the lifting property with
respect to trivial cofibrations.  One defines the level-wise injective model structure on
the categories $\Spt^{\rmG}(\rmS)$ and ${\widetilde {\Spt}}({\rmS})$ similarly. 
\begin{proposition}
 \label{inj.model.struct}
 \begin{enumerate}[\rm(i)]
\item
This defines a combinatorial (in fact, tractable) simplicial monoidal model structure on ${\widetilde {\Spt}}^{\rmG}({\rmS})$ that is left proper. 
\item
Every level-wise injective fibration is a level fibration, that is if $f:\chi' \ra \chi$ is a fibration in the level-wise injective model structure,
each of the induced maps $\chi'(\rmT_{\rmV}^{\rmS}) \ra \chi(\rmT_{\rmV}^{\rmS})$ is a fibration. 
\item
The cofibrations are the monomorphisms. 
\item
The unit of the monoidal structure on ${\widetilde {\Spt}}^{\rmG}({\rmS})$ and in fact every object in ${\widetilde {\Spt}}^{\rmG}({\rmS})$ is cofibrant in this model structure. 
\item
The corresponding results hold
for the categories $\Spt^{\rmG}(\rmS)$ and ${\widetilde {\Spt}}({\rmS})$. 
\end{enumerate}
\end{proposition}
\begin{proof} We will only discuss the proofs for the category ${\widetilde {\Spt}}^{\rmG}({\rmS})$ since the proofs in the other two cases are quite similar. We start with the observation that the categories $\C=\Spc_*(\rmS)$, $\C^{\rmG}=\Spc_*^{\rmG}(\rmS)$ are simplicially enriched tractable simplicial model categories.
 The left-properness is obvious, since the cofibrations and weak-equivalences are defined level-wise.
 The first conclusion follows now from \cite[Proposition A.3.3.2]{Lur}: observe that the pushout-product axiom holds
since cofibrations (weak-equivalences) are injective cofibrations (weak-equivalences, \res) and the pushout-product
 axiom holds in the monoidal model category $\C$. This proves the first statement for ${\widetilde {\Spt}}^{\rmG}({\rmS})$,
  as well as for ${\widetilde {\Spt}}({\rmS})$.
 \vskip .1cm
 The second statement follows from Proposition ~\ref{key.props.unstable}(iii)
  making use of the adjunction between the functors ${\mathcal Eval}_{\rmT_{\rmV}^{\rmS}}$ and ${\mathcal F}_{\rmT_{\rmV}^{\rmS}}$ discussed below.
The third statement follows readily since we start with the injective model structure on $ \Spc_*(\rmS)$.
  Recall the unit of ${\widetilde {\Spt}}^{\rmG}({\rmS})$ is the functor
$\Sph ^{\group}_{\rmS} \ra \C$, which is the sphere spectrum $\mbS^{\rmG}_{\rmS}$. To prove it is cofibrant, all one has to observe is that 
$\mbS^{\rmG}_{\rmS}(\rmT_{\rmV}^{\rmS})=\rmT_{\rmV}^{\rmS}$ which is cofibrant in $\Spc_*(\rmS)$ for every $\rmT_{\rmV}^{\rmS} \in \Sph^{\group}_{\rmS}$. 
It should be clear that the same arguments hold for the categories $\Spt^{\rmG}(\rmS)$ and ${\widetilde {\Spt}}({\rmS})$.
\end{proof}
\subsubsection{\bf The level-wise projective model structures}\index{model structures: level-wise projective}
We will consider explicitly only ${\widetilde {\Spt}}^{\rmG}({\rmS})$. We will start with the projective model structure on $\Spt_*(\rmS_{\rm mot})$.
First we functorially replace every object $\rmT_{\rmV}^{\rmS}$ in $\Sph^{\rmG}_{\rmS}$ by
  an object that is cofibrant in $\C = \Spc_*(\rmS)$. The functoriality of the cofibrant replacement shows that then, these
  functorial cofibrant replacements all come equipped with $\rmG$-actions. Therefore, we will still denote these
  cofibrant replacements by $\{\rmT_{\rmV}^{\rmS}|\rmV\}$. In the case of $\Spt^{\rmG}$, we need to do the same with the category $\C$ replaced by
  $\C^{\rmG} = \Spc_*^{\rmG}(\rmS)$.
  \vskip .1cm
 We should also point out that the work of \cite[Theorems 4.2, 4.4]{DRO1} in fact provides such a model structure on ${\widetilde {\Spt}}^{\rmG}({\rmS_{\rm mot}})$, ${\widetilde {\Spt}}({\rmS}_{\rm mot})$ and on ${{\Spt}}^{\rmG}({\rmS}_{\rm mot})$ and that they extend readily to the 
 corresponding categories defined on the etale site. Therefore, the discussion below should be viewed
  as summarizing their results in this case. 
 \vskip .1cm
The weak-equivalences (fibrations) in the {\it level-wise projective model structure} are those maps of spectra
$f:\X \ra \Y$, for which each $f(\rmT_{\rmV}^{\rmS}): \X(\rmT_{\rmV}^{\rmS}) \ra \Y(\rmT_{\rmV}^{\rmS})$, $ \rmT_{\rmV}^{\rmS} \in \Sph^{\rmG}_{\rmS}$,
 are weak-equivalences (fibrations, \res) in $\C=\Spc_*(\rmS_{\rm mot})$. The cofibrations in this model structure are defined
 by left-lifting property with respect to the maps that are trivial fibrations in this model structure.
 \vskip .2cm
Next let ${\mathcal F}_{\rmT_{\rmV}^{\rmS}}$ denote the left-adjoint to the evaluation functor ${\mathcal Eval}_{\rmT_{\rmV}^{\rmS}}$ sending a spectrum $\X \in {\widetilde {\Spt}}^{\group}({\rmS}_{\rm mot})$ ($\X \in {\widetilde {\Spt}}({\rmS}_{\rm mot})$, $\Spt^{\rmG}(\rmS_{\rm mot})$) to $\X(\rmT_{\rmV})$. One may observe that this is the spectrum defined by 
\be \begin{align}
 \label{free.sp.1}
 {\mathcal F}_{{\rmT_{\rmV}^{\rmS}}}(C) (\rmT_{\rmV \oplus \rmW}^{\rmS}) &= (\underset {\alpha: \rmV \ra \rmV \oplus \rmW} \sqcup C \wedge \rmT_{\rmW}^{\rmS}) \sqcup *, \rmW \neq \{0\}\\
		 	&= (\underset {\alpha: \rmV \ra \rmV } \bigvee C )\sqcup *, \rmW=\{0\} \notag
\end{align} \ee 
For each ${\rmT}_{\rmV}^{\rmS}$, let ${\mathcal R}_{\rmT_{\rmV}^{\rmS}}$ denote the $\Spc_*(\rmS_{\rm mot})$-enriched functor defined by 
\[{\mathcal R}_{\rmT_{\rmV}^{\rmS}}(\rmP)(\rmT_{\rmW}^{\rmS})= \Hom_{\Spc_*(\rmS)}((\sqcup_{\alpha}\rmT_{\rmU}^{\rmS, \alpha})_+, \rmP)\]
when $\rmV= \rmW\oplus \rmU$, and the sum
 is indexed by homothety classes of $k$-linear injective maps $\rmU \ra \rmV$. (When $\rmV $ is not of the form $\rmU \oplus \rmW$, we 
 let ${\mathcal R}_{\rmT_{\rmV}^{\rmS}}(\rmP)(\rmT_{\rmW}^{\rmS}) = \rmS^0$.)
 
 Then ${\mathcal R}_{\rmT_{\rmV}^{\rmS}}: \Spc_*(\rmS_{\rm mot}) \ra {\widetilde {\Spt}}^{\group}({\rmS}_{\rm mot})$ is right adjoint to the functor ${\mathcal Eval}_{\rmT_{\rmV}^{\rmS}}$. One defines 
 a right adjoint ${\mathcal R}_{\rmT_{\rmV}^{\rmS}}$ to ${\mathcal Eval}_{\rmT_{\rmV}^{\rmS}}: {\widetilde {\Spt}}({\rmS}_{\rm mot}) \ra \Spc_*(\rmS_{\rm mot})$ and to
 ${\mathcal Eval}_{\rmT_{\rmV}^{\rmS}}: {\Spt}^{\group}({\rmS}_{\rm mot}) \ra  \Spc_*^{\rmG}(\rmS_{\rm mot})$ similarly.
 \vskip .2cm
Let $\oI$ ($\oJ$) denote the generating cofibrations (generating trivial cofibrations, \res ) of
the model category $\Spc_*(\rmS_{\rm mot})$. 
We define the generating cofibrations $\oI_{{\widetilde {\Spt}}^{\rmG}({\rmS}_{\rm mot})}$ (the generating trivial cofibrations $\oJ_{{\widetilde {\Spt}}^{\rmG}({\rmS}_{\rm mot})}$) to be 
\be \begin{equation}
\label{gen.cofibs}
\bigcup_{\rmT_{\rmV}^{\rmS} \in \Sph^{\rmG}_{\rmS}} \{\F_{\rmT_{\rmV}^{\rmS}}(i)\mid i \in \oI \} \quad (\bigcup_{\rmT_{\rmV}^{\rmS} \in \Sph^{\rmG}}\{\F_{\rmT_{\rmV}^{\rmS}}(j)|j \in \oJ\}) \, \, .
\end{equation} \ee
\vskip .1cm \noindent 
One defines the generating cofibrations $\oI_{{\widetilde {\Spt}}({\rmS}_{\rm mot})}$ ($\oI_{\Spt^{\rmG}(\rmS_{\rm mot})}$) (the generating trivial cofibrations $\oJ_{{\widetilde {\Spt}}({\rmS}_{\rm mot})}$  ($\oJ_{\Spt^{\rmG}(\rmS_{\rm mot})}$)) of the 
level-wise projective model structure on ${{\widetilde {\Spt}}({\rmS}_{\rm mot})}$ ($\Spt^{\rmG}(\rmS_{\rm mot})$) similarly.
\begin{proposition} 
\label{key.props.unstable}
\begin{enumerate}[\rm(i)]
\item If $A \eps \Spc_*(\rmS_{\rm mot})$ is small relative to the cofibrations (trivial cofibrations) 
 in $\Spc_*(\rmS_{\rm mot})$, then $\F_{T_V^{\rmS}}(A)$ is small relative $\oI_{{\widetilde {\Spt}}^{\rmG}({\rmS}_{\rm mot})}$.
\item A map $f$ in ${{\widetilde {\Spt}}^{\rmG}({\rmS}_{\rm mot})}$  is a level cofibration if and only if it has the left lifting property
 with respect to all maps of the form $\R_{T_V^{\rmS}}(g)$ where $g$ is a trivial fibration (fibration, \res)in $\Spc_*(\rmS_{\rm mot})$.
A map $f$ in ${{\widetilde {\Spt}}^{\rmG}({\rmS}_{\rm mot})}$  is a level trivial cofibration if and only if it has the left lifting property
 with respect to all maps of the form $\R_{T_V^{\rmS}}(g)$ where $g$ is a fibration in $\Spc_*(\rmS_{\rm mot})$.
\item Every map in $\oI_{{\widetilde {\Spt}}^{\rmG}({\rmS}_{\rm mot})}-cof$  is a level cofibration and every map in $\oJ_{{\widetilde {\Spt}}^{\rmG}({\rmS}_{\rm mot})}-cof$ 
is a level trivial cofibration. (Here $\oI_{{\widetilde {\Spt}}^{\rmG}({\rmS}_{\rm mot})}-cof$ ($\oJ_{{\widetilde {\Spt}}^{\rmG}({\rmS}_{\rm mot})}-cof$) denotes the cofibrations generated by $\oI_{{\widetilde {\Spt}}^{\rmG}({\rmS}_{\rm mot})}$ ($\oJ_{{\widetilde {\Spt}}^{\rmG}({\rmS}_{\rm mot})}$, \res).
\item The domains of $\oI_{{\widetilde {\Spt}}^{\rmG}({\rmS}_{\rm mot})}$ ($\oJ_{{\widetilde {\Spt}}^{\rmG}({\rmS}_{\rm mot})}$) are small relative to $\oI_{{\widetilde {\Spt}}^{\rmG}({\rmS}_{\rm mot})}-cell$ ($\oJ_{{\widetilde {\Spt}}^{\rmG}({\rmS}_{\rm mot})}-cell$, \res). 
\item Corresponding results hold for  ${{\widetilde {\Spt}}({\rmS}_{\rm mot})}$ and $\Spt^{\rmG}(\rmS_{\rm mot})$.
\end{enumerate}
\end{proposition}
\begin{proof}
(i) The main point here is that the functor ${\mathcal E}val_{T_V^{\rmS}}$ being right adjoint to $\F_{T_V^{\rmS}}$ commutes 
with all small colimits.
 \vskip .3cm
(ii) Since $\R_{T_V^{\rmS}}$ is right adjoint to ${\mathcal E}val_{T_V^{\rmS}}$,  $f$ has the left lifting property with 
respect to
$\R_{T_V^{\rmS}}(g)$  if and only if ${\mathcal E}val_{T_V^{\rmS}}(f)$  has the left-lifting property with respect to $g$. (ii) follows
readily from this observation.
\vskip .3cm
(iii) Recall every object of $\Spc_*(\rmS_{\rm mot})$ is assumed to be cofibrant in the injective model structure and that in the projective model
structure we first replace every object $\T_{\rmU}^{\rmS} \in \Sph^{\rmG}$ functorially by a cofibrant replacement. Therefore, smashing with  any $\rmT_{\rmU}^{\rmS}$
preserves cofibrations of $\Spc_*(\rmS)$ with either the injective or the projective model structures. Therefore, every map in $\oI_{{\widetilde {\Spt}}^{\rmG}({\rmS}_{\rm mot})}$ is a level cofibration. By (ii) this means
 $\R_{T_V^{\rmS}}(g) \eps \oI_{{\widetilde {\Spt}}^{\rmG}({\rmS}_{\rm mot})}-inj$ for all trivial fibrations $g$ in $\Spc_*(\rmS_{\rm mot})$. Recall every map in $\oI_{{\widetilde {\Spt}}^{\rmG}({\rmS})}-cof$ has
the left lifting property with respect to every map in $\oI_{{\widetilde {\Spt}}^{\rmG}({\rmS}_{\rm mot})}-inj$ and in particular with 
respect to every 
map $\R_{T_V^{\rmS}}(g)$, with $g$ a trivial fibration in $\Spc_*(\rmS_{\rm mot})$. Now the adjunction between ${\mathcal E}val_{T_V^{\rmS}}$ and $\R_{T_V^{\rmS}}$
completes the proof for $\oI_{{\widetilde {\Spt}}^{\rmG}({\rmS}_{\rm mot})}-cof$. The proof for $\oJ_{{\widetilde {\Spt}}^{\rmG}({\rmS}_{\rm mot})}-cof$ and for $\Spt^{\rmG}(\rmS_{\rm mot})$ is similar. 
\vskip .3cm
(iv) follows readily in view of the adjunction between the free functor $\F_{T_V^{\rmS}}$ and ${\mathcal E}val_{T_V^{\rmS}}$. 
\end{proof}
\vskip .3cm

\begin{proposition} 
\label{level.model.1}
 The projective cofibrations, the level fibrations and level equivalences define a cofibrantly 
generated model category structure on ${\widetilde {\Spt}}^{\rmG}({\rmS}_{\rm mot})$ with the generating cofibrations (generating trivial
cofibrations) being $\oI_{{\widetilde {\Spt}}^{\rmG}({\rmS}_{\rm mot})}$ ($\oJ_{{\widetilde {\Spt}}^{\rmG}({\rmS}_{\rm mot})}$, \res). This model structure (called {\it the 
level-wise projective model structure})
  has the following properties:
\be \begin{enumerate}[\rm (i)]
\item Every projective cofibration (projective trivial cofibration)  is a level cofibration (level trivial cofibration, \res). 
\item{It is left-proper, right proper and is cellular.}
\item The objects in 
$\bigcup_{\rmT_{\rmV} \in \Sph^{\rmG}}\{\F_{\rmT_{\rmV}}(\Sph^{\rmG})\}$ are all finitely presented. The category ${\widetilde {\Spt}}^{\rmG}({\rmS}_{\rm mot})$
is symmetric monoidal with the pairing defined in  Definition ~\ref{tildeSptG}(ii).
\item {This category is locally presentable and hence  is a tractable (and hence a combinatorial)
model category.}
\item{With the above structure, ${\widetilde {\Spt}}^{\rmG}({\rmS}_{\rm mot})$ is a symmetric monoidal model category satisfying the monoidal axiom.}
\item Corresponding results hold for the level-wise projective model structure on ${{\widetilde {\Spt}}({\rmS}_{\rm mot})}$ and $\Spt^{\rmG}(\rmS_{\rm mot})$ as well
as on the corresponding categories defined on the \'etale sites.
\end{enumerate} \ee
\end{proposition}
\begin{proof} 
First we sketch a proof showing the existence of a cofibrantly generated model category structure.
The retract and two out of three axioms for level equivalences are 
immediate, as is the lifting axiom
for a projective cofibration and a level trivial fibration. Clearly a map is a level trivial
fibration if and only if it is in $\oI_{{\widetilde {\Spt}}^{\rmG}({\rmS}_{\rm mot})}-inj$ and a map is a projective cofibration if and only if it is in
$\oI_{{\widetilde {\Spt}}^{\rmG}({\rmS}_{\rm mot})}-cof$. Now Proposition ~\ref{key.props.unstable}(iv) shows that \cite[Theorem 2.1.14]{Hov01} applied to $\oI_{{\widetilde {\Spt}}^{\rmG}({\rmS}_{\rm mot})}$ then produces
a functorial factorization of a map as the composition of a projective cofibration followed by a level trivial fibration.
\vskip .2cm
By adjunction, a map is a level fibration if and only if it is in $\oJ_{{\widetilde {\Spt}}^{\rmG}({\rmS}_{\rm mot})}-inj$. Proposition ~\ref{key.props.unstable}(iii) shows that
every map in $\oJ_{{\widetilde {\Spt}}^{\rmG}({\rmS}_{\rm mot})}-cof$ is a level equivalence. Such maps have left-lifting property with respect to all level   fibrations and hence
with respect to all level trivial fibrations. Now Proposition ~\ref{key.props.unstable}(iv) shows that \cite[Theorem 2.1.14]{Hov01} applied to 
$\oJ_{{\widetilde {\Spt}}^{\rmG}({\rmS}_{\rm mot})}$ then produces
a functorial factorization of a map as the composition of a projective  cofibration which is also a level equivalence 
followed by a level fibration. 
\vskip .2cm
Next we show that any projective cofibration and level equivalence $f$ is
in $\oJ_{{\widetilde {\Spt}}^{\rmG}({\rmS}_{\rm mot})}-cof$, and hence has the left lifting property with respect to level fibrations.
To see this, we factor $ f = pi$ where $i$ is in $\oJ_{{\widetilde {\Spt}}^{\rmG}({\rmS}_{\rm mot})}-cof$ and $p$ is in $\oJ_{{\widetilde {\Spt}}^{\rmG}({\rmS}_{\rm mot})}-inj$. Then $p$ is a
level fibration. Since $f$ and $i$ are both level equivalences, so is $p$. Therefore $p$ is a level trivial fibration and $f$ has the
left lifting property with respect to $p$. This shows $f$ is a retract of $i$: see, for example, \cite[Lemma 1.1.9]{Hov99}.  In particular 
$f$ belongs to  $\oJ_{{\widetilde {\Spt}}^{\rmG}({\rmS}_{\rm mot})}-cof$. These prove the existence of the projective model structure on ${\widetilde {\Spt}}^{\rmG}({\rmS}_{\rm mot})$.
Clearly it is cofibrantly generated.
\vskip .2cm
Statement (i) is essentially Proposition ~\ref{key.props.unstable}(iii). Since colimits and limits in ${\widetilde {\Spt}}^{\rmG}({\rmS}_{\rm mot})$ are taken level-wise, the statements in (ii) are clear.
The first assertion in (iii) is clear since the objects in the subcategory $\Sph^{\rmG}$ are assumed to be finitely presented in $\C$. 
The assertions in (iii) on the monoidal structure follow from 
a theorem of Day: see \cite{Day}. 
 Statements (iv) and (v) follows from \cite[Theorems 4.2, 4.4]{DRO1}.
\end{proof}
\vskip .2cm
\subsubsection{Module spectra over a ring spectrum}
\label{modsp.ringsp}
Let $\cE^{\rmG} \in { {\Spt}}^{\rmG}({\rmS})$ denote a ring spectrum, and let $\tilde \rmU(\cE^{\rmG}) \in {\widetilde {\Spt}}^{\rmG}({\rmS})$ 
($\cE =\tilde {\mathbb P}(\tilde \rmU(\cE^{\rmG})) \in  {{\widetilde {\Spt}}({\rmS})}$) denote the associated ring spectra.
One then invokes the free $\cE^{\rmG}$-module
	functor and the forgetful functor sending an $\cE^{\rmG}$-module spectrum to
	its underlying spectrum along with \cite[Lemma 2.3, Theorem 4.1(2)]{SSch} to obtain
	a corresponding cofibrantly generated model category structure on ${{\Spt}}^{\rmG}({\rmS}, {\cE^{\rmG}})$   (${\widetilde {\Spt}}^{\rmG}({\rmS}, \tilde \rmU({\cE^{\rmG}}))$, 
	${\widetilde {\Spt}}({\rmS}, {\cE})$, \res). Observe that in this model structure the fibrations are those maps $f$ in ${\Spt}^{\rmG}({\rmS}, {\cE^{\rmG}})$  for which $f$
		is a fibration in $ {\Spt}^{\rmG}({\rmS})$ and similarly for the other two model categories considered here.

\subsection{\bf The stable model structures on ${\widetilde {\Spt}}^{\rmG}({\rmS})$, ${\widetilde {\Spt}}^{\group}({\rmS}, \tilde \rmU({\cE^{\group}}))$, $\Spt^{\rmG}(\rmS)$, $\Spt^{\rmG}(\rmS, {\cE^{\rmG}})$ \\ and on ${\widetilde {\Spt}}({\rmS})$, ${\widetilde {\Spt}}({\rmS}, {\cE})$.}
\label{stable.model}
We proceed to define the stable model structure by applying a suitable Bousfield
localization to the level-wise injective (projective)model structures considered above. This follows the approach in
\cite[section 3]{Hov01}. We will explicitly consider only the case of ${\widetilde {\Spt}}^{\group}({\rmS})$, since essentially the same description applies to the categories 
${\widetilde {\Spt}}({\rmS})$, ${\widetilde {\Spt}}({\rmS}, {\cE})$, ${\widetilde {\Spt}}^{\group}({\rmS}, \tilde \rmU({\cE^{\group}}))$, $\Spt^{\rmG}(\rmS)$ and $\Spt^{\rmG}(\rmS, {\cE^{\rmG}})$, with the only difference that while considering the last two categories 
$\Spt^{\rmG}(\rmS)$ and $\Spt^{\rmG}(\rmS, {\cE^{\rmG}})$, 
any reference
to the category $\Spc_*(\rmS)$ will have to be replaced by the category $\Spc_*(\rmS)^{\rmG}$. The corresponding model structure will be called the {\it the injective (projective) stable
 model structure}. (One may observe that the domains and co-domains of objects of the generating cofibrations are 
cofibrant, so that there is no need for a cofibrant replacement functor $\rmQ$ as in \cite[section 3]{Hov01}.)
\index{model structures: stable}
\vskip .1cm 
Let $\X \in {\widetilde {\Spt}}^{\rmG}({\rmS})$. Since $\X$ is a $\Spc_*(\rmS)$-enriched functor $\Sph^{\rmG}_{\rmS} \ra \Spc_*(\rmS)$,
we obtain a natural map 
\be \begin{equation}
\label{Omega.0}
 (\sqcup_{\alpha} \rmT_{\rmW}^{\rmS,\alpha})_+ = \Hom_{\Sph^{\rmG}}(\rmT_{\rmV}^{\rmS}, \rmT_{\rmV}^{\rmS}\wedge \rmT_{\rmW}^{\rmS}) \ra \Hom_{\Spc_*(\rmS)}(\X(\rmT_{\rmV}^{\rmS}), \X(\rmT_{\rmV\oplus \rmW}^{\rmS}))),
\end{equation} \ee
where $\rmT_{\rmW}^{\rmS, \alpha}$ is a copy of $\rmT_{\rmW}^{\rmS}$ indexed by $\alpha$, and where $\alpha$ varies over all homothety classes of $k$-linear injective and $\rmG$-equivariant maps $\rmV\ra \rmV\oplus \rmW$.
\begin{definition} ($\Omega$-spectra)
 A spectrum ${ \chi} \in {\widetilde {\Spt}}^{\group}({\rmS})$  is an  {\it $\Omega$-spectrum} if it is level-wise fibrant and 
each of the natural maps
${ \chi}(\rmT_{\rmV}^{\rmS}) \ra \Hom_{\C}(\rmT_{\rmW}^{\rmS, \alpha}, { \chi}(\rmT_{\rmV}^{\rmS} \wedge \rmT_{\rmW}^{\rmS}))$, for each $\alpha$ as in ~\eqref{Omega.0} is an unstable weak-equivalence 
in the corresponding model structure on $\Spc_*(\rmS)$.
\end{definition}
\vskip .3cm
Let ${\mathcal F}_{\rmT_{\rmV}^{\rmS}}$ denote the left-adjoint to the evaluation functor sending a spectrum $\X \in {\widetilde {\Spt}}^{\group}({\rmS})$  to $\X(\rmT_{\rmV}^{\rmS})$: see ~\eqref{free.sp.1}.
Let $\rmC \in \Spc_*(\rmS)$ be an object that is cofibrant, and let $\chi \in {\widetilde {\Spt}}^{\group}({\rmS})$ be fibrant in the level-wise injective (projective) model structure. Then 
\[Map(\rmC, \chi(\rmT_{\rmV}^{\rmS})) =Map(\rmC, {\mathcal E}val_{\rmT_{\rmV}^{\rmS}}(\chi)) \simeq Map(\F_{\rmT_{\rmV}^{\rmS}}(\rmC), \chi) \mbox{ and }\]
\[ Map(\rmC, \Hom_{\C}(\rmT_{\rmW^{\rmS, \alpha}}, \chi(\rmT_{\rmV}^{\rmS} \wedge \rmT_{\rmW}^{\rmS}))) = Map(\rmC, \Hom_{\C} (\rmT_{\rmW^{\rmS, \alpha}}, {\mathcal E}val_{\rmT_{\rmV}^{\rmS}\wedge \rmT_{\rmW}^{\rmS}}(\chi))) \simeq Map(\F_{\rmT_{\rmV}^{\rmS} \wedge \rmT_{\rmW}^{\rmS}}(\rmC \wedge \rmT_{\rmW^{\rmS, \alpha}} ), \chi).\]
Therefore,  to convert
$\chi$ into an $\Omega$ -spectrum, it suffices to invert the maps in ${\mathbf S}$, where
\be \begin{equation}
\label{S.1}     
{\mathbf S} = \{{\mathcal F}_{\rmT_{\rmV}^{\rmS} \wedge \rmT_{\rmW}^{\rmS}}(\rmC \wedge \rmT_{\rmW^{\rmS, \alpha}} ) \ra {\mathcal F}_{\rmT_{\rmV}^{\rmS}}(\rmC) \mid \rmC \in
\mbox{ Domains or Co-domains of } \oI, \rmT_{\rmV}, \rmT_{\rmW} \in \Sph^{\group}, \alpha \}
\end{equation} \ee
\vskip .3cm \noindent
corresponding to the above  maps $\rmC \wedge \rmT_{\rmW^{\rmS, \alpha}}  \ra \rmC \wedge \Hom_{\Sph^{\rmG}}(\rmT_{\rmV}^{\rmS}, \rmT_{\rmV}^{\rmS} \wedge \rmT_{\rmW}^{\rmS})$ by adjunction, as $\alpha$ varies over all homothety classes of $k$-linear injective $\rmG$-equivariant maps $\rmV \ra \rmV\oplus \rmW$. (Here ${\rm I}$ denotes the generating cofibrations of $\Spc_*(\rmS)$.)
Similarly, for a commutative ring spectrum $\cE^{\rmG} \in {\widetilde {\Spt}}^{\group}({\rmS})$, one lets ${\mathbf S}_{\cE^{\rmG}}$ be defined
using the corresponding free-functors for $\cE^{\rmG}$-module-spectra. (See \cite[Proposition 3.2]{Hov01} that shows it suffices to consider
the objects $\rmC$ that form the domains and co-domains of the generating cofibrations in $\Spc_*(\rmS)$.)

\vskip .2cm
The {\it stable injective (projective) model structure} on ${\widetilde {\Spt}}^{\group}({\rmS})$ (${\widetilde {\Spt}}^{\group}({\rmS}, \tilde \rmU({\cE^{\rmG}}))$) is obtained by localizing the level-wise injective (projective) model
structure with respect to the maps in ${\mathbf S}$ (${\mathbf S}_{\cE^{\rmG}}$, \res). 
The ${\mathbf S}$-local weak-equivalences
(${\mathbf S}$-local fibrations) will be referred to as the {\it stable equivalences} 
({\it stable fibrations}, \res). The cofibrations in the localized model structure are the 
cofibrations in the level-wise projective or injective model structures on ${\widetilde {\Spt}}^{\group}({\rmS})$ (${\widetilde {\Spt}}^{\group}({\rmS}, \tilde \rmU({\cE^{\rmG}}))$, \res).

\begin{proposition} 
\label{stable.model.1}
\begin{enumerate}[\rm(i)]
\item The corresponding stable model structure on ${\widetilde {\Spt}}^{\group}({\rmS})$ (${\widetilde {\Spt}}^{\group}({\rmS}, {\cE^{\rmG}})$) is cofibrantly 
generated and left proper. It is also
locally presentable, and hence combinatorial (tractable).
\item The fibrant objects in the stable model structure on ${\widetilde {\Spt}}^{\group}({\rmS})$ (${\widetilde {\Spt}}^{\group}({\rmS}, {\cE^{\rmG}})$) are the
 $\Omega$-spectra defined above. 
\item The category ${\widetilde {\Spt}}^{\group}({\rmS})$ (${\widetilde {\Spt}}^{\group}({\rmS}, {\cE^{\rmG}})$) is a symmetric monoidal model category 
(i.e., satisfies the pushout-product axiom: see \cite[Definition  3.1]{SSch}) in  the 
injective stable model structures with
the monoidal structure being the  same in both the
 model structures. In the injective model structure, the unit is cofibrant and the monoidal axiom (see \cite[Definition 3.3]{SSch})
is also satisfied.
\item The first two statements also hold for the categories ${{\Spt}}^{\group}({\rmS})$, ${{\Spt}}^{\group}({\rmS}, {\cE^{\rmG}})$, while
all three statements hold also for 
${\widetilde {\Spt}}({\rmS})$ and ${\widetilde {\Spt}}({\rmS}, {\cE})$.
\end{enumerate}
\end{proposition}
\begin{proof} The proof of the first statement in (i) is entirely similar to the proof of \cite[Theorem 3.4]{Hov01} and is 
therefore skipped. (Since $\Spc_*(\rmS)$ is left proper (cellular), so is the projective model structure on ${\widetilde {\Spt}}^{\group}({\rmS})$ (which is
a category of $\Spc_*(\rmS)$-enriched functors $\Sph^{\rmG}_{\rmS} \ra \Spc_*(\rmS)$ as proved above. It is shown in \cite[Proposition 3.4.4 and Theorem 4.1.1]{Hirsch} that then the localization of the projective model structure
is also left proper and cellular.) The fact it is locally presentable and hence combinatorial follows from the corresponding property of the
unstable model categories. Our hypotheses show that the domains of the maps in $\rmI$ and $\rmJ$ are cofibrant. These prove all
the statements in (i). (ii) is clear. 
\vskip .2cm
Clearly, since the monoidal structure is the same as in the unstable setting, the above category of spectra is clearly symmetric monoidal. Now to
prove it is a monoidal model category, it suffices to prove that the pushout-product axiom holds. 
Moreover, since an enriched functor 
$\X:\Sph^{\rmG}_{\rmS} \ra \Spc_*(\rmS)$ is cofibrant if ${\X}(\rmT_{\rmV}^{\rmS})$ is cofibrant in $\Spc_*(\rmS)$ for every $\rmT_{\rmV}^{\rmS} \eps \Sph^{\rmG}$ and every object of $\Spc_*(\rmS)$ is assumed to be 
cofibrant, it follows that
every such functor is cofibrant in the injective stable model structure.
Thus every object is cofibrant in the injective model structure, and therefore the pushout-product axiom and the monoidal axiom
are satisfied. It is clear that the unit is cofibrant in the injective model structure.   This proves (iii).

 \end{proof}   
\section{Comparison of several model categories of spectra} \index{model structures: comparison}
\label{comparison.cat.Spt}
 \vskip .1cm
 We adopt the framework of  ~\ref{Nis.presh} and ~\ref{et.presh}.
 Let $\Spt(\rmS_{\rm mot})$ denote the (usual) category of motivic spectra defined as in Definition ~\ref{usual.spectra}.  Recall  its objects are 
 \[\X=\{\rmX_n \in \Spc_*(\rmS_{\rm mot}), \mbox{ along with structure maps } \T^{\wedge m} \wedge \rmX_n \ra \rmX_{n+m} |n, m \in {\mathbb N}\}.\]  Morphisms between two such objects  $\X$ and $\Y$ are defined as compatible collection of maps
  $\X_n \ra \Y_n$, $n \in {\mathbb N}$ compatible with suspensions by $\T^{\wedge m}$, $m \in {\mathbb N}$. 
  $\Spt(\rmS_{et})$ will denote the corresponding category of $\T$-spectra defined on the big \'etale site of $\Speck$.
  \vskip .2cm
  We proceed to relate the
  category $\Spt(\rmS_{\rm mot})$ with ${\widetilde {\Spt}}({\rmS}_{\rm mot})$ and the category $\Spt(\rmS_{et})$ with 
  ${\widetilde {\Spt}}({\rmS}_{et})$, which were defined in Definition ~\ref{tildeSpt.1}. In order to handle
  both at the same time, we will let $\Spt(\rmS)$ denote either $\Spt(\rmS_{\rm mot})$ or $\Spt(\rmS_{et})$ and ${\widetilde {\Spt}}({\rmS})$ will denote 
  either ${\widetilde {\Spt}}({\rmS}_{\rm mot})$ or ${\widetilde {\Spt}}({\rmS}_{et})$.
  \vskip .2cm
  
  We will now identify ${\mathbb N}$ with
 the $\Spc_*(\rmS)$-enriched subcategory of $\Sph_{\rmS}$ consisting of the objects 
 $\{\T^{\wedge n}_{\rmS}|n \ge 0\}$
 and where
 \be \begin{align}
  \label{Spt}
      \Hom_{{\mathbb N}}(\T^{\wedge n}_{\rmS}, \T^{\wedge n+m}_{\rmS}) &= \T^{\wedge m}_{\rmS}, \mbox {if } m > 0\\
						&= \rmS^0, \mbox{if } m=0. \notag
 \end{align} \ee
 Next we define a faithful functor of $\Spc_*(\rmS)$-enriched categories 
 $i: {\mathbb N} \ra \Sph_{\rmS}$ as follows. We send each object 
  $\T^{\wedge n}_{\rmS}$ to itself. We send 
  the $\T^{\wedge m}_{\rmS}$ on the right hand side of ~\eqref{Spt} to the $\T^{\wedge m}_{\rmS}$
on the right-hand-side of ~\eqref{enrich.hom} indexed by the imbedding of ${\mathbb A}^n$ in ${\mathbb A}^{n+m}$ as the
first $n$-factors when $m>0$, while we also send the $\rmS^0$ on the right-hand-side of ~\eqref{enrich.hom}
to the summand $\rmS^0$ indexed by the identity map ${\mathbb A}^n \ra {\mathbb A}^n$ appearing 
in the right-hand-side of ~\eqref{enrich.hom}, when $m=0$.
Thus, we obtain a $\Spc_*(\rmS)$-enriched {\it faithful functor} $i: {\mathbb N} \ra \Sph_{\rmS}$. 
 The functor $i^*$ defines a simplicially enriched functor ${\widetilde {\Spt}}({\rmS}) \ra \Spt(\rmS)$. The functor $i^*$ admits a 
 left adjoint, which we denote by 
 \be \begin{equation}
 \label{functor.P}
 {\mathbb P}: \Spt(\rmS) \ra {\widetilde {\Spt}}({\rmS}).
 \end{equation} \ee
  One defines both a projective, as well as an injective model
  structure on the category $\Spt(\rmS)$, both level-wise and stably: this may be done just as in the last section and therefore we skip the details. Though for the most part we
  will only work with the injective model structures, the projective model structures seem helpful for
  comparing the model categories $\Spt(\rmS)$ and ${\widetilde {\Spt}}({\rmS})$. 
  
  \vskip .1cm
  The free functor $\Spc_*(\rmS) \ra \Spt(\rmS) $ left adjoint to the evaluation functor $Eval_{\T^{\wedge n}_{\rmS}}: \Spt(\rmS) \ra \Spc_*(\rmS)$, sending 
  $\X \mapsto \X(\T^{\wedge n}_{\rmS})$ will be denoted $F_{\T^{\wedge n}_{\rmS}}$. Let $\oI$ denote the set of generating cofibrations of the model category $\Spc_*(\rmS)$ provided 
  with the projective model structure.
  The stable model structure on $\Spt(\rmS)$ will
  be obtained by inverting maps in 
  \be \begin{equation}
\label{S.2}     
{\mathbf S}_{\mathbb N} = \{{ F}_{\T^{\wedge n}_{\rmS} \wedge \T^{\wedge m}_{\rmS}}(\rmC \wedge \T^{\wedge m}_{\rmS} ) \ra { F}_{\T^{\wedge n}_{\rmS}}(\rmC) \mid \rmC \in
\mbox{ Domains or Co-domains of } \oI, m, n \in {\mathbb N} \}.
\end{equation} \ee
\vskip .1cm
  The free functor $\Spc_*(\rmS) \ra {\widetilde {\Spt}}(\rmS) $ left adjoint to the evaluation functor $Eval_{\T^{\wedge n}_{\rmS}}: {\widetilde {\Spt}}(\rmS) \ra \Spc_*(\rmS)$, sending 
  $\X \mapsto \X(\T^{\wedge n}_{\rmS})$ will be denoted ${\mathcal  F}_{\T^{\wedge n}_{\rmS}}$. Let $\oI$ again denote the set of generating cofibrations of the model category $\Spc_*(\rmS)$ provided 
  with the projective model structure.
  The stable model structure on ${\widetilde {\Spt}}(\rmS)$ will
  be obtained by inverting maps in 
  \be \begin{equation}
\label{S.3}     
\tilde {\mathbf S}_{\mathbb N} = \{{\mathcal F}_{\T^{\wedge n}_{\rmS} \wedge \T^{\wedge m}_{\rmS}}(\rmC \wedge \T^{\wedge m}_{\rmS} ) \ra {\mathcal F}_{\T^{\wedge n}_{\rmS}}(\rmC) \mid \rmC \in
\mbox{ Domains or Co-domains of } \oI, m, n \in {\mathbb N} \}.
\end{equation} \ee

\vskip .1cm \noindent  
 We will provide both ${\widetilde {\Spt}}({\rmS})$ and $\Spt(\rmS)$ with
  the projective level-wise and the corresponding projective stable model structures. 
     \vskip .1cm
  
  \begin{remark}
    Observe that a  key difference between the two categories $\Sph_{\rmS}$ and ${\mathbb N}$ is that the simplicially
    enriched hom in the category $\Sph_{\rmS}$ has many symmetries making $\Sph_{\rmS}$ a symmetric monoidal category and  much bigger than the 
    indexing category ${\mathbb N}$ for $\Spt(\rmS)$. 
    Nevertheless we proceed to show that at the homotopy category level, the categories ${\widetilde {\Spt}}({\rmS})$ and $\Spt(\rmS)$ are
    equivalent. (This should be viewed as the analogue of the equivalence between the homotopy categories of orthogonal spectra and spectra in the topological setting.)
  \end{remark}
  \vskip .2cm
  Next we proceed to relate the categories ${\widetilde {\Spt}}^{\rmG}(\rmS)$ and ${\widetilde {\Spt}}(\rmS)$.
  At this point the reader may want to consult section ~\ref{equiv.sph.sp.0}. One may then recall that the indexing category for
  the category ${\widetilde {\Spt}}^{\rmG}(\rmS)$ is ${\rm Sph}^{\rmG}_{\rmS}$, which is the $\Spc_*^{\rmG}(\rmS)$-enriched category whose objects are the Thom-spaces
  $\{{\rm T}_{\rmV}^{\rmS}|\rmV\}$ of all finite dimensional representations of the given linear algebraic group $\rmG$, while the indexing category for the category 
  ${\widetilde {\Spt}}({\rmS})$ is $\{\T^{\wedge n}_{\rmS}|n \ge 0\}$, and where 
  the $\Spc_*(\rmS)$-enriched hom between $\T^{\wedge n}_{\rmS} $ and $\T^{\wedge n+m}_{\rmS}$ is defined in Definition ~\ref{enrich.hom}.
\vskip .1cm
  To relate the $\Spc_*(\rmS)$-enriched categories, ${\widetilde {\Spt}}({\rmS})$ and ${\widetilde {\Spt}}^{\rmG}({\rmS})$, one first observes that
  there is a forgetful functor $j: \Sph^{\rmG}_{\rmS} \ra \Sph_{\rmS}$ that sends the Thom-space, $\rmT_{\rmV}^{\rmS}$,
  of a $\rmG$-representation $\rmV$ to $\rmT_{\rmV}^{\rmS}$ but viewing $\rmV$ as just a $k$-vector space, forgetting the $\rmG$-action. Therefore, pull-back by $j$ 
  defines the $\Spc_*(\rmS)$-enriched functor $j^*:{\widetilde {\Spt}}({\rmS}) \ra {\widetilde {\Spt}}^{\rmG}({\rmS}) $. One defines a functor 
  \be \begin{equation}
  \label{functor.tildeP}
  \tilde {\mathbb P}: {\widetilde {\Spt}}^{\rmG}({\rmS}) \ra {\widetilde {\Spt}}({\rmS})
  \end{equation} \ee
  as the left-adjoint to $j^*$. For a $\X \in {\widetilde {\Spt}}^{\rmG}({\rmS})$, $\tilde {\mathbb P}(\X)$ is  defined
  as a $\Spc_*(\rmS)$-enriched left Kan-extension along the functor $j: \Sph^{\rmG}_{\rmS} \ra \Sph_{\rmS}$. Moreover,  the stable projective model structure on
  ${\widetilde {\Spt}}^{\rmG}({\rmS})$ is obtained from the level-wise projective model structure by inverting maps in
  the  collection ${\mathbf S}$ defined  in ~\eqref{S.1}.
\vskip .2cm  
 We proceed to show that the $\Spc_*(\rmS)$-enriched stable model categories ${\widetilde {\Spt}}^{\rmG}({\rmS})$, ${\widetilde {\Spt}}({\rmS})$ and
  $\Spt(\rmS)$ are Quillen equivalent. The proof will compare both ${\widetilde {\Spt}}^{\rmG}({\rmS})$ and $\Spt(\rmS)$ with ${\widetilde {\Spt}}({\rmS})$.
  Since both these comparisons proceed similarly, we deal with them both in the following proposition.  
\begin{proposition}\footnote{We skip the proof that the injective and projective stable model structures appearing below are Quillen equivalent, which may be proven in the usual manner.}
 \label{comp.1}
 \begin{enumerate}[\rm(i)]
 \item The functors ${\mathbb P}$ and $i^*$ define a Quillen adjunction between the projective stable model structures on
 ${\widetilde {\Spt}}({\rmS})$ and $\Spt(\rmS)$. This is, in fact, a Quillen equivalence.
 
\item The functors $\tilde {\mathbb P}$ and $j^*$ define a Quillen-equivalence between the stable projective model structures 
 on ${\widetilde {\Spt}}({\rmS})$ and ${\widetilde {\Spt}}^{\rmG}({\rmS})$.
 \item The functors ${\mathbb P}$ and $\tilde {\mathbb P}$ are strong monoidal functors.
 \end{enumerate}
\end{proposition}
\begin{proof} It should be clear that $i^*$ ($j^*$) preserves fibrations and weak-equivalences in the level-wise
 projective model structures. Therefore, its left adjoint ${\mathbb P}$ (${\tilde {\mathbb P}}$) preserves the cofibrations and trivial
 cofibrations in the level-wise projective model structures. It is also clear that $i^*$ sends $\Omega$-spectra
 in ${\widetilde {\Spt}}({\rmS})$ to $\Omega$-spectra in $\Spt(\rmS)$ and that $j^*$ sends $\Omega$-spectra in ${\widetilde {\Spt}}({\rmS})$ to
 $\Omega$-spectra in ${\widetilde {\Spt}}^{\rmG}({\rmS})$. Therefore, the functors ${\mathbb P}$ and ${\tilde {\mathbb P}}$ preserve stable weak-equivalences
 between cofibrant objects.
 \vskip .2cm
 Next we consider the generating trivial cofibrations for the projective stable model structure on ${\widetilde {\Spt}}^{\rmG}({\rmS})$.
 For this one starts with the generating trivial cofibrations for the level-wise projective model structure defined in 
 ~\eqref{gen.cofibs}. Then one replaces  each of the maps in the set  ${\mathbf S}$ defined in ~\eqref{S.1}) by the corresponding simplicial 
 mapping cylinder and adds these maps to the generating trivial cofibrations for the level-wise projective model structure.
 One may denote the resulting set by $\tilde \rmJ^{\rmG}$. Finally one takes the pushout-products of the maps in $\tilde \rmJ^{\rmG}$
 with  $\delta \Delta[n]_+ \ra \Delta [n]_+$, $n \ge 0$. This will be the set of generating trivial cofibrations for the 
 stable projective model structure on ${\widetilde {\Spt}}^{\rmG}({\rmS})$.
 \vskip .1cm
 By replacing the set ${\mathbf S}$ with $\tilde {\mathbf S}_{{\mathbb N}}$ defined in ~\eqref{S.3} (${\mathbf S}_{{\mathbb N}}$ defined in ~\eqref{S.2}) and the 
 generating trivial cofibrations for the level-wise projective model structure defined in 
 ~\eqref{gen.cofibs} with the corresponding generating trivial cofibrations for the level-wise projective model structure 
 for the category ${\widetilde {\Spt}}({\rmS})$ ($\Spt(\rmS)$, \res), one obtains the generating trivial cofibrations for the category ${\widetilde {\Spt}}({\rmS})$
 ($\Spt(\rmS)$, \res).
 \vskip .2cm
 Next, the adjunction between the free functors and the evaluation functors provides the identification:
 \be \begin{equation}
 \label{free.ident}
 {\mathbb P}(F_{\T^{\wedge n}_{\rmS}}) = {\mathcal F}_{i(\T^{\wedge n}_{\rmS})} \mbox{ and } \tilde {\mathbb P}({\mathcal F}_{\rmT_{\rmV}^{\rmS}}) = {\mathcal F}_{j(\rmT_{\rmV}^{\rmS})}.
 \end{equation} \ee
 (This follows readily from the identifications
 $Eval_{\T^{\wedge n}_{\rmS}}(i^*(\X)) = Eval_{i(\T^{\wedge n}_{\rmS})}(\X)$ and $Eval_{\rmT_{\rmV}^{\rmS}}(j^*(\X)) = Eval_{j(\rmT_{\rmV}^{\rmS})}(\X)$.) Therefore, it follows that ${\mathbb P}$ 
 (${\tilde {\mathbb P}}$) sends the 
 generating trivial cofibrations in the projective stable model structure on $\Spt(\rmS)$ (${\widetilde \Spt}^{\rmG}({\rmS})$) to the generating trivial
 cofibrations in the projective stable model structure on ${\widetilde \Spt}({\rmS})$ ( ${\widetilde \Spt}({\rmS})$, \res). Since the functor ${\mathbb P}$ (${\tilde {\mathbb P}}$) also preserves
 pushouts and filtered colimits, it follows that it preserves trivial cofibrations in the stable projective model structure. Since the cofibrations 
 in the projective stable model structure are the same as in the projective level-wise model structure, it follows
 that ${\mathbb P}$ (${\tilde {\mathbb P}}$) also preserves these, thereby proving that the functors $({\mathbb P}, i^*)$ (${\tilde {\mathbb P}}, j^*)$) 
 define a
 Quillen adjunction of the projective stable model structures on $\Spt(\rmS)$ and ${\widetilde {\Spt}}({\rmS})$ (${\widetilde {\Spt}}^{\rmG}({\rmS})$ and ${\widetilde {\Spt}}({\rmS})$, \res). 
\vskip .1cm
Next observe that the functor $i^*$ ($j^*$) being a right Quillen functor preserves trivial fibrations in the stable model structure, and therefore, 
(by Ken Brown's lemma: see \cite[Lemma 1.1.12]{Hov99}), it preserves all stable weak-equivalences between stably fibrant objects.
 In fact a stable weak-equivalence between stably fibrant objects is a level-wise weak-equivalence and $i^*$ ($j^*$) clearly preserves
  these. Next we already saw that $i^*$ ($j^*$) preserves $\Omega$-spectra and therefore all stably fibrant objects.
  Therefore, suppose $f: \X \ra \Y$ is a map in ${\widetilde {\Spt}}({\rmS})$ between stably fibrant objects, so that $i^*(f): i^*(\X) \ra i^*(\Y)$ is a stable weak-equivalence.
  Since both $i^*(\X)$ and $i^*(\Y)$ are stably fibrant, this is a level-wise weak-equivalence of spectra in $\Spt(\rmS)$,
  i.e., the induced map $i^*(f)(\T^{\wedge n}_{\rmS}): {\it i}^*(\X(\T^{\wedge n}_{\rmS})) \ra {\it i}^*(\Y(\T^{\wedge n}_{\rmS}))$ is a weak-equivalence for every $n$.
  Since the objects in $\Sph_{\rmS}$ are also just finite dimensional $k$-vector spaces (i.e. without any $\rmG$-action),
  it follows that $f$ itself is a level-wise weak-equivalence of spectra and therefore also a stable weak-equivalence in ${\widetilde {\Spt}}({\rmS})$. Stated another way,
  this shows that the functor $i^*$ both detects and preserves stable weak-equivalences between fibrant objects. An entirely similar argument
   proves that $j^*$ both preserves and detects stable weak-equivalences between fibrant objects.
  \vskip .1cm
  Next we make the following observation: 
  \be \begin{equation}
       \label{i*.preserves.free}
 i^*({\mathcal F}_{\rmT_{\rmV}^{\rmS}}) = F_{\T^{\wedge n}_{\rmS}}, \mbox{ where } n= dim(\rmV)  \mbox{ and } {\it j}^*({\mathcal F}_{{\it j}(\rmT_{\rmV}^{\rmS})}) = {\mathcal F}_{\rmT_{\rmV}^{\rmS}}.
      \end{equation} \ee
\vskip .1cm \noindent
One may see the first by evaluating both sides at $\T^{\wedge m}_{\rmS}$, $m \in {\mathbb N}$ and the second by evaluating both sides at $\rmT_{\rmW}^{\rmS} \in \Sph^{\rmG}_{\rmS}$.
\vskip .1cm
The next step is to show the following holds: let $Q$ denote the fibrant replacement functor in the projective
stable model category structures on any one of the model categories ${\widetilde {\Spt}} ^{\rmG}_{\rmS}$, ${\widetilde {\Spt}}({\rmS})$ and $\Spt(\rmS)$. Then the functors $i^*$ and $j^*$ strictly commute with
$Q$ in the sense 
\be \begin{equation}
\label{Q.commutes}
Q \circ i^* = i^* \circ Q \mbox{ and } Q \circ j^* = j^* \circ Q.
\end{equation} \ee
We will only provide a proof for the first equality, since the second equality may be proved in a similar manner.
To see this, one needs to recall how a functorial fibrant replacement is constructed
 making use of the small object argument: see \cite[Theorem 2.1.14]{Hov99}.  We will consider this for an object $\X \in {\widetilde {\Spt}}({\rmS})$. It is defined as the transfinite colimit of a filtered direct system of
 spectra $\X_{\alpha} \in {\widetilde {\Spt}}({\rmS})$, starting with $\X_0 = \X$. In order to obtain $\X_{\alpha+1}$ from $\X_{\alpha}$, we consider 
 all commutative squares of the form
 \[ \xymatrix{{A_{\alpha}} \ar@<1ex>[r] \ar@<1ex>[d] & {\X_{\alpha}} \ar@<1ex>[d]\\
              {B_{\alpha}} \ar@<1ex>[r] & {*} }
 \]
\vskip .1cm \noindent
with $A_{\alpha} \ra B_{\alpha}$ one of the generating trivial cofibrations in the projective stable model structure. 
Then we let $X_{\alpha+1}$ be defined as the corresponding pushout, after having replaced 
$A_{\alpha} \ra B_{\alpha}$ by the sum of all
 such maps as one varies over the generating trivial cofibrations. Since the above pushout and the colimit are taken 
after evaluating a spectrum at each object $\rmT_{\rmV}$, it should be clear that the functor $i^*$ commutes with
such colimits and pushouts. Moreover, ~\eqref{i*.preserves.free} shows that the functor $i^*$ sends the generating 
trivial cofibrations for the stable projective model structure on ${\widetilde {\Spt}}({\rmS})$ to the generating trivial cofibrations 
of the stable projective model structure on $\Spt$ and that every generating trivial cofibration in this model structure on
$\Spt(\rmS)$ may be obtained by applying the functor $i^*$ to a generating trivial cofibration in the above model structure on
${\widetilde {\Spt}}({\rmS})$. 
\vskip .1cm
Finally, we now observe from \cite[Lemma 4.1.7]{HSS}, that it suffices to prove that for any object
$\X \in \Spt(\rmS)$ ($\Y \in {\widetilde {\Spt}}^{\rmG}({\rmS})$), which is cofibrant in the projective stable model structure there, the composite map 
$\X \ra i^*{\mathbb P}(\X) \ra i^* Q({\mathbb P}(\X))$ ($\Y \ra j^*\tilde {\mathbb P}(\Y) \ra j^*Q(\tilde {\mathbb P}(\Y)$) is a stable weak-equivalence. In view of ~\eqref{Q.commutes},
we obtain the identification $i^*Q({\mathbb P}(\X)) = Q (i^*( {\mathbb P}(\X)))$
($j^*Q(\tilde {\mathbb P}(\Y)) = Q(j^*(\tilde {\mathbb P}(\Y)))$, \res).
  Clearly the map $i^*{\mathbb P}(\X) \ra
Q (i^*( {\mathbb P}(\X)))$ ($j^*\tilde{\mathbb P}(\Y) \ra Q(j^*\tilde {\mathbb P}(\Y))$) is a stable weak-equivalence, since $Q (i^*( {\mathbb P}(\X)))$ 
($Q (j^*(\tilde {\mathbb P}(\Y)))$ 
is a stably fibrant replacement of
$i^*( {\mathbb P}(\X))$ ($j^*(\tilde {\mathbb P}(\Y))$, \res).

\vskip .2cm
Next one recalls how a functorial cofibrant replacement is constructed
 making use of the small object argument: see \cite[Theorem 2.1.14]{Hov99}.  We will consider this for an object $\X \in {\widetilde {\Spt}}({\rmS})$. 
 It is defined as the transfinite colimit of a filtered direct system of
 spectra $\X_{\alpha} \in {\widetilde {\Spt}}({\rmS})$, starting with $\X_0 = *$. In order to obtain $\X_{\alpha+1}$ from $\X_{\alpha}$, we consider 
 all commutative squares of the form
 \[ \xymatrix{{A_{\alpha}} \ar@<1ex>[r] \ar@<1ex>[d] & {\X_{\alpha}} \ar@<1ex>[d]\\
              {B_{\alpha}} \ar@<1ex>[r] & {\X } }
 \]
\vskip .1cm \noindent
with $A_{\alpha} \ra B_{\alpha}$ one of the generating  cofibrations in the projective stable model structure. 
Then we let $\X_{\alpha+1}$ be defined as the corresponding pushout, after having replaced 
$A_{\alpha} \ra B_{\alpha}$ by the sum of all
 such maps as one varies over the generating  cofibrations. Since the above pushout and the colimit are taken 
after evaluating a spectrum at each object $\T^{\wedge n}_{\rmS}$, it should be clear that the functor $i^*$ commutes with
such colimits and pushouts. One may also see that the functor ${\mathbb P}$ (being left-adjoint to $i^*$) commutes with such colimits.
Moreover the map $* \ra i^*{\mathbb P}(*)$ is a stable weak-equivalence, in fact an isomorphism as may be seen by taking the vector space $\rmV=\{0\}$ in
~\eqref{i*.preserves.free}. Next observe from
~\eqref{free.ident} and ~\eqref{i*.preserves.free} that the map $\X \ra i^*({\mathbb P}(\X))$ is a stable weak equivalence when $\X$
is the source or target of a generating cofibration. Therefore, the conclusion from the above discussion is that 
both the maps 
\be \begin{equation}
\label{c.chi}
\X \leftarrow c\X \ra i^*({\mathbb P}(c\X))
\end{equation} \ee
are stable weak equivalence, when $c\X$ is a cofibrant replacement of $\X$. 
Since ${\mathbb P}$ is a left Quillen functor of the stable projective model structures, it clearly preserves
stable weak-equivalences between cofibrant objects. Moreover, all the maps in the composition
\be \begin{equation}
 \label{compos.eq}
  i^*({\mathbb P}(c\X)) \ra Q(i^*({\mathbb P}(c\X))) \simeq i^*(Q({\mathbb P}(c\X))) \ra i^*(Q({\mathbb P}(\X)))
\end{equation} \ee
are weak-equivalences. The first map is one since $Q$ is the fibrant replacement functor considered above. The weak-equivalence 
after that comes from the fact that the functor $Q$ commutes with the functor $i^*$ as proven in ~\eqref{Q.commutes}. The fact that 
last map is a weak-equivalence comes from the fact that the functor $i^*$ preserves stable weak-equivalences between stably fibrant objects.
Therefore, the commutative diagram
\[\xymatrix{{c\X} \ar@<1ex>[r]^{\simeq} \ar@<1ex>[d]^{\simeq}& {i^*{\mathbb P}(c\X)} \ar@<1ex>[r]^{\simeq} \ar@<1ex>[d] & {i^*Q({\mathbb P}(c\X))} \ar@<1ex>[d]^{\simeq}\\
             {\X} \ar@<1ex>[r] & {i^*{\mathbb P}(\X)} \ar@<1ex>[r] & {i^*Q({\mathbb P}(\X))}}
\]
shows that composition of the maps in the bottom row is a weak-equivalence for $\X$ stably cofibrant in $\Spt({\rmS})$.
  One may similarly prove that composition of the maps in the diagram 
 \be \begin{equation}
 \label{compos.eq.2}
 \Y \ra j^*\tilde {\mathbb P}(\Y) \ra Q(j^*\tilde {\mathbb P}(\Y)) \simeq j^*Q(\tilde {\mathbb P}(\Y)) \ra j^*Q(\tilde {\mathbb P}(\Y))
 \end{equation} \ee
is a stable weak-equivalence for $\Y \in {\widetilde {\Spt}}({\rmS})$ which is stably cofibrant.
 \vskip .2cm
 These complete the proof of the first two statements. Observe that both the functors ${\mathbb P}$ and
 $\tilde {\mathbb P}$ are left-Kan extensions and therefore, commute with the smash-products of spectra, which are also 
 left-Kan extensions. This completes the proof of the proposition.
 \end{proof}

\vskip .2cm
Given a commutative ring spectrum $\cE^{\rmG} \in \Spt^{\rmG}$, we let $\cE = i^*({\tilde {\mathbb P}}\tilde \rmU(\cE^{\rmG})$, which is 
 a commutative ring spectrum in $\Spt(\rmS)$. For example, the equivariant sphere spectrum $\mbS^{\rmG}$ provides 
  $\mbS = i^*({\tilde {\mathbb P}}\tilde \rmU(\mbS^{\rmG}))$, the usual sphere spectrum.
  Then one readily proves the existence of a Quillen equivalence between the stable
  model categories ${\widetilde {\Spt}}^{\rmG}({\rmS}, \tilde \rmU({\cE^{\rmG}}))$ and ${\widetilde {\Spt}}({\rmS}, {\tilde {\mathbb P}}\tilde \rmU({\cE^{\rmG}}))$
  as well as between the stable model categories
  ${\widetilde {\Spt}}({\rmS}, {\tilde {\mathbb P}}\tilde \rmU({\cE^{\rmG}}))$ and $\Spt(\rmS, {\cE})$, just as in Proposition ~\ref{comp.1}. 

\section{Spanier-Whitehead duality in the motivic and \'etale setting}
We begin with the following result on the Spanier-Whitehead dual in the motivic and \'etale setting.
\begin{theorem} 
\label{dualizable.pos.char}
Let $k$ denote a perfect field of characteristic $p \ge 0$ and let $\rmX$ denote a smooth quasi-projective scheme over $k$. 
Let $\ell$ denote a fixed prime different from $char(k)$ and let $\cE$ denote a commutative motivic ring spectrum
which is {\it $Z_{(\ell)}$-local}. 
Then $\cE \wedge \rmX_+$ is  dualizable in the category $\Spt(\k_{\rm mot}, \cE)$ of module spectra over $\cE$ with the same conclusion holding with no conditions on the
spectrum $\cE$ if $\rmX$ is projective and smooth. In particular, this holds for ring spectra $\cE$ of the form 
$\rmK{\overset L {\underset {\mbS_{\k}} \wedge}}\H(\bZ/\ell^{n})$, where $\rmK$ is a commutative motivic ring spectrum, $\ell$ is a prime different from $p$
and $n \ge 1$. Here $\H(\bZ/\ell^{n})$ denotes the usual
$\rmZ/\ell^{n}$- motivic Eilenberg-Maclane spectrum, and ${\overset L {\underset {\mbS_{\k}} \wedge}}$ is the derived smash product.
\end{theorem}
\begin{proof}
 The proof makes strong use of Gabber's refined alterations.  Though the arguments below are now 
rather well-known (see, for example, \cite{K} or \cite[2.5]{HKO}), 
it is necessary for us to sketch the relevant arguments in some detail, {\it so as to show that
 they indeed carry through under \'etale realization and change of base fields}.
 \vskip .2cm
 We will give two somewhat different proofs of this result, one of which holds only when $\cE$ admits weak traces 
in the sense of \cite{K} and the other holds more generally making use of \cite{Ri13}. Next one may observe that to prove the spectrum  $\cE\wedge \rmX_+$ is dualizable in $\Spt(\k_{\rm mot}, \cE)$, it suffices to prove that the  natural maps
\be \begin{equation}
\label{st.dual.alterations}
    \eta _{\cE}^{\rmX}: \rmP {\overset L {\wedge  _{\cE}} }\Hom _{\cE}(\cE \wedge \rmX_+, \cE) \ra \Hom _{\cE}(\cE \wedge \rmX_+, P), \quad \cE \wedge \rmX_+ \ra \rmD _{\cE}(\rmD _{\cE}(\cE \wedge \rmX_+))
    \end{equation} \ee
\vskip .2cm \noindent
are weak-equivalences for every $\cE$-module spectrum $\rmP$. 
\vskip .2cm

In case $\rmX$ is projective and smooth, it is well-known
 (see the remarks following Appendix A: Definition ~\ref{V.collapse}) that the Thom-space of the {\it virtual normal} bundle (defined as in Appendix A: Definition ~\ref{V.collapse})  over $\rmX$ 
de-suspended a finite number of times is a (Spanier-Whitehead) dual  of ${\Sigma^{\infty}_{\T}}\rmX_+$. Therefore, the above Thom-space
de-suspended a finite number of times and smashed with $\cE$ will be a (Spanier-Whitehead) dual of $\cE \wedge \rmX_+$ in the category of $\cE$-module spectra.
\vskip .2cm
Recall $\SH(\k, {\cE})$ denote the motivic stable homotopy category of $\cE$-module spectra, i.e., the homotopy category associated to $\Spt(\k_{\rm mot}, \cE)$.
Let $\SH_d(\k, {\cE})$ denote the localizing subcategory of $\SH(\k, \cE)$ which is generated by the shifted $\cE$-suspension spectra of smooth connected schemes of 
dimension $\le d$.
In general one proceeds by ascending induction on the dimension of $\rmX$ to prove that the maps in ~\eqref{st.dual.alterations} are weak-equivalences, the case of dimension $0$ 
reducing to the case $\rmX$ is projective and smooth. When $\rmX$ is quasi-projective of dimension $d$, one may assume $j:\rmX \ra \rmY$ is an open immersion 
in a projective scheme $\rmY$ and let $\rmf:\rmY' \ra \rmY$ denote the map given by Gabber's refined alteration 
so that $\rmX'= \rmf^{-1}(\rmX)$ is the
complement of a divisor with strict normal crossings. \index{refined alteration} Let $\rmU \subseteq \rmX$ denote the open subscheme over which $\rmf$ restricts to an $fps\ell'$-cover 
${\rm g}: \rmV = \rmf^{-1}(\rmU) \ra \rmU$. 
\vskip .2cm
Since $\rmY'$ is smooth and projective, $\cE \wedge \rmY'_+$ is  dualizable in the category of $\cE$-module 
spectra. One may next observe that if 
\be \begin{multline}
  \label{two.out.of.three.dual}
     \begin{split}
\rmA' \ra \rmA \ra \rmA'' \ra \rmA[1] \mbox{ is a stable cofiber sequence in } \Spt(\k_{\rm mot}, \cE) \mbox{ and if two of the three terms } \\
\rmA', \, \rmA \, and \,\rmA'' \mbox{ are dualizable in } \Spt(\k_{\rm mot}, \cE), \mbox{ so is the third term }.
\end{split}
\end{multline} \ee
Therefore, by homotopy-purity, induction on the
number of irreducible components of $\rmY' - \rmX'$  one observes that $\cE \wedge \rmX'_+ $ is also  dualizable in the
same category.  Now one considers the stable cofiber sequences:
\be \begin{equation}
     \label{cofiber.seqs}
\cE \wedge \rmV_+ \ra \cE \wedge \rmX'_+ \ra \cE \wedge \rmX'/\rmV, \quad \cE \wedge \rmU_+ \ra \cE \wedge \rmX_+ \ra \cE \wedge \rmX/\rmU.
\end{equation} \ee
By an argument as in \cite[Lemma 66]{RO} (see also \cite[2.5]{HKO}), both $\cE \wedge \rmX'/\rmV$ and $\cE \wedge \rmX/\rmU $ belong to $\SH_{d-1}(\k, {\cE})$. 
We will provide some details on this argument, for the convenience of the reader. In case the complement $\rmZ= \rmX'-\rmV$ is also smooth, the homotopy purity Theorem
\cite[Theorem 3.2.33]{MV} shows that $\rmX'/\rmV$ is weakly equivalent to the Thom-space of the normal bundle $\rmN$ associated to the closed immersion $\rmZ \ra \rmX'$. 
Ascending induction on the number of open sets in a Zariski open covering over which the normal bundle $\rmN$ trivializes reduces it to the case
when $\rmN$ is trivial. In this case the conclusion is clear. In general, since the base field is assumed to be perfect, one can stratify  $\rmZ$ by a finite
number of locally closed subschemes that are smooth. This gives rise to a sequence of motivic spaces filtering $\rmX'/\rmV$, so that the homotopy cofiber
of two successive terms will be of the form considered earlier. An entirely similar argument applies to $\rmX/\rmU$.
\vskip .2cm
 Therefore, by the induction hypotheses,
both the maps $\eta _{\cE}^{\rmX'/\rmV}$ and $\eta _{\cE}^{\rmX/\rmU}$ are weak-equivalences, where
$\eta _{\cE}^{\rmX'/\rmV}$ ($\eta _{\cE}^{\rmX/\rmU}$) denotes the map corresponding to $\eta_{\cE}^{\rmX}$ in ~\eqref{st.dual.alterations} when $\rmX$ there is
replaced by $\rmX'/\rmV$ ($\rmX/\rmU$, \res).
It follows therefore, by  ~\eqref{two.out.of.three.dual}, that the map
$\eta_{\rmV}^{\cE}$ is also a weak-equivalence. Now we make the key observation proven below that 
$\cE \wedge \rmU_+$ is a retract of $\cE \wedge \rmV_+ $ at least when $\rmU$ is a sufficiently small Zariski open subscheme and that,
therefore, the map $\eta_{\rmU}^{\cE}$ is also a weak-equivalence. Now the second stable cofiber sequence in ~\eqref{cofiber.seqs} together with another application of
 ~\eqref{two.out.of.three.dual} proves that the map $\eta_{\rmX}^{\cE}$ is also a weak-equivalence. One may prove the second map
in ~\eqref{st.dual.alterations} is a weak-equivalence by a similar argument.
\vskip .2cm
It follows straight from the definition that the spectra $\rmK{\overset L {\underset {\mbS_{\k}} \wedge}}\H(\bZ/\ell^{n})$ and $\H(\bZ/\ell^{n})$ are 
$\rmZ_{(\ell)}$-local. (Observe also that  $\H(\bZ/\ell^{n})$ admits weak-traces 
and that $\rmK{\overset L {\underset {\mbS_{\k}} \wedge}}\H(\bZ/\ell^{n})$ admits weak traces when $\rmK$ admits weak-traces.)
\end{proof}
\begin{lemma}
\begin{enumerate}[\rm(i)]
 \item Let $\rmV$, $\rmU$ denote two smooth schemes over $k$ and let ${\rm g}: \rmV \ra \rmU$ denote an $fps\ell'$-cover, where $\ell$ is a fixed prime different from $char(k)$. If $\cE$ is
a commutative ring spectrum which is $\rmZ_{(\ell)}$-local and which admits weak traces as in \cite{K}, then the map $id _{\cE} \wedge {\rm g}_+  : \cE \wedge \rmV_+  \ra \cE \wedge \rmU_+ $ has a section.
\item More generally the same conclusion holds if $\rmU$ is a sufficiently small Zariski open subscheme and for any commutative motivic ring spectrum $\cE$
that is $\rmZ_{(\ell)}$-local.
\end{enumerate}
\end{lemma}
\begin{proof}  $(i)$ The first observation is that it suffices to show the induced natural transformation:
\[ [\cE \wedge \rmU_+ ,  \quad ] {\overset {{\rm g}^*} \ra} [\cE \wedge \rmV_+ , \quad ] \]
has a splitting, where $[ \rmK, \rmL]$ denotes homotopy classes of maps in $\Spt(\k_{\rm mot}, \cE)$, with $\rmK$ cofibrant and $\rmL$ fibrant. Denoting the structure map $\rmU \ra Spec \, \k$ by $\rma$, 
one may identify $[\cE \wedge \rmU_+ , \rmF]$ ($[\cE \wedge \rmV_+ , \rmF]$) with $[\cE, \rmR\rma_*\rma^*(\rmF)]$ ($[\cE, \rmR\rma_* \rmR{\rm g}_*{\rm g}^*\rma^*(F)]$, \res) for any fibrant $\cE$-module spectrum $\rmF$.
Since $\cE$ has a structure of traces, so does $\rmF$.
Therefore, the natural map
$\rmR\rma_*\rma^*(\rmF) \ra \rmR\rma_*R{\rm g}_*{\rm g}^*\rma^*(\rmF)$ has a splitting provided by the map $\rmd^{-1}{\rm {Tr(g)}}$, where $\rmd$ is the degree of the map ${\rm g}$ and
${\rm Tr(g)}$ denotes the trace associated to ${\rm g}$.
\vskip .2cm
(ii) The proof of (ii) is essentially worked out in \cite{Ri13}. 
\end{proof}
\vskip .2cm
We proceed to show that the notion of dualizability is preserved by various standard operations, like 
change of base fields, or change of sites. Recall that we have already assumed the base scheme is a perfect field $\k$ satisfying the hypothesis 
~\eqref{etale.finiteness.hyps}.  Let $\bar k$ denote an algebraic closure of $k$. Then we obtain the following functors (which in fact denote the corresponding left-derived functors):
\be \begin{equation}
     \label{maps.topoi.2}
\epsilon^*:\Spt(\k_{\rm mot}) \ra \Spt(\k_{et}),  \bar \epsilon^*:\Spt({\bar \k_{\rm mot}}) \ra \Spt({\bar \k_{et}}) \mbox{ and } \eta^*: \Spt(\k_{et}) \ra \Spt({\bar \k_{et}}).
\end{equation}\ee
\vskip .2cm
Since \'etale cohomology is well-behaved only with torsion coefficients prime to the characteristic, one will need to also consider the 
functors $\theta: \Spt(\bar \k_{et}) \ra \Spt(\bar \k_{et})$  sending commutative ring spectra $\cE$ to 
$\cE{\overset L {\underset {\mbS_{\k_{et}}} \wedge}}\H(\bZ/\ell^n)$
where $\H(\bZ/\ell^n)$ denotes the  mod-$\ell^n$ Eilenberg-Maclane spectrum in $\Spt(\k_{et}, \epsilon^*(\T))$, with  $\ell$ a fixed prime different from $char(k)$.
Again, if $\ell$ is a fixed prime different from $char(k)$, and $\cE$ is a commutative ring spectrum in $\Spt(\k_{et}, \cE)$, 
we will also consider the functor sending
spectra $\rmM \in \Spt(\k_{et}, \cE)$ to $\rmM \wedge  _{\cE} \cE(\ell^{n})$: we will denote this functor by $\phi _{\cE}$.
We will adopt  the convention that 
the above functors in 
fact denote their corresponding 
left derived functors.
\begin{proposition} \index{\'etale realization}
 \label{compat.dualizability}
Let $\ell$ denote a fixed prime different from $char(k)$,  where $k$ is assumed to be a perfect field satisfying the hypothesis 
~\eqref{etale.finiteness.hyps}. Let $n$ denote a positive integer.
 If $\cE$ is a commutative motivic ring spectrum  so that it is $\ell$-primary torsion as in Definition ~\ref{Zl.local}, then the functors
$\epsilon^*$,  $ \bar \epsilon^*$, $\eta^*$ send the dualizable objects of the form $\cE \wedge \rmX_+$ appearing in 
Theorem ~\ref{dualizable.pos.char} to dualizable objects. 
\vskip .1cm
The same conclusion holds for the functors $\theta$ and $\phi _{\cE}$ if $\cE$ is a motivic ring spectrum that is $\ell$-complete. 
\end{proposition}
\begin{proof} One may make use of the fact that the base field is perfect to see that 
 base-change to the algebraic closure of the base field sends
projective smooth schemes to projective smooth schemes and preserves strict normal crossings divisors. 
\vskip .2cm
Observe that the functor $\epsilon^*$ sends motivic spectra which are $\ell$-primary torsion for 
a fixed prime different from $char(k)$ to \'etale spectra which are $\ell$-primary torsion
 and preserves all split maps. It also sends
(motivic) spectra with traces to spectra with traces. Therefore,
the same argument making use of the stable cofiber sequences in ~\eqref{cofiber.seqs} carries over to prove that 
 $\epsilon^*$ and $\bar \epsilon^*$ send dualizable objects in Theorem ~\ref{dualizable.pos.char} to dualizable objects, when the
 spectrum $\cE$ is $\ell$-primary torsion. 
One may prove similarly that the functor $\eta^*$ sends dualizable objects appearing in Theorem ~\ref{dualizable.pos.char} 
to dualizable objects. The conclusion that
the functors $\theta$ and $\phi _{\cE}$ send dualizable objects to dualizable objects should be straight-forward.
 This completes the proof of the Proposition.
\end{proof}
\section{Construction of the transfer}
In this section, we proceed 
to obtain transfer maps for torsors for linear algebraic groups, i.e., when $\rmp : \rmE \ra \rmB$ is a smooth map of 
smooth quasi-projective schemes that is 
a $\rmG$-torsor for a linear algebraic group
$\rmG$. We adopt the framework discussed
in Theorem ~\ref{thm.transf.const}.
\vskip .1cm
 Next, recall the definition of weakly monoidal functors  from Definition ~\ref{monoidal.functs}. Let $\Spt'$ and $\Spt$ denote 
two symmetric monoidal stable model categories.
We say a weakly monoidal functor $\rmF: \Spt' \ra \Spt$ is {\it a monoidal functor} if the  map $\mu: \rmF(\X') \otimes \rmF(\Y') \ra \rmF(\X' \otimes \Y')$ in
~\eqref{monoidal.funct} is a weak-equivalence for all objects $\rmX'$ and $\rmY'$ in $\Spt'$ and if the given map $\epsilon: \rmA \ra \rmF(\rmA')$ is a weak-equivalence, where $\rmA'$ ($\rmA$) denotes the unit of the 
category $\Spt'$ ($\Spt$, \res).
\begin{proposition} 
\label{DP.compat}
{\rm (See \cite[2.2 Theorem]{DP} and \cite[2.4 Corollary]{DP}.)}
 Assume that the functor $\rmF$ is monoidal, induces a functor of the corresponding homotopy categories, that the object 
 $\rmX' \in \Spt'$ is dualizable, and $\rmA'$ is the unit of $\Spt'$. Then $\rmF(\rmX') \in \Spt$ is dualizable and 
 $\rmF(\rmD(\rmX')) \simeq \Hom(\rmF(\rmX'), \rmF(\rmA'))$, where $\Hom$ again denotes the derived (internal) $\Hom$ in $\Spt$.
\end{proposition}
\vskip .1cm
At this point we make implicit use of the chain of equivalences of stable model category structures on ${\widetilde {\Spt}}^{\rmG}(\k_{\rm mot})$, 
${\widetilde {\Spt}}(\k_{\rm mot})$ and $\Spt(\k_{\rm mot})$ proven in Proposition ~\ref{comp.1} which are in fact given by monoidal functors.
Therefore,  Proposition ~\ref{DP.compat}) shows that the theory of Spanier-Whitehead duality 
 currently known in $\Spt(\k_{\rm mot})$ carries over to ${\widetilde {\Spt}}^{\rmG}(\k_{\rm mot})$. 
 \vskip .2cm
\subsection{Construction of the transfer in a general framework}
\label{gen.transfer}
 Assume that $\Spt$ denotes a symmetric monoidal
stable model category where the monoidal structure is denoted $\wedge$ and where the unit of the monoidal structure
 is denoted $\mbS$. We will further assume that the object $\X$ in $\Spt$ comes equipped with a diagonal map $\Delta: \X \ra \X \wedge \X$
and a co-unit map $\kappa: \X \ra \mbS$ so that $\Delta$ provides $\X$ with the structure of a co-algebra: see \cite[section 5]{DP}.
\begin{definition}
 \label{tr.general}
(i) Now one may define the {\it trace} associated to any self-map $\rmf: \X \ra \X$ of an object that is
 dualizable as follows. Recall that we have denoted the {\it evaluation} map $\rmD\X \wedge \X \ra \mbS$ by $e$.
The dual of this map is the {\it co-evaluation} map $c: \mbS \ra \X \wedge \rmD\X$.
Now the {\it trace of $\rmf$} (denoted $\tau_{\rmX}(\rmf)$ or often just $\tau(\rmf)$) is the composition (in $\SH$) 
\be \begin{equation}
     \label{gen.trace}
    \mbS {\overset c \ra} \X \wedge \rmD\X {\overset {\tau} \ra} \rmD\X \wedge \X  {\overset {id \wedge \rmf} \ra} \rmD\X \wedge \X {\overset e \ra} \mbS.  
\end{equation} \ee
\vskip .2cm \noindent
where $\tau$ is the map interchanging the two factors. \index{trace}
\vskip .2cm
(ii) Then, assuming $\X$ comes equipped with a diagonal map $\Delta: \X \ra \X \wedge \X$ so that $\X$ has the structure of a co-algebra,  we define the {\it transfer} as the composition in $\SH$ \index{transfer}:
\be \begin{equation}
     \label{gen.transfer}
\tr(\rmf):\mbS {\overset c \ra} \X \wedge \rmD\X {\overset {\tau} \ra} \rmD\X \wedge \X {\overset {id \wedge \Delta} \ra }
 \rmD\X \wedge \X \wedge \X {\overset {id \wedge \rmf \wedge \rmf } \ra} \rmD\X \wedge \X  \wedge \X {\overset  {e \wedge id} \ra} \mbS \wedge \X = \X
\end{equation} \ee
\vskip .2cm 
(iii) If $\Y \eps \Spt$ is another object, we will also consider the following variant $\tr(\rmf_{\Y}) = \Y \wedge \tr(\rmf): \Y \wedge \mbS \ra \Y \wedge \X$.
\end{definition}
The composition $\rmD\X \wedge \X {\overset {id \wedge \Delta} \ra }
 \rmD\X \wedge \X \wedge \X {\overset {id \wedge \rmf \wedge \rmf } \ra} \rmD\X \wedge \X  \wedge \X $ will often be denoted $id \wedge \Delta(\rmf)$.
\vskip .2cm
Assume in addition to the above situation that $\rmA$ denote a commutative ring object in the symmetric monoidal model category $\Spt$. Let $\Spt_{\rmA}$ denote the 
subcategory of $\Spt$ consisting of objects $\rmM$ provided with an associative and commutative pairing $\rmA \wedge \rmM  \ra \rmM$: see, for example, \cite{SSch}. The category $\Spt_{\rmA}$ 
will be provided with the monoidal structure defined by $\rmM \wedge_{\rmA} \rmN $ defined as the co-equalizer: 
$\rmM \wedge \rmA \wedge \rmN \stackrel{\ra}{\ra} \rmM \wedge \rmN$, where the two arrows denote the multiplication by $\rmA$ on the right on $\rmM$ and on the 
left on $\rmN$.

We will denote this category by $\Spt_{\rmA}$. This will be provided with the model structure defined in \cite[Theorem 4.1]{SSch} so that it is also a stable symmetric monoidal model category.
\begin{proposition}
 \label{D.A.wedgeX}
 Assume the above situation. 
 \begin{enumerate}[\rm(i)]
 \item Then the functor $\Spt \ra \Spt_{\rmA}$, given by $\rmX \mapsto \rmA \wedge \rmX$ is a monoidal functor.
 \item Let $\rmX \in \Spt$ and let $\rmf: \rmX \ra \rmX$ denote any map in $\Spt$. Then $\rmA \wedge \tau_{\rmX}(\rmf) = \tau_{\rmA \wedge \rmX}(id \wedge \rmf)$.
 \end{enumerate}
 \end{proposition}
 \begin{proof} (i) is clear. To see (ii), first observe that $\rmA \wedge {\rmD}(\rmX) = \rmA \wedge \RHom(\rmX, \mbS) \simeq \RHom_{\rmA}(\rmA \wedge \rmX, \rmA) = {\rmD}_{\rmA}(\rmA \wedge \rmX)$,
 where $\RHom$ ($\RHom_{\rmA}$) denotes the derived internal $Hom$ in $\Spt$ ($\Spt_{\rmA}$, \res). Now the definition of the trace $\tau_{\rmA \wedge \rmX}(id \wedge \rmf)$
 shows that it identifies with $\rmA \wedge \tau_{\rmX}(\rmf)$. This completes the proof of (ii).
 \end{proof}
\vskip .2cm

\subsection{\bf The $\group$-equivariant pre-transfer}
\label{pretransfer} \index{pre-transfer} \index{trace} \index{equivariant pre-transfer}
Let $\rmX$ denote a smooth quasi-projective scheme, or more generally an unpointed simplicial presheaf defined on $\Sm_{\k}$, subject to the requirement that  
${\Sigma^{\infty}_{\T}}\rmX_+ \in \Spt(\k_{\rm mot})$ be dualizable. Corresponding results will hold if $ \cE \wedge \rmX_+$ is dualizable in $\Spt(\k_{\rm mot}, \cE)$
 where $\cE^{\rmG} \in \Spt^{\rmG}(\k_{\rm mot})$ is a commutative ring spectrum, with $\cE = i^*(\tilde {\mathbb P}\tilde \rmU(\cE^{\rmG})) \in \Spt(\k_{\rm mot})$ denoting the corresponding non-equivariant ring spectrum.
 Then the equivariant sphere spectrum $\mbS^{\rmG}$ will be replaced by $\cE^{\rmG}$ everywhere in the construction discussed below.
\vskip .1cm	
We will further assume $\rmX$ is provided with an action by the linear algebraic group $\group$.
Associated to any $\group$-equivariant self-map $\rmf: \rmX \ra \rmX$, over the base field $\k$,  we
will presently define a 
pre-transfer map following roughly the definition given in \cite{DP}. The main improvement we
need is to make all the maps that enter into the definition of the pre-transfer ${\group}$-equivariant.  We will define
the ${\group}$-equivariant pre-transfer as the composition of a sequence of maps in ${\widetilde {\Spt}}^{\group}(\k_{\rm mot})$ which are 
all ${\group}$-equivariant. Throughout the following definition we will often abbreviate $\mbS^{\rmG} \wedge\rmX_+$ to just $\rmX_+$ and
taking the dual will mean as in ~\eqref{Dual}.
\vskip .2cm
\begin{definition}
   \label{coeval.pretr.trace}
(i) Accordingly we proceed to first define a {\it ${\group}$-equivariant co-evaluation map}, where the source is the ${\group}$-sphere spectrum $\mbS^{\group}$. We start with the evaluation map
$e: \rmD(\rmX_+) \wedge \rmX_+ \ra \mbS^{\group}$. On taking its dual in ${\widetilde {\Spt}}^{\rmG}(\k_{\rm mot})$, we obtain the map
\be \begin{equation}
\label{G.equiv.coeval}
c:\mbS^{\group} \simeq \rmD(\mbS^{\group}) \ra \rmD(\rmD(\rmX_+) \wedge \rmX_+) {\overset {\simeq} \leftarrow} \rmD(\rmX_+) \wedge \rmD \rmD(\rmX_+){\overset {\simeq} \leftarrow} \rmD(\rmX_+) \wedge \rmX_+  {\overset {\tau} \rightarrow} \rmX_+ \wedge \rmD( \rmX_+).
\end{equation} \ee
The above composition will be the {\it co-evaluation} map $c$.
 Observe that all the maps above are ${\group}$-equivariant and the maps going in the wrong-direction are in fact weak-equivalences.
 \vskip .1cm 
(ii)  Next we consider the map:
\be \begin{equation}
\label{pr.tr.2}
\rmX_+ \wedge \rmD(\rmX_+) {\overset \tau \ra} \rmD(\rmX_+) \wedge \rmX_+ {\overset {id \wedge \Delta} \ra}  \rmD(\rmX_+) \wedge \rmX_+ \wedge \rmX_+
{\overset {id \wedge \rmf \wedge \rmf } \ra } \rmD(\rmX_+)  \wedge \rmX_+ \wedge \rmX_+ {\overset {e \wedge id } \ra } 
 \mbS^{\group} \wedge \rmX_+ . 
\end{equation} \ee
Observe that  the above
diagram is in fact a diagram in ${\widetilde {\Spt}}^{\rmG}(\k_{\rm mot})$ where all the spectra and the maps are ${\group}$-equivariant. 
\footnote{One can put in a slightly more general form of the diagonal map $\Delta$, which will in fact be important for establishing the 
localization or Mayer-Vietoris properties of the pre-transfer. This is discussed in   \cite[Definition 2.2]{JP23}.)}
Now we may compose the co-evaluation map in ~\eqref{G.equiv.coeval} with the map in ~\eqref{pr.tr.2} to define 
 {\it the ${\group}$-equivariant pre-transfer}, denoted $\tr'^{\rmG}(\rmf)$.  Therefore, this will be  the following composition:
\be \begin{equation}
    \label{G.equiv.pretransfer}
\tr'^{\rmG}(\rmf_+):\mbS^{\group} \simeq \rmD(\mbS^{\group}) \ra \rmD( \rmD (\rmX_+) \wedge \rmX_+) {\overset {\simeq} \leftarrow} \rmD (\rmX_+) \wedge \rmD \rmD (\rmX_+){\overset {\simeq} \leftarrow} \rmD (\rmX_+) \wedge \rmX_+  {\overset {\tau} \rightarrow} \rmX_+ \wedge \rmD (\rmX_+) \ra  \mbS^{\group} \wedge \rmX_+.
\end{equation} \ee
\vskip .1cm
(iii) Given $\rmY$,  which is another smooth quasi-projective scheme, or more generally an unpointed simplicial presheaf defined on $\Sm_{\rm S}$, provided with an action by $\rmG$,
we define $\tr'^{\rmG}(\rmf_{\rmY+}): \rmY \times \mbS^{\rmG} \ra \rmY \times  (\mbS^{\rmG} \wedge \rmX_+)$ to be $id_{\rmY} \times \tr'^{\rmG}(\rmf_+)$.
\vskip .1cm
(iv) We define the {\it trace}, $\tau_{\rmX}^{\rmG}(\rmf_+)$ to be the composition of the pre-transfer with the map $ \mbS^{\group}\wedge \rmX_+ \ra \mbS^{\group}$
 collapsing all of $\rmX_+$ to $Spec \, \k_+$. Similarly we define $\tau_{\rmX}^{\rmG}(\rmf_{\rmY+})$ to be the composition of the pre-transfer $\tr'^{\rmG}(\rmf_{\rmY+})$ with
the map $\rmY \times (\mbS^{\group} \wedge \rmX_+)  \ra \rmY \times \mbS^{\group}$. {\it For the most part we will suppress the superscript $\rmG$ and denote the 
above traces as $\tau_{\rmX}(\rmf_+)$ or $\tau_{\rmX}(\rmf_{\rmY+})$.}
 \vskip .1cm
 (v) If $\cE^{\rmG} \in \Spt^{\rmG}(\k_{\rm mot})$ is a commutative ring spectrum,  $\cE = i^* \tilde{\mathbb P}(\tilde\rmU(\cE^{\rmG}))$ is the associated ring spectrum in $\Spt(\k_{\rm mot})$,
 and $\cE \wedge \rmX_+ \in \Spt(\k_{\rm mot})$ is dualizable, one defines  co-evaluation, 
 pre-transfer and trace maps similarly by replacing $\mbS^{\rmG} \wedge\rmX_+$ ($\mbS^{\rmG}$) by $\cE^{\rmG} \wedge \rmX_+$ ($\cE^{\rmG}$, \res).
\end{definition}

\vskip .2cm

The next goal is to define a transfer map that will define a wrong-way map in generalized cohomology for a $\rmG$-torsor $\rmp: \rmE \ra \rmB$,
as well as in Borel-style equivariant generalized motivic (and \'etale) cohomology associated to actions of linear algebraic groups.
Our approach follows closely the construction in \cite[section 3]{BG75}, in { spirit}.
\vskip .1cm
\subsubsection{\bf Convention}
\label{spec.grp}
Let ${\group}$ denote a linear algebraic group. We need to carry out the construction of the transfer in two distinct contexts:
(i) when the group ${\group}$ is {\it special} in Grothendieck's terminology: see \cite{Ch}. For example, ${\group}$ could be a $\GL_n$ for some $n$ or a finite product of $\GL_n$s \mbox{ and } (ii) when ${\group}$ is not
necessarily special. In the first case, every $\rmG$-torsor is locally trivial on the Zariski (and hence the Nisnevich) topology, 
while in the second case $\rmG$-torsors are locally trivial only in the \'etale topology.
\vskip .2cm
In both cases, we will let ${\BG}^{\it gm,m}$ (${\EG}^{\it gm,m}$) denote the $m$-th degree approximation to the classifying 
space of the group $\group$ (its principal $\group$-bundle, \res) as in  \cite{MV} (see also \cite{Tot}).
These are, in general, quasi-projective smooth schemes over $k$. It is important for us to observe that each ${\EG}^{\it gm,m}$,
with $m$ sufficiently large has $k$-rational points, where $k$ is the base field. (This will imply that ${\BG}^{\it gm,m}$, with $m$ sufficiently large also has
$k$-rational points.)
\vskip .2cm
Next we start with a $\rmG$-torsor $\rmE \ra \rmB$, with both $\rmE$ and $\rmB$ smooth quasi-projective schemes over $\k$. We will further assume that $\rmB$ is {\it always connected}.
Next, we will find affine replacements for these schemes. 
One may first find an affine
replacement ${\widetilde {\rmB}}$ for $\rmB$ (${\widetilde {{ \BG}^{\it gm,m}}}$ for ${ \BG}^{\it gm,m}$) by applying the well-known construction of Jouanolou (see \cite{Joun}) 
and then define ${\widetilde {\rmE}}$  (${\widetilde {{ \EG}^{\it gm,m}}}$) as the pull-back:
\be \begin{multline}
\label{Joun.def}
\begin{split}
{\widetilde {\rmE}} = {\widetilde {\rmB}}{\underset {\rmB} \times} {\rmE}, \tilde \rmp: {\widetilde {\rmE}} \ra {\widetilde {\rmB}}, \quad ({\widetilde {{ \EG}^{\it gm,m}}} = {\widetilde {{ \BG}^{\it gm,m}}}{\underset {{ \BG}^{\it gm,m}} \times} {{ \EG}^{\it gm,m}}, \tilde \rmp_m: {\widetilde {{ \EG}^{\it gm,m}}} \ra {\widetilde {{ \BG}^{\it gm,m}}})\\
\pi_{\rmY}: {\widetilde {\rmE}}_{\rmY \times \rmX}={\widetilde {\rmE}}\times_{\rmG} (\rmY \times \rmX) \ra {\widetilde {\rmE}}\times_{\rmG} \rmY= {\widetilde {\rmE}}_{\rmY}, \quad (\pi_{\rmY, m}: {\widetilde {{ \EG}^{\it gm,m}}}\times_{\rmG}(\rmY \times  \rmX) \ra {\widetilde {{ \EG}^{\it gm,m}}}\times_{\rmG} \rmY)\\
\pi : {\widetilde {\rmE}}_{\rmY}= {\widetilde {\rmE}}\times_{\rmG} \rmY \ra {\widetilde {\rmB}}, \quad (\pi_{ m}: {\widetilde \cE}_{m, \rmY}={\widetilde {{ \EG}^{\it gm,m}}}\times_{\rmG}\rmY  \ra {\widetilde {{ \BG}^{\it gm,m}}} ={\widetilde \cB}_m.)
\end{split} 
\end{multline} \ee

\vskip .2cm
\subsection{\bf The Borel construction applied to  simplicial presheaves with $\rmG$-action} \index{Borel construction}
\label{Borel.construct}
 We break this discussion into two cases,
depending on whether the group $\rmG$ is {\it special} in Grothendieck's classification (see \cite{Ch}). 
In both cases, $\Spc_*^{\rmG}(\k_{\rm mot})$ will denote the category of pointed $\rmG$-equivariant presheaves on 
{\it the big Nisnevich site} of $\k$ as in Definition ~\ref{equiv.prshvs}.
\vskip .1cm
{\it Case 1: when $\rmG$ is special}. Recall this includes all the linear algebraic groups ${\rm GL}_n$, ${\rm SL}_n$, ${\rm Sp}_{2n}$, $n \ge 1$. 
In this case, we start with the construction (i.e., the functor):
\be \begin{equation}
\label{Borel.1}
 \Spc_*^{\rmG}(\k_{\rm mot}) \ra \Spc_*({\widetilde {\rmB}}), {\it X} \mapsto {\widetilde {\rmE}}\times_{\rmG} {\it X} 
\end{equation} \ee
where the quotient construction is explained below. (If we start with an unpointed simplicial presheaf $\rmX$, we let ${\it X} = \rmX_+$ and
we will always assume that the action by $\rmG$ on ${\it X}$ preserves the base point. Therefore,  there is a canonical section 
${\widetilde {\rmB}} \ra {\widetilde {\rmE}}\times_{\rmG} \itX$.) Clearly this extends to a functor:
\be \begin{equation}
\label{Borel.2}
\Spt^{\rmG}(\k_{\rm mot}) \ra {\widetilde \Spt}^{\rmG}({\widetilde {\rmB}}), \X \mapsto {\widetilde {\rmE}}\times_{\rmG} \X.
\end{equation} \ee
where ${\widetilde \Spt}^{\rmG}({\widetilde {\rmB}}) =[\Sph^{\rmG}, \Spc_*(\tilde \rmB)]$.
\vskip .1cm
 In ~\eqref{Borel.1}, one cannot view the
product ${\widetilde {\rm E}} \times \itX$  as a presheaf on the big Nisnevich site and take the
quotient by the action of ${\group}$, with ${\group}$ again viewed as a Nisnevich presheaf: though such a quotient will
be a presheaf on the big Nisnevich site, this will not be the presheaf represented by the scheme (or algebraic space) ${\widetilde {\rmE}}\times_{\rmG}\itX$, when $\itX$ is a scheme.
In order to get this latter presheaf, when ${\group}$ is special, one needs to start with a Zariski open cover $\{\rmU_i|i\}$ of ${\widetilde {\rmB}}$  over
which ${\widetilde {\rmE}}$  is trivial, and then glue together the sheaves  $\rmU_i \times \itX$
making use of the gluing data provided by the torsor ${\widetilde {\rmE}} \ra {\widetilde {\rmB}}$. 
\vskip .2cm
A nice way to
view this construction is as follows, at least when $\itX$ is a Nisnevich sheaf: one needs to in fact take the {\it quotient sheaf} associated to the presheaf quotient
of ${\widetilde {\rmE}} \times \itX$  by the ${\group}$-action on the big Nisnevich site. Then this produces the right object. 
\vskip .2cm 
Denoting by $({\widetilde {\rmE}}\times_{\rmG}\X)_{|\rmU_i}$  the restriction of ${\widetilde {\rmE}}$  to $\rmU_i$, it is clear that $({\widetilde {\rmE}}\times_{\rmG}\X)_{|\rmU_i}$
identifies with
 $\rmU_i \times \X$. Therefore,  it is clear that the construction in ~\eqref{Borel.2} sends a $\rmG$-equivariant map $\alpha: \X \ra \Y$ so that $\tilde \rmU(\alpha)$ is a (stable) weak-equivalence in
 ${\widetilde {\Spt}}^{\rmG}(\k_{\rm mot})$ to a (stable) weak-equivalence in ${\widetilde {\Spt}}^{\rmG}({\widetilde {\rmB}})$.
 \vskip .2cm
{\it Case 2}: Next assume that ${\group}$ is {\it not necessarily special, in which case we will assume the base field $k$ is infinite} to avoid
the issues discussed in \cite[ Example 2.10, 4.2]{MV}.  Observe that the list of non-special linear algebraic groups includes all the linear algebraic groups
 such as all finite groups, ${\rm PGL}_n$, ${\rm O}(n)$, $n \ge 1$ etc.  Recall $\Spc_*(\k_{et})$ denotes the ${\mathbb A}^1$-localized category of pointed simplicial presheaves on the big \'etale site
 $\Sm_{\k, et}$.
  Let  ${\rm BG}$ denote the simplicial classifying space of $\rmG$ viewed as a simplicial presheaf on the big \'etale site $\Sm_{\k, et}$ and let ${\widetilde {{\BG}^{\it gm,m}_{et}}}$ denote the scheme ${\widetilde {{\BG }^{\it gm,m}}}$ viewed as a simplicial presheaf on the big \'etale site $\Sm_{\k, et}$. Then the first
  observation we make is that  one obtains the weak-equivalence
  \be \begin{equation}
       \label{BG.1}
       {\rm BG} \simeq \colimm {\widetilde {{\BG}^{\it gm,m}_{et}}} 
      \end{equation} \ee
\vskip .1cm \noindent
 in $\Spc_*(\k_{et})$. To prove this one may proceed as follows. Either one may adopt 
 the same arguments as in \cite[p. 131 and  Lemma 2.5, Proposition 2.6 in 4.2]{MV} or consider the diagram:
 \be \begin{equation}
  \xymatrix{&{\EG \times_{{\rmG}} \EG ^{\it gm} } \ar@<1ex>[ld]_{p_1} \ar@<-1ex>[rd]^{p_2}\\
{\BG}   && {\BG^{\it gm}}.}
     \end{equation} \ee
\vskip .2cm \noindent
 Then, one may observe that the fibers of both maps $p_1$ and $p_2$ over a strictly Hensel ring are acyclic: the fibers
 of $p_1$ are acyclic because we have inverted ${\mathbb A}^1$ (and therefore, $\EG^{\it gm}$ is acyclic), and the fibers of $p_2$ are acyclic because they are the
 simplicial $\EG$. Thus $p_1$ and $p_2$ induce weak-equivalences of the corresponding simplicial sheaves.
 (See \cite[Theorem 1.5]{J22} for a similar argument at the level of equivariant derived categories.)
 Let $\epsilon: \Sm_{\k, et} \ra \Sm_{\k,Nis}$ denote the map of sites from the big \'etale site of
$\rmS= Spec  \, \k$ to the big Nisnevich site of $\rmS$. It follows therefore that one obtains the identification 
\be \begin{equation}
       \label{BG.2}
       \rmR\epsilon_*({\rm BG}) \simeq \rmR\epsilon_*( \colimm {\widetilde {{\BG} ^{\it gm,m}_{et}}}) \footnote{A main result of \cite[Proposition 2.6, p. 135]{MV} is that the term on the right is weakly-equivalent to $
        \epsilon_*(\colimm{\widetilde {\BG^{\it gm,m}}}) =\colimm \epsilon_*({\widetilde {\BG^{\it gm,m}}})$.}
\end{equation} \ee
\vskip .1cm \noindent
 in $\Spc_*(\k_{\rm mot})$. (Here we will use the injective model structure on simplicial presheaves prior to
 ${\mathbb A}^1$-localization: see ~\ref{inj.model}.) In this case, the construction ~\eqref{Borel.1} is replaced by:
 \be \begin{align}
   \label{et.case.Borel.const.1}
   \itX \mapsto \rmR\epsilon_*({\widetilde {\rmE}} \times_{\rmG}^{\it et} (a \circ \epsilon^*)(\itX)), &\quad \Spc_*^{\rmG}(\k_{\rm mot}) {\overset {a \circ \epsilon^*} \ra} \Spc_*^{\rmG}(\k_{et}) {\overset {{\widetilde {\rmE}} \times_{\rmG}^{\it et}(\quad)} \ra} \Spc_*^{\rmG}(\k_{et}) {\overset {\rmR\epsilon_*}\ra } {\Spc}_*(\rmR\epsilon_*({\widetilde {\rmB}_{et}}))
   \end{align} \ee
Here we have adopted the following conventions: the functor $\epsilon^*$ sends a simplicial presheaf on the big Nisnevich site $\Sm_{\k,Nis}$ to
a simplicial presheaf on the big \'etale site $\Sm_{\k,et}$, and the functor $a$ sends a simplicial presheaf on the big \'etale site $\Sm_{\k,et}$
to its associated sheaf on the same site, and 
the superscript $et$ denotes the fact we are {\it taking quotient sheaves on the \'etale site}. 
${\Spc}_*(\rmR\epsilon_*({\widetilde {{\rmB}_{et}}}))$
 denotes the category of simplicial presheaves on 
$\Sm_{\k, Nis}$ pointed over the simplicial presheaf $\rmR\epsilon_*({\widetilde {{\rmB}_{et}}})$. 
\vskip .1cm
Clearly this
extends to a functor
\be \begin{align}
   \label{et.case.Borel.const.2}
\X \mapsto  \rmR\epsilon_*({\widetilde {\rmE}} \times_{\rmG}^{\it et} (a \circ \epsilon^* )(\X)), &\quad \Spt^{\rmG}(\k_{\rm mot}) {\overset { a \circ \epsilon^*} \ra} \Spt^{\rmG}(\k_{et}) {\overset {\rmR\epsilon_*\circ{\widetilde {\rmE}} \times_{\rmG}^{\it et}(\quad)} \longrightarrow} {\widetilde {\Spt}}^{\rmG}(\rmR\epsilon_*({\widetilde {\rmB}_{et}}))  
\end{align} \ee
where ${\widetilde {\Spt}}^{\rmG}(\rmR\epsilon_*({\widetilde {\rmB}_{et}})) =[\Sph^{\rmG}, {\Spc}_*(\rmR\epsilon_*({\widetilde {{\rmB}_{et}}}))]$.
\vskip .2cm
If $\{\rmU_i| i \in I\}$ is an \'etale cover of ${\widetilde {\rmB}}$ over which ${\widetilde {\rmE}}$
is trivial, the same argument as above shows that 
\[({\widetilde {\rmE }} \times_{\rmG}^{et} (a \circ \epsilon^*)(\X))_{|{\rmU}_i}= \rmU_i \times (a \circ \epsilon^*)(\X),\]
so that
the functor $\X \mapsto {\widetilde {\rmE}} \times_{\rmG}^{\it et} (a \circ \epsilon^*)( \X)$ 
sends a $\rmG$-equivariant map $\alpha:\X \ra \Y$ for which ${\tilde \rmU}(\alpha)$ is a (stable) weak-equivalence in ${\widetilde \Spt}^{\rmG}(\k_{\rm mot})$ to a (stable) weak-equivalence in ${\widetilde \Spt}^{\rmG}({{\widetilde \rmB }_{et}})$.  
 Therefore, the functor $\X \mapsto \rmR\epsilon_*({\widetilde {\rmE}} \times_{\rmG}^{\it et} (a _{et}(\X)))$
sends a $\rmG$-equivariant map $\alpha:\X \ra \Y$ for which ${\tilde \rmU}(\alpha)$ is a (stable) weak-equivalence in ${\widetilde \Spt}^{\rmG}(\k_{\rm mot})$ to a (stable) weak-equivalence in ${\widetilde \Spt}^{\rmG}({ \rmR\epsilon_*({\widetilde \rmB}_{et})  }).$
\vskip .1cm
In case $\itX$ is already a sheaf on the big \'etale site, $(a \circ \epsilon^*)(\itX) = \itX$ and therefore, we may replace $(a \circ \epsilon^*)(\itX)$ in ~\eqref{et.case.Borel.const.1} 
by just $\itX$ in the definition of the Borel construction.  (This applies to the case where $\itX = \rmX$ is a scheme.)
\vskip .1cm
\begin{terminology}
 \label{Borel.not}
 Throughout the remainder of the paper, we will abbreviate the functor in ~\eqref{et.case.Borel.const.1} (~\eqref{et.case.Borel.const.2}) by
 \[\itX \mapsto \rmR\epsilon_*({\widetilde {\rmE}} \times_{\rmG}^{\it et} \itX),  \itX \in \Spc_*^{\rmG}(\k_{\rm mot}), \quad (\X \mapsto  \rmR\epsilon_*({\widetilde {\rmE}} \times_{\rmG}^{\it et} \X), \X \in \Spt^{\rmG}(\k_{\rm mot}), \res) .\]
\end{terminology}
\vskip .2cm
Though there is a discussion of the classifying spaces of linear algebraic groups in \cite[4.2]{MV}, it lacks a corresponding discussion
 on the Borel construction $\EG^{\it gm,m}\times_{\rmG} \rmX$, for $\rmX$ a smooth scheme. We complete our discussion, by providing a comparison of $\EG^{\it gm,m}\times_{\rmG} X$
with $\rmR\epsilon_*({\widetilde {{\EG }^{\it gm,m }}} \times_{\rmG}^{\it et} \rmX)$ when $\rmX$ is a smooth scheme. We first replace 
$\colimm {\BG}^{\it gm,m}_{et}$ and $\colimm\EG_{et}^{\it gm,m}\times_{\rmG}^{et} X$ by fibrant simplicial presheaves ${\widehat {\BG}}_{et}$ and
${\widehat {\EG_{et}\times_{\rmG}^{et} X}}$, fibrant in $\Spc_*(\k_{et})$,  so that the induced map ${\widehat {\EG_{et}\times_{\rmG}^{et} X}} \ra
 {\widehat {\BG}}_{et}$ is a fibration with fiber $\hat \rmX$, which is a fibrant replacement for $\rmX$.
 Let $\rmU_{\infty} = \colimm \EG^{\it gm,m}$. Now one forms the cartesian square in $\Spc_*(\k_{et})$:
\be \begin{equation}
     \xymatrix{{{\rmE(\rmU_{\infty}, \rmG)_{et}} \times_{\rmG}^{et}\hat \rmX} \ar@<1ex>[r] \ar@<1ex>[d] & {\widehat {{\EG}_{et} \times_{\rmG}^{et}\rmX}} \ar@<1ex>[d]\\
               {\rmB(\rmU_{\infty}, \rmG)_{et}} \ar@<1ex>[r] & {{\widehat {\BG}}_{et}}.}
    \end{equation} \ee
\vskip .1cm \noindent
Here $\rmE(\rmU_{\infty}, \rmG)_{et}$ is the \'etale simplicial presheaf given in degree $n$ by $\rmU_{\infty}^{n+1}$, and with the 
structure maps provided by the projections of $\rmU_{\infty}^m $ to the various factors $\rmU_{\infty}$ and by the diagonal maps
$\rmU_{\infty} \ra \rmU_{\infty}^m$. $\rmB(\rmU_{\infty}, \rmG)_{et} = \rmE(\rmU_{\infty}, \rmG)_{et}/G$. 
This square remains a cartesian square on applying the push-forward $\epsilon_*$ to the Nisnevich site. 
\cite[Lemma 2.5, 4.2]{MV} shows that the resulting map in the bottom row is an isomorphism in $\Spc_*(\k_{\rm mot})$, so that so is the resulting map in 
the top row. Finally an argument exactly as on \cite[p. 136]{MV} shows that one obtains an identification 
$\epsilon_*({{\rmE(\rmU_{\infty}, \rmG)_{et}} \times_{\rmG}^{et}\hat \rmX}) \simeq \epsilon _*(\rmU_{\infty} \times_{\rmG}^{et}\hat \rmX) =  \epsilon_*(\colimm \EG^{\it gm,m}\times_{\rmG}^{et}\hat \rmX) $. Therefore,  we obtain the identification for a smooth scheme $\rmX$:
\be \begin{equation}
     \label{ident.Borel.const}
\rmR\epsilon_*( \colimm {\widetilde {{\EG }^{\it gm,m }}} \times_{\rmG}^{\it et} \rmX) \simeq \rmR\epsilon_*( \colimm {{{\EG }^{\it gm,m }}} \times_{\rmG}^{\it et} \rmX)=\epsilon_*({\widehat {{\EG}_{et} \times_{\rmG}^{et}\rmX}})  \simeq \epsilon_*(\colimm (\EG^{\it gm,m}\times_{\rmG}^{et} \hat \rmX)).
\end{equation} \ee
\vskip .2cm
Finally, for convenience in the following steps, { we will denote both the Borel constructions given in ~\eqref{Borel.2} and 
~\eqref{et.case.Borel.const.2} by the notation} $\X \mapsto {\widetilde {\rmE}}\times_{\rmG} \X$. Moreover, we will 
denote by ${\widetilde {\rmB}}$, the object denoted by this symbol in ~\eqref{Joun.def} when $\rmG$ is special, and the object $\rmR\epsilon_*({\widetilde {\rmB}}_{et})$ considered in 
~\eqref{et.case.Borel.const.1} when $\rmG$ is not special.

\subsection{\bf Construction of the transfer}
\label{Transf.const} \index{transfer}
Next we proceed to construct the transfer as a stable map, i.e. a map in $\SH(k)$, when ${\Sigma^{\infty}_{\T}}\rmX_+$ is dualizable in $\SH(k)$ and $\rmG$ is special (and a variant of this map when $\rmG$ is non-special):
\be \begin{equation}
     \label{transfer.1}
\tr(\rmf_{\rmY}): {\Sigma^{\infty}_{\T}}({\widetilde {\rmE}}\times_{\rmG} \rmY)_+ \ra {\Sigma^{\infty}_{\T}}({\widetilde {\rmE}}\times_{\rmG}(\rmY\times \rmX))_+ \quad (\tr(\rmf_{\rmY}): ({\Sigma^{\infty}_{\T}}({\widetilde {{\EG}^{\it gm,m}}}\times_{\rmG} \rmY)_+ \ra  {\Sigma^{\infty}_{\T}}({\widetilde {{\EG}^{\it gm,m}}}\times_{\rmG} (\rmY \times \rmX))_+.
    \end{equation} \ee
\vskip .2cm \noindent
This will be constructed as a composition of several maps in $\Spt(\k_{\rm mot})$, with some of the maps going the wrong-way, and these wrong-way maps will all be
weak-equivalences in $\Spt(\k_{\rm mot})$. 
In case $\cE \wedge \rmX_+ \in \Spt(\k_{\rm mot}, \cE)$ is dualizable for a commutative ring spectrum $\cE^{\rmG} \in \Spt^{\group}(\k_{\rm mot})$ with $\cE= i^*(\tilde {\mathbb P}\tilde \rmU(\cE^{\rmG}))$, ($\cE \wedge \rmX_+ \in \Spt(\k_{et}, \cE)$ is dualizable 
for a commutative ring spectrum $\cE^{\rmG} \in \Spt^{\group}(\k_{et})$, so that $\cE$ is $\ell$-complete for some prime $\ell \ne char(k)$, \res) the transfer we
obtain will be of the following form when $\rmG$ is special (and a variant  of this map when $\rmG$ is non-special):
\be \begin{equation}
     \label{transfer.2}
\tr(\rmf_{\rmY}): \cE \wedge ({\widetilde {\rmE}}\times_{\rmG} \rmY)_+ \ra \cE \wedge ({\widetilde {\rmE}}\times_{\rmG} (\rmY \times \rmX))_+ \quad (\tr(\rmf_{\rmY}): \cE \wedge ({\widetilde {{\EG }^{\it gm,m}}}\times_{\rmG} \rmY)_+ \ra \cE \wedge ({\widetilde {{\EG }^{\it gm,m}}}\times_{\rmG}(\rmY \times \rmX))_+.)
\end{equation} \ee
\vskip .2cm \noindent
\begin{remark}
 \label{just.USptG}
The following remarks may provide some insight and motivation to the construction of the transfer discussed in Steps 1 through 5 below. We have tried to define a
transfer that depends only on the $\rmG$-object $\rmX$ and the $\rmG$-equivariant self-map $\rmf$ and which does {\it not} depend on any further
choices. This makes it necessary to start with the $\group$-equivariant pre-transfer as in ~\eqref{G.equiv.pretransfer}.
As a result, we are forced to make use of the framework of the category ${\widetilde {\Spt}}^{\rmG}(\k_{\rm mot})$. However, if one chooses to replace the
$\group$-equivariant sphere spectrum $\mbS^{\group}$ by just the suspension spectrum of the Thom-space $\rmT_{\rmV}$, for a 
fixed (but large enough) representation $\rmV$ of $\group$, then the use of the category ${\widetilde {\Spt}}^{\rmG}(\k_{\rm mot})$ could be circumvented
 by just using a variant of  Proposition ~\ref{functorial.rep.inhert.G.act} valid for suspension spectra. The construction
 of the transfer in \cite{BG75} in fact adopts this latter approach: in their framework, the co-evaluation map corresponds to a
 Thom-Pontrjagin collapse map associated to the Thom-space of a fixed $\group$-representation. Such an approach does not seem to
 work in general in the motivic context, though it could be made to work when $\rmX$ denotes a {\it projective smooth scheme}, provided one makes use of the Voevodsky collapse (see Appendix A, Definition \ref{virt.norm.bundle.proj}) in the place of the classical Thom-Pontrjagin collapse.
 \end{remark}
 \vskip .1cm

\vskip .2cm
{\it Step 1}. As the next step in the construction of the transfer map $\tr(\rmf_{\rmY})$, we start with the 
$\group$-equivariant pre-transfer 
 $\tr'^{\rmG}(\rmf_{\rmY+})$ in ~\eqref{G.equiv.pretransfer} to obtain the stable map over ${\widetilde {\rmE}}_{\rmY}$, i.e., as a composition of several maps in ${\widetilde \Spt}^{\rmG}({{\widetilde {\rmE}}}_{\rmY})$, where the wrong-way maps are all weak-equivalences.
\be \begin{equation}
     \label{transf.step.1.0}
{\widetilde {{\rmE}}} \times_{\rmG}(\rmY \times  \mbS^{\group}) {\overset {id\times_{\rmG} tr'^{\group}(\rmf_{\rmY+})} \ra}{\widetilde {{\rmE}}} \times_{\rmG}(\rmY \times (\mbS^{\group} \wedge \rmX_+)).  
\end{equation} \ee
    \vskip .2cm \noindent
(Here we are making use of the observation that the above Borel construction 
preserves weak-equivalences as observed in the discussion on the Borel construction, so that we can suppress the fact that the above map is in fact a
composition of several maps, some of which go the wrong-way as observed in ~\eqref{G.equiv.pretransfer}.) On applying the construction ${\widetilde {{\rmE}}} \times_{\rmG}(\quad )$ with a ${\group}$-equivariant ring spectrum $\cE^{\rmG}$ 
(as in ~\eqref{choice.ring.spectra}) in the place of $\mbS^{\group}$, 
 the resulting stable map takes on the 
form:
\be \begin{equation}
     \label{transf.step.1.1}
{\widetilde {\rmE}} \times_{\rmG} (\rmY \times \cE^{\rmG}) {\overset {id\times_{\rmG} tr'^{\group}(\rmf_{\rmY+})} \ra} {\widetilde {\rmE}} \times_{\rmG}(\rmY \times (\cE^{\rmG}\wedge \rmX_+)).
\end{equation} \ee
\begin{remark}
	\label{insight.rem.steps}
 The remaining steps in the construction of the transfer may be easily explained by fact that the sphere spectrum $\mbS^{\rmG}$ and the ring spectrum  $\cE^{\rmG}$ appearing above
  have non-trivial actions by $\rmG$, so that neither the source nor the target of the maps in ~\eqref{transf.step.1.0} and
  ~\eqref{transf.step.1.1} will become suspension spectra of ${\widetilde {\rmE}}_{\rmY}$ or ${\widetilde {\rmE}}\times_{\rmG}(\rmY \times \rmX)_+$ without 
  the considerable efforts in the remaining steps. We will discuss the remaining steps in detail only for the sphere spectrum $\mbS^{\rmG}$. This suffices, since the
   only other ring spectra $\cE^{\rmG}$ we consider will be restricted to those appearing in the list in ~\eqref{choice.ring.spectra}.
\end{remark}
\vskip .2cm \noindent

\vskip .2cm \noindent
{\it Step 2}. Next let $\rmV$ denote a fixed (but arbitrary) finite dimensional representation of the group $\rmG$.\footnote{Here we use $\rmV$ to
denote both the representation of $\rmG$ and the corresponding symmetric algebra over $k$.} 
At this point we need to briefly consider two cases, (a) where $\rmG$ is special and (b) where it is not. In case (a), it should be
clear that 
\be \begin{equation}
     \label{psi}
{\widetilde {\rmE}}\times_{\rmG} \rmV \mbox{ is a vector bundle } \xi^{\rmV} \mbox{ on the affine scheme } {\widetilde  {\rmB}},
\end{equation} \ee
where the quotient construction is done as in ~\eqref{Borel.1}, that is on the Zariski site. In case (b), one considers instead:
\be \begin{equation}
     \label{psi.1}
{\widetilde {\rmE}}\times_{\rmG}^{et} \rmV,
\end{equation} \ee
where the quotient is taken on the \'etale topology. Apriori, this is a vector bundle that is locally trivial on the \'etale topology of
$\widetilde \rmB$. But  any such vector bundle corresponds
to a ${\rm GL}_n$-torsor on the \'etale topology of $\widetilde \rmB$, and hence (by Hilbert's theorem 90: see \cite[Chapter III, proposition 4.9]{Mil}), 
is in fact locally trivial on the Zariski topology of $\widetilde \rmB$. We will denote this vector bundle also  by $\xi^{\rmV}$.
\vskip .2cm
Since ${\widetilde  {\rmB}}$ is an affine scheme over $\Speck$, we can find a complimentary
vector bundle $\eta^{\rmV}$ on ${\widetilde  {\rmB}}$ so that 
\be \begin{equation}
     \label{eta.0}
\xi^{\rmV} \oplus \eta^{\rmV} \mbox{ is a trivial bundle over } {\widetilde  {\rmB}} \mbox{ and of rank } \rmN, \mbox{ for some integer } \rmN.
\end{equation} \ee
\vskip .1cm
For the remainder of this step, we will consider the case when ${\widetilde {\rmE}} = {\widetilde {{\EG}^{\it gm,m}}}$ and 
${\widetilde {\rmB}} = {\widetilde  {{\BG}^{\it gm,m }}}$, for a fixed integer $m$. We will denote the first by $\cE_m$ and the latter by $\cB_m$.
We will denote the vector bundle ${\widetilde {{\EG}^{\it gm,m}}}\times_{\rmG} \rmV$
(${\widetilde {{\EG}^{\it gm,m}}}\times_{\rmG}^{et} \rmV$) $\mbox{ on the affine scheme }  \cB_m={\widetilde  {{\BG}^{\it gm,m }}}$ by 
$ \xi_m^{\rmV}$. The complimentary vector bundle $\eta^{\rmV}$ chosen above will now denoted $\eta^{\rmV}_m$.
We proceed to show that we can choose the integer $\rmN$ independent of $m$, so that
a single choice of $\rmN$ will work for all $m$. Since the map ${\EG}^{\it gm,m} \ra {\BG}^{\it gm,m}$ is affine, one can readily see that the scheme ${\widetilde {{\EG}^{\it gm,m}}}$ is also an affine scheme.  Let $\rmR_m$ denote the co-ordinate ring of 
${\widetilde {{\EG}^{\it gm,m}}}$ and let $\rmR = \limm \rmR_m$. Under the correspondence between projective modules over $\rmR_m^{\rmG}$ and
vector bundles over ${\rm Spec} \, \rmR_m^{\rmG}$, $(\rmR_m{\underset {\k} \otimes}\rmV)^{\group}$ corresponds to ${\widetilde {{\EG}^{\it gm,m}}} \times_{\rmG} \rmV = \xi_m^{\rmV}$.
\vskip .2cm
We proceed to  show that 
\[(\rmR{\underset {\k} \otimes}\rmV)^{\group}\] 
is a { finitely generated projective module over the ring $\rmR^{\group}$}.  To see this, we proceed as follows.
Let $\rmI_m$ be the ideal defining ${\widetilde {{\BG}^{\it gm,m}}}$ as a closed
subscheme in ${\rm Spec}\, (\rmR^{\group}) $. Then, $(\rmR{\underset {\k} \otimes}\rmV)^{\group}/(\rmI_m {\underset {\rmR^{\group}} \otimes }(\rmR{\underset {\k} \otimes}\rmV)^{\group})$ corresponds to the vector bundle $ \xi_m^{\rmV}$, and therefore,  is a finitely generated projective $\rmR^{\group}/\rmI_m$-module. In fact, if ${\frak M}$ denotes a maximal ideal in the ring $\rmR^{\rmG}$ and $\bar \rmI_m$ denotes the image of the ideal $\rmI_m$
in the local ring $\rmR^{\rmG}_{({\frak M})}$, then one can see that the ranks of the
inverse system of free modules $\{(\rmR{\underset {\k} \otimes}\rmV)^{\group}_{({\frak M})}/(\bar \rmI_m {\underset {\rmR^{\group}} \otimes }(\rmR{\underset {\k} \otimes}\rmV)^{\group}_{({\frak M})})|m\}$ are the same finite integer given by the rank of $\rmV$. Therefore,
their inverse limit, which identifies with $(\rmR {\underset {\k} \otimes}\rmV)^{\rmG}_{({\frak M})}$ is a free $\rmR^{\rmG}_{({\frak M})}$-module of rank equal to the rank of $\rmV$. It follows that, $(\rmR {\underset {\k }\otimes}\rmV)^{\rmG}$ is a finitely generated projective module over the ring $\rmR^{\rmG}$.
\vskip .2cm
Therefore, there exists some finitely generated free $\rmR^{\group}$-module $\rmF$ (of rank $\rmN$) and a {\it split } surjection
 \be \begin{equation}
   \label{zeta}
   \zeta: \rmF \twoheadrightarrow  (\rmR{\underset {\k} \otimes}\rmV)^{\group}.
 \end{equation} \ee
 Then one sees that the induced maps
\be \begin{equation}
   \label{eta.1}
   \zeta/\rmI_m: \rmF/(\rmI_m {\underset {\rmR^{\group}} \otimes }\rmF) \twoheadrightarrow  (\rmR{\underset {\k} \otimes}\rmV)^{\group}/(\rmI_m {\underset {\rmR^{\group}} \otimes }(\rmR{\underset {\k} \otimes}\rmV)^{\group})
 \end{equation} \ee
 are also split surjections for each $m$, and these splittings are in fact compatible, as they are all induced by the splitting to the map in ~\eqref{zeta}.  Therefore, 
 we obtain a compatible collection of complements to the inverse system of bundles  $ \xi_m^{\rmV}$
 in the trivial bundle of rank $\rmN$ over ${\widetilde {\BG^{\it gm,m}}}$, compatible as $m$ varies. We denote the
  complement to  $ \xi_m^{\rmV}$ in the trivial bundle of rank $\rmN$ over 
  ${\widetilde {\BG^{\it gm,m}}}$ as $\eta_m^{\rmV}$.
 
 \vskip .2cm
 Next we will consider {\it the case the group ${\group}$ is special}, in which the case the arguments in the 
 following paragraph hold.
 Denoting by $\rmT_{\rmV}$ the Thom-space of the representation $\rmV$, the bundle
${\widetilde {{\EG}^{\it gm,m}}}\times_{\rmG}\rmT_{\rmV}$ is a sphere-bundle over $\cB_m$, which will be denoted $\rmS(\xi_m^{\rmV} \oplus 1)$
in the terminology of  Appendix, ~\eqref{props.Thom.spaces}. Similarly $\rmS(\eta_m^{\rmV} \oplus 1)$ denotes the corresponding sphere bundle over $\cB_m$.
Now Lemma ~\ref{AppA.fiberwise.join.2} (see the Appendix)  shows that one obtains the identification:
\be \begin{equation}
     \label{Ch4.Thom.sp.1}
\rmS(\xi_m^{\rmV} \oplus 1) \wedge ^{\cB_m} \rmS(\eta_m^{\rmV} \oplus 1) \simeq \rmS(\xi_m^{\rmV} \oplus \eta_m^{\rmV} \oplus 1).
    \end{equation} \ee
Observe that there is a canonical section $s_{\xi^{\rmV}}:  \cB_m \ra {\widetilde {{\EG}^{\it gm,m}}}\times_{\rmG}\rmT_{\rmV} = \rmS(\xi_m^{\rmV} \oplus 1)$,
 and a canonical section $s_{\eta^{\rmV}}:\cB_m \ra   \rmS(\eta_m^{\rmV} \oplus 1)$, which together define a 
section $s_m:  \cB_m \ra \rmS(\xi_m^{\rmV} \oplus 1) \wedge ^{\cB_m} \rmS(\eta_m^{\rmV} \oplus 1)$ of pointed simplicial presheaves
 over $\cB_m={\widetilde {{\BG}^{\it gm,m}}}$. Then 
  the quotient $(\rmS(\xi_m^{\rmV} \oplus 1) \wedge ^{\cB_m} \rmS(\eta _m^{\rmV} \oplus 1))/ s(\cB_m)$
identifies with the Thom-space of the bundle $\xi_m^{\rmV}\oplus \eta_m^{\rmV}$. Since $\eta_m^{\rmV}$ was chosen to be a vector bundle complimentary to
$\xi_m^{\rmV}$, $\xi_m^{\rmV} \oplus \eta_m^{\rmV}$ is a trivial bundle (of rank $\rmN$) so that  the above Thom-space identifies with
${\T}^{\wedge \rmN} ({\widetilde {\BG}^{\it gm,m}})_+$.  Moreover, this holds independent of $m$.
\vskip .2cm 
In case the group ${\group}$ is {\it not special}, one has to replace ${\widetilde {{\EG}^{\it gm,m}}}\times_{\rmG}\rmT_{\rmV}$ by
${\widetilde {{\EG}^{\it gm,m}}}{\underset {\rmG} \times^{et}} \epsilon^*(\rmT_{\rmV})$. This identifies with $\rmS(\epsilon^*(\xi_m^{\rmV} \oplus 1)))$.
Now one has to take the smash product of the above object with $\rmS(\epsilon^*(\eta_m^{\rmV} \oplus 1))$ over $\epsilon^*(\cB_m)$. This will 
identify with $\rmS(\epsilon^*(\xi_m^{\rmV} \oplus \eta_m^{\rmV} \oplus 1))$. Since $\eta^{\rmV}$ was chosen to be complementary to $\xi^{\rmV}$, it follows that the bundle
$\xi^{\rmV} \oplus \eta^{\rmV}$ is a trivial vector bundle of rank $\rmN$, so that $\rmS(\epsilon^*(\xi_m^{\rmV} \oplus \eta_m^{\rmV} \oplus 1))
\cong \epsilon^*(\T^{\wedge \rmN}) \times \cB_m$. Then
 one applies $\rmR\epsilon_*$ to the resulting object to obtain a pointed simplicial presheaf over $\rmR\epsilon_*(\cB_m)$. 
 This identifies with $\rmR\epsilon_*(\epsilon^*(\T^{\wedge \rmN})) \times R\epsilon_*(({\widetilde {\BG}^{\it gm,m}}))$.
Finally one has to collapse the corresponding section to obtain $\rmR\epsilon_*(\epsilon^*(\T^{\wedge \rmN})) \wedge  (R\epsilon_*(({\widetilde {\BG}^{\it gm,m}})))_+$.

\vskip .2cm \noindent
{\it  Let $\pi_{\rmY}$ denote either of the two projections ${\widetilde {\rmE}}\times_{\rmG}(\rmY \times \rmX) \ra {\widetilde {\rmE}}\times_{\rmG}(\rmY)$ or $ \cE_m ={\widetilde {{\EG}^{\it gm,m}}}\times_{\rmG} (\rmY \times \rmX) \ra {\widetilde {{\EG}^{\it gm,m}}}\times_{\rmG} \rmY= \cB_m$. 
Since the second case is subsumed by the first, we will only discuss the first case explicitly in steps 3 through 5.}
\vskip .1cm
{\it Step 3}. First we will again assume that the group-scheme ${\group}$ is special. 
Now observe that the sphere bundle ${\widetilde {\rmE}}\times_{\rmG}(\rmY \times  \rmT_{\rmV})$ identifies with the
pullback $\rmS(\pi^*(\xi^{\rmV}) \oplus 1) = \pi^*(\rmS(\xi^{\rmV}\oplus 1))$ and the sphere bundle ${\widetilde {\rmE}}\times_{\rmG}(\rmY \times \rmX \times \rmT_{\rmV})$ identifies with the
pullback $\rmS(\pi_{\rmY}^*\pi^*(\xi^{\rmV}) \oplus 1) = \pi_{\rmY}^*\pi^*(\rmS(\xi^{\rmV}\oplus 1))$, where the maps $\pi$ and $\pi_{\rmY}$ are as in ~\eqref{Joun.def}.
Next consider
\[\rmS(\pi_{\rmY}^*\pi^*(\xi^{\rmV}) \oplus 1)\wedge^{\tilde \rmE_{\rmY}} \rmS(\pi_{\rmY}^*\pi_{}^*(\eta^{\rmV}) \oplus 1) = \pi_{\rmY}^*\pi^*(\rmS(\xi^{\rmV} \oplus 1)) \wedge^{\widetilde \rmE_{\rmY}} \pi_{\rmY}^*\pi^*(\rmS(\eta^{\rmV} \oplus 1)).\]
This is a sphere bundle over $\widetilde \rmE_{\rmY}$ and it has a canonical section, which we will denote $\sigma$, collapsing which provides the Thom-space of the
bundle $\rmS(\pi_{\rmY}^*\pi^*(\xi^{\rmV} \oplus \eta^{\rmV} \oplus 1))$. Since $\eta^{\rmV}$ was chosen to be complementary to $\xi^{\rmV}$, it follows that the bundle
$\pi_{\rmY}^*\pi^*(\xi^{\rmV} \oplus \eta^{\rmV})$ is trivial, so that the resulting Thom-space identifies with $\T^{\wedge \rmN}({\widetilde {\rmE}}\times_{\rmG}(\rmY \times \rmX))_+$.
\vskip .1cm 
When the group-scheme ${\group}$ is {\it not special}, one adopts an argument as in the last paragraph of Step 2 to obtain
a corresponding result.
\vskip .2cm \noindent
{\it Step 4}. Observe that there is section $t': {\widetilde {\rmE}}_{\rmY} \ra {\widetilde {\rmE}}\times_{\rmG}(\rmY \times (\rmX_+ \wedge \rmT_{\rmV}))$.
Combining that with the canonical section ${\widetilde \rmE}_{\rmY} \ra \rmS(\pi^*(\eta^{\rmV} \oplus 1))$ defines a section $t: {\widetilde {\rmE}}_{\rmY} \ra ({\widetilde {\rmE}}\times_{\rmG}(\rmY \times (\rmX_+ \wedge \rmT_{\rmV})) \wedge ^{{\widetilde \rmE}_{\rmY}}S(\pi^*(\eta^{\rmV} \oplus 1))$.
Now a key observation is that $({\widetilde {\rmE}}{\underset {\rmG }\times} (\rmY \times (\rmX_+ \wedge \rmT_{\rmV}))\wedge ^{{\widetilde \rmE}_{\rmY}} \rmS(\pi^*(\eta^{\rmV} \oplus1))$ 
is an object defined over ${\widetilde \rmE}_{\rmY}$  and 
that collapsing the section $t$ identifies the resulting object with 
$(\rmS(\pi_{\rmY}^*\pi^*(\xi^{\rmV} \oplus 1)) \wedge ^{{\widetilde \rmE}_{\rmY\times \rmX}} \rmS(\pi_{\rmY}^*\pi^*(\eta^{\rmV} \oplus 1)))/\sigma({\widetilde \rmE}_{\rmY\times \rmX})$, where $\sigma:{\widetilde \rmE}_{\rmY\times \rmX} \ra
\rmS(\pi_{\rmY}^*\pi^*(\xi^{\rmV} \oplus 1)) \wedge ^{{\widetilde \rmE}_{\rmY}} \rmS(\pi_{\rmY}^*\pi^*(\eta^{\rmV} \oplus 1))$ is the canonical section. (See \cite[(3.7) and (3.8)]{BG75} for the classical case.)
\vskip .2cm
One may see this as follows, first under the assumption that the group-scheme ${\group}$ is special.
Assume that $\{\rmU_i|i\}$ is a Zariski open cover of ${\widetilde  {\rmB}}$ over which the ${\group}$-torsor
$\rmp: {\widetilde  {\rmE}} \ra {\widetilde  {\rmB}}$ trivializes.  
$(\rmS(\pi_{\rmY}^*\pi^*(\xi^{\rmV} \oplus 1))_{|\rmU_i}$ now is of the form: $\rmU_i \times (\rmY \times \rmX \times \rmT_{\rmV}) \ra \rmU_i \times \rmY \times \rmX $.
 We may assume that the vector
bundle $\eta^{\rmV}$ also trivializes over the cover $\{\rmU_i|i\}$.  Then
$(\rmS(\pi_{\rmY}^*\pi^*(\xi^{\rmV} \oplus 1)) \wedge ^{{\widetilde \rmE}_{\rmY\times \rmX}} S(\pi_{\rmY}^*\pi^*(\eta^{\rmV} \oplus 1)))_{|\rmU_i}= \rmU_i \times ((\rmY \times \rmX) \times (\rmT_{\rmV} \wedge \rmT_{\rmW}))$,
where $\rmW$ corresponds to the fibers of the vector bundle $\eta^{\rmV}$. The section 
$\sigma_{|\rmU_i}:{\widetilde \rmE}_{\rmY\times \rmX|\rmU_i} \ra (\rmS(\pi_{\rmY}^*\pi^*(\xi^{\rmV} \oplus 1)) \wedge ^{{\widetilde \rmE}_{\rmY \times \rmX}} \rmS(\pi_{\rmY}^*\pi^*(\eta^{\rmV} \oplus 1)))_{|\rmU_i}$ now
corresponds to the canonical section $\rmU_i \times \rmY \times \rmX \ra \rmU_i \times ((\rmY \times \rmX) \times (\rmT_{\rmV} \wedge \rmT_{\rmW}))$. 
Intermediate to collapsing the section $\sigma$ is to take the pushout of 
\be \begin{equation}
\label{intermed.quot}
{\widetilde {\rmE}}_{\rmY} \leftarrow {\widetilde \rmE}\times _{\rmG} (\rmY \times \rmX) \rightarrow {\rmS(\pi_{\rmY}^*\pi^*(\xi^{\rmV} \oplus 1)) \wedge ^{{\widetilde \rmE}_{\rmY\times \rmX}} S(\pi_{\rmY}^*\pi^*(\eta^{\rmV} \oplus 1)))}.
\end{equation} \ee
Over $\rmU_i$, this corresponds to taking the pushout of $\rmU_i \times \rmY \leftarrow \rmU_i \times (\rmY \times \rmX) \ra \rmU_i \times ((\rmY \times \rmX) \times (\rmT_{\rmV} \wedge \rmT_{\rmW}))$. The resulting pushout then identifies
with $\rmU_i \times \rmY \times (\rmX_+ \wedge (\rmT_{\rmV} \wedge \rmT_{\rmW}))$, which in fact identities with 
$({\widetilde {\rmE}}\times_{\rmG} (\rmY \times (\rmX_+ \wedge \rmT_{\rmV})))\wedge ^{{\widetilde \rmE}_{\rmY}} \rmS(\pi^*(\eta^{\rmV} \oplus1)))_{|\rmU_i}$. 
\vskip .2cm
Observe that collapsing the section $\sigma$ can be done in two stages, by first taking the pushout in ~\eqref{intermed.quot} and then by
collapsing the resulting section from ${\widetilde \rmE}_{\rmY}$. These complete the verification of the observation in Step 4, at least in the case the group-scheme ${\group}$ is special. When ${\group}$ is not special, one  adopts a similar argument using an
\'etale cover $\{\rmU_i|i \in I\}$  of ${\widetilde {\rmB}}$  over which ${\widetilde {\rmE}} \ra 
{\widetilde {\rmB}}$ is trivial.
\vskip .2cm \noindent
{\it Step 5}. Let $s: {\widetilde \rmE}_{\rmY} \ra {\widetilde \rmE}\times_{\rmG} (\rmY \times \rmT_{\rmV})\wedge ^{{\widetilde E}_{\rmY}}\rmS(\pi^*(\eta^{\rmV} \oplus 1))$
denote the canonical section. 
 Then, we proceed to show that 
the sections $s$ and $t$  are  compatible in the
sense that the diagram
\be \begin{equation}
     \label{compat.sections.V}
\xymatrix{{{\widetilde \rmE}_{\rmY}}  \ar@<1ex>[r]^(.2){s}  \ar@<1ex>[dr]_(.3){t} & {{\widetilde {\rmE}}\times_{\rmG} (\rmY \times \rmT_{\rmV}) \wedge^{{\widetilde \rmE}_{\rmY}} S(\pi^*(\eta^{\rmV} \oplus 1))} \ar@<1ex>[d]^{(id\times_{\rmG}tr'^{\group}(\rmf_{\rmY+})(T_{\rmV}))\wedge^{{\widetilde \rmE}_{\rmY}} id}\\
             & {{\widetilde {{\rmE}}}\times_{\rmG} (\rmY \times (\rmX_+ \wedge \rmT_{\rmV}))\wedge^{{\widetilde \rmE}_{\rmY}} S(\pi^*(\eta^{\rmV} \oplus 1)))}}
    \end{equation} \ee
\vskip .2cm \noindent
commutes, that is, in the sense discussed next. Here $(id{\underset {\rmG}\times}tr_{\group}(\rmf_{\rmY})'(T_{\rmV})$ is the component of
the map of spectra $id\times_{\rmG}tr_{\group}(\rmf_{\rmY})'$ indexed by $\rmT_{\rmV}$. 
One may break this map into a sequence of maps 
\be \begin{multline}
     \begin{split}
   \label{zigzag.tr}
\Y_0(\rmT_{\rmV})={\widetilde {\rmE}}\times_{\rmG}(\rmY \times \rmT_{\rmV}) \ra \Y_1(\rmT_{\rmV})={\widetilde {\rmE}}\times_{\rmG} (\rmY \times \X_1(\rmT_{\rmV})) \leftarrow 
	\Y_2(\rmT_{\rmV})={\widetilde {{\rmE}}}\times_{\rmG}(\rmY  \times  \X_2(\rmT_{\rmV}) )\\
	\ra \Y_3(\rmT_{\rmV})={\widetilde {{\rmE}}}\times_{\rmG}(\rmY \times (\rmX_+ \wedge \rmT_{\rmV})), \end{split}
\end{multline} \ee
	where the maps $\{\rmT_{\rmV} \ra \X_1(\rmT_{\rmV}) \leftarrow \X_2(\rmT_{\rmV}) \ra \rmX_+ \wedge \rmT_{\rmV}|\rmV\}$ define the $\rmG$-equivariant pre-transfer considered in 
	 ~\eqref{G.equiv.pretransfer}. Observe that each of the objects in ~\eqref{zigzag.tr}
	 is pointed over ${\widetilde {\rmE}}_{\rmY}$. (When $\rmG$ is non-special, the quotient sheaves in 
	 the diagram ~\eqref{zigzag.tr}, and in the discussion below, are all taken in the \'etale topology on ${\widetilde {\rmE}}_{\rmY}$ 
	 and one will have to replace the diagram in ~\eqref{compat.sections.V} with $\rmR\epsilon_*$ applied to all the terms there.)  One may observe that 
	 the corresponding sections from
	 ${\widetilde \rmE}_{\rmY}$ are all compatible as the group action leaves the  base points of $\rmT_{\rmV},
	 \X_1(\rmT_{\rmV}), \X_2(\rmT_{\rmV})$ and $\rmX_+ \wedge {\rmT_{\rmV}}$ fixed. This results in the following commutative diagram 
	 over ${\widetilde \rmE}_{\rmY}$: 
 \fontsize{8}{12}	 
\be \begin{equation}
\label{zigzag.tr.1}
\xymatrix{{\Y_0(\rmT_{\rmV}) \wedge^{{\widetilde \rmE}_{\rmY}} S(\pi^*(\eta^{\rmV} \oplus 1))} \ar@<1ex>[r]& {\Y_1(\rmT_{\rmV}) \wedge^{{\widetilde \rmE}_{\rmY}} S(\pi^*(\eta^{\rmV} \oplus 1))} & {\Y_2(\rmT_{\rmV}) \wedge^{{\widetilde \rmE}_{\rmY}} S(\pi^*(\eta^{\rmV} \oplus 1))} \ar@<-1ex>[l] \ar@<1ex>[r]& {\Y_3(\rmT_{\rmV}) \wedge^{{\widetilde \rmE}_{\rmY}} S(\pi^*\eta^{\rmV} \oplus 1))} \\
	{{\widetilde \rmE}_{\rmY}} \ar@<1ex>[u]^{y_0=s} \ar@<1ex>[ur]^(.27){y_1}  \ar@<1ex>[urr]^(.27){y_2} \ar@<1ex>[urrr] _(.4){y_3=t}}
\end{equation} \ee
\normalsize
\vskip .1cm \noindent
By the commutativity of the triangle in ~\eqref{compat.sections.V}, we mean the commutativity of all the corresponding triangles that make up the diagram in ~\eqref{zigzag.tr.1} and this is now clear in view of the above observations.
When ${\widetilde {\rmE}}_{\rmY}= \cE_{m, \rmY}= {\widetilde {{\EG}^{\it gm,m}}}\times_{\rmG}\rmY$, one may again observe that 
	 the corresponding sections from
	 $\cE_{m, \rmY}$ are all compatible as the group action leaves the  base points of $\rmT_{\rmV},
	  \wedge\X_1(\rmT_{\rmV}),  \X_2(\rmT_{\rmV})$ and $ \rmX_+ \wedge {\rmT_{\rmV}}$ fixed. This results in a corresponding diagram over each $\cE_{m, \rmY}$
	 and the arguments in Step 2 above show that such commutative triangles are compatible as $m$ varies.
\vskip .2cm	
 Moreover, the commutativity of the diagram
 ~\eqref{zigzag.tr.1} shows that there is an induced map on the quotients by the sections $y_i$, $i=0, 1,2, 3$. Observe that on taking smash product {\it over} ${\widetilde \rmE}_{\rmY}$
 with ${\widetilde {\rmE}}\times_{\rmG} ((\rmY \times \rmX)_+ \wedge \rmT_{\rmW})\wedge^{{\widetilde \rmE}_{\rmY}} S(\pi^*(\eta^{\rmW} \oplus 1)) = ({\widetilde \rmE}_{\rmY} \times \T^ {dim( \rmW)+ rank(\eta^{\rmW})})_+ $ , one obtains a map of the
 	diagram in ~\eqref{zigzag.tr.1} to the corresponding diagram with $\rmV\oplus \rmW$ in the place of $\rmV$. This observation shows that if we 
define spectra $\Z_i$, $i=0, 1,2, 3$ in $\Spt_S$ by
 \be \begin{equation}
   \Z_{i, \rmN_{\rmV}} = (\Y_i(\rmT_{\rmV}) \wedge^{{\widetilde \rmE}_{\rmY}} S(\pi^*(\eta^{\rmV} \oplus 1)))   /y_i({{\widetilde \rmE}_{\rmY}}), \rmN_{\rmV} = dim (\rmV) + rank(\eta^{\rmV})
 \end{equation} \ee
 and if $\rmN_{\rmW} = dim(\rmW) + rank(\eta^{\rmV})$, the smash product
 pairings $\T^{\rmN_{\rmW}} \wedge \Z_{i, \rmN_{\rmV}} \ra \Z_{i, \rmN_{\rmV\oplus \rmW}}$  are compatible with the maps between the $\Z_i$ considered above. (Note that these spectra are indexed by the integers $\{\rmN_{\rmV}| \rmV\}$ and not by all the non-negative integers. However, since $\{\rmN_{\rmV}| \rmV\}$ is cofinal in ${\mathbb N}$, this suffices.)
 One may also observe that
 the wrong-way map $\Z_2 \ra \Z_1$ is a stable equivalence.
 Therefore, collapsing out the sections $y_i$,  $i=0, 3$, then provides the stable map (which in fact is a composition of several maps, with the ones going in the wrong direction being stable weak-equivalences) 
\be \begin{equation}
\label{tr(f)G.2}
\tr(\rmf_{\rmY}): {\Sigma^{\infty}_{\T}} ({\widetilde {{\rmE}}}\times_{\rmG} \rmY)_+ \ra {\Sigma^{\infty}_{\T}} ({\widetilde {{\rmE}}}\times_{\rmG} (\rmY \times \rmX))_+, \quad \tr(\rmf_{\rmY})^m: {\Sigma^{\infty}_{\T}} ({\widetilde {{\EG}^{\it gm,m}}}\times_{\rmG} \rmY)_+ \ra {\Sigma^{\infty}_{\T}} ({\widetilde {{\EG}^{\it gm,m}}}\times_{\rmG} (\rmY \times \rmX))_+
\end{equation} \ee
in case the group $\rmG$ is special, and the following stable map (which in fact is a composition of several maps, with the ones going in the wrong direction being stable weak-equivalences) in case $\rmG$ is not special:
\be \begin{align}
    \label{tr(f)G.3}
 \tr(\rmf_{\rmY}): \rmR\epsilon_*(\epsilon^*{\mbS_{\k}}) \wedge  \rmR\epsilon_*({\widetilde {\rmE}}\times_{\rmG}^{et} \rmY)_+ &\ra \rmR\epsilon_*(\epsilon^*{\mbS_{\k}}) \wedge \rmR\epsilon_*({\widetilde {\rmE}}\times_{\rmG}^{et}(\rmY \times  \rmX))_+, \\
 \tr(\rmf_{\rmY})^m:  \rmR\epsilon_*(\epsilon^*{\mbS_{\k}}) \wedge \rmR\epsilon_*({\widetilde {{\EG}^{\it gm,m}}}\times_{\rmG}^{et} \rmY)_+ &\ra \rmR\epsilon_*(\epsilon^*{\mbS_{\k}}) \wedge \rmR\epsilon_*({\widetilde {{\EG}^{\it gm,m}}}\times_{\rmG}^{et}(\rmY \times  \rmX))_+.  \notag
\end{align}
\ee
\vskip .1cm \noindent
These maps are also compatible as $m$ varies, as observed above and in Step 2. The pairings $\T \wedge \rmR\epsilon_*(\epsilon^*(\T^{\wedge ^n}) \ra \rmR\epsilon_*(\epsilon^*(\T) \wedge \rmR\epsilon_*(\epsilon^*(\T^{\wedge n}) \ra\rmR\epsilon_*(\epsilon^*(\T ^{\wedge ^{n+1}}) $ shows
that $\rmR\epsilon_*(\epsilon^*{\mbS_{\k}})$ is indeed a motivic spectrum. 
\vskip .2cm 
 \begin{definition} (The transfer.) 
\label{transfer:def}
Therefore, taking the colimit over $m \ra \infty$, 
one obtains the following stable transfer map (in $\SH(k)$) on the Borel construction when  $\rmG$  is {\it special}: \index{Becker-Gottlieb transfer}
\be \begin{align}
     \label{tr(f)G.4}
     \tr(\rmf_{\rmY}): {\Sigma^{\infty}_{\T}}  ({\widetilde {\rmE}}\times_{\rmG} \rmY)_+ \ra {\Sigma^{\infty}_{\T}}  ({\widetilde {\rmE}}\times_{\rmG} (\rmY \times \rmX))_+, \\
     \tr(\rmf_{\rmY}): {\Sigma^{\infty}_{\T}}  ({\widetilde {{\EG}^{gm}}}\times_{\rmG} \rmY)_+ \ra {\Sigma^{\infty}_{\T}}  ({\widetilde {{\EG}^{gm}}}\times_{\rmG} (\rmY \times \rmX))_+.  \notag
\end{align} \ee
Next we consider the when the group $\rmG$ is non-special. One may observe from the commutative diagram ~\eqref{zigzag.tr.1}
that all the maps involved in the definition of the transfer maps $\widetilde \tr(\rmf_{\rmY})$ and $\widetilde \tr(\rmf_{\rmY})^m$ in ~\eqref{tr(f)G.3} 
are maps of module spectra over 
the motivic ring spectrum $\rmR\epsilon_*(\epsilon^*{\mbS_{\k}})$. Therefore, we will now define the transfer maps, when $\rmG$ is {\it non-special} to be
\be \begin{align}
     \label{tr(f)G.5}
\tr(\rmf_{\rmY}) = \rmR\epsilon_*(\epsilon^*{\mbS_{\k}}) \wedge  \rmR\epsilon_*({\widetilde {\rmE}}\times_{\rmG}^{et} \rmY)_+ &\ra \rmR\epsilon_*(\epsilon^*{\mbS_{\k}}) \wedge \rmR\epsilon_*({\widetilde {\rmE}}\times_{\rmG}^{et}(\rmY \times  \rmX))_+, \mbox{ and}\\
\tr(\rmf_{\rmY})= \colimm \tr(\rmf_{\rmY})^m: \rmR\epsilon_*(\epsilon^*{\mbS_{\k}}) \wedge \rmR\epsilon_*({\widetilde {{\EG}^{\it gm, m}}}\times_{\rmG}^{et} \rmY)_+ &\ra \rmR\epsilon_*(\epsilon^*{\mbS_{\k}}) \wedge \rmR\epsilon_*({\widetilde {{\EG}^{\it gm,m}}}\times_{\rmG}^{et}(\rmY \times  \rmX))_+.  \notag
\end{align} \ee 
\vskip .1cm \noindent
Henceforth we will let ${\widetilde {{\EG}^{gm}}}\times_{\rmG} \rmY =
\colimm \rmR\epsilon_*{\widetilde {{\EG}^{\it gm,m}}}\times_{\rmG} \rmY$ and ${\widetilde {{\EG}^{gm}}}\times_{\rmG}^{et} (\rmY \times \rmX) =
\colimm \rmR\epsilon_* {\widetilde {{\EG}^{\it gm,m}}}\times_{\rmG}(\rmY \times  \rmX)$.
\vskip .2cm
If $\cE^{\rmG}$ denotes a commutative ${\group}$-equivariant ring spectrum as in  ~\eqref{choice.ring.spectra}, $\cE= i^*(\tilde {\mathbb P}\tilde \rmU(\cE^{\rmG}))$ is the
corresponding ring spectrum in $\Spt$, and $\cE \wedge X_+$ is dualizable in $\SH(\k, \cE)$, 
 the same constructions applied to the ${\group}$-equivariant pre-transfer 
~\eqref{G.equiv.pretransfer} and making use of smashing with the spectrum $\cE$ in the place of smashing with $\mbS$ 
provides us with the transfer map (in $\SH(k, \cE)$):
\be \begin{align}
     \label{tr(f)G.E}
\tr(\rmf_{\rmY})_{\cE}:\cE \wedge  ({\widetilde {{\rmE}}}\times_{\rmG} \rmY)_+ \ra \cE \wedge ({\widetilde {{\rmE}}}\times_{\rmG} (\rmY \times \rmX))_+, \\ 
\tr(\rmf_{\rmY})_{\cE}:\cE \wedge  ({\widetilde {{\EG}^{gm}}}\times_{\rmG} \rmY)_+ \ra \cE \wedge ({\widetilde {{\EG}^{gm}}}\times_{\rmG}(\rmY \times \rmX))_+, &\mbox{ when } \rmG \mbox{ is special, and }\notag\\
\tr(\rmf_{\rmY})_{\cE}: \rmR\epsilon_* \epsilon^*(\cE )\wedge \rmR\epsilon_* ({\widetilde {\rmE}}\times_{\rmG}^{et} \rmY))_+ \ra  \rmR\epsilon_* \epsilon^*(\cE) \wedge \rmR\epsilon_* {\widetilde {\rmE}}\times_{\rmG}^{et} (\rmY \times \rmX))_+,\notag \\
\tr(\rmf_{\rmY})_{\cE}: \rmR\epsilon_* \epsilon^*(\cE) \wedge \rmR\epsilon_* ({\widetilde {{\EG}^{gm}}}\times_{\rmG}^{et} \rmY))_+  \ra  \rmR\epsilon_* \epsilon^*(\cE) \wedge \rmR\epsilon_* ({\widetilde {{\EG}^{gm}}}\times_{\rmG}^{et} (\rmY \times \rmX))_+, &\mbox{ when } \rmG \mbox{ is non-special}. \notag
    \end{align} \ee
\end{definition} \qed
\begin{remark} 
\label{transf.homog.spaces}
Suppose $\rmX = \rmG/\rmH$ for a closed linear algebraic subgroup and $\rmY = Spec\, \k$. Then the identification 
\[\rmR\epsilon_*({\widetilde {{\EG}^{\it gm}_{et}}}\times_{\rmG}^{et} \rmG/\rmH)_+ \simeq \rmR \epsilon_*(\BH_{et}^{\it gm}) \simeq \epsilon_*(\colimm {\widetilde {{\BH}^{\it gm,m}}})\]
shows that in this case the target of the transfer map in ~\eqref{tr(f)G.4}
is ${\Sigma^{\infty}_{\T}}({\widetilde {{\BH}^{\it gm}}})_+$ and the target of the transfer map in ~\eqref{tr(f)G.E} is $\cE\wedge ({\widetilde {\BH}^{gm}})_+$ or $\rmR\epsilon_*(\epsilon^*(\cE)\wedge ({\widetilde {\BH}^{gm}})_+$ depending on  whether  $\rmG$ and $\rmH$ are special or not.
 
\end{remark}
\begin{example}{\rm  The following provides a notable class of examples of such a transfer.
Let $i:\rmH \ra \rmG$ be as above and let $\rmY$ denote a smooth quasi-projective scheme over $\k$ 
with an action by $\rmG$. 
Assume further that $\cE$ is a commutative ring spectrum in $\Spt(\k_{\rm mot})$ and $\ell$ is a 
prime different from $char(\k)$ so that $\cE$ is $Z_{(\ell)}$-local.
\vskip .1cm
 Then, the $\rmG$-scheme $\rmG\times_{\rmH} \rmY$ identifies as a $\rmG$-scheme with $\rmG/\rmH \times \rmY$ (provided with the diagonal action by $\rmG$).
Clearly $\rmG/\rmH$ is dualizable in $\Spt(\k_{\rm mot})$ in case $char (\k)= {\rm 0}$ and $\cE \wedge \rmG/\rmH_+$ is dualizable in $\Spt(\k_{\rm mot}, \cE)$ in case $char (\k) =p> {\rm 0}$.  The corresponding transfer, when both $\rmG$ and $\rmH$ are special, is then
the stable map 
\[\tr(id_{\rmY}):{\rm \Sigma^{\infty}_{\T}} (\EG^{\it gm,m}\times _{\rmG} \rmY)_+ \ra {\Sigma^{\infty}_{\T}} ( \EG^{\it gm,m}\times _{\rmG}(\rmG\times_{\rmH} \rmY))_+ \simeq {\Sigma^{\infty}_{\T}} (\EH^{\it gm,m}\times_{\rmH} \rmY)_+\] 
in the first case and  the map 
\[\tr(id_{\rmY}): \cE \wedge (\EG^{\it gm,m}\times _{\rmG} \rmY)_+ \ra \cE \wedge ( \EG^{\it gm,m}\times _{\rmG}(\rmG\times_{\rmH}\rmY))_+ \simeq \cE \wedge (\EH^{\it gm,m}\times_{\rmH} \rmY)_+\]
in the second case. In case the groups $\rmG$ and $\rmH$  are not special, one obtains corresponding stable transfer
maps involving $\rmR\epsilon_*$ as in \eqref{tr(f)G.3} and ~\eqref{tr(f)G.E}.}
\end{example}

\vskip .2cm

\section{Appendix: Spherical fibrations and Thom-spaces  in the motivic and \'etale setting}
\vskip .1cm
The main {\it goal of this section} is to collect together various basic results on Thom spaces of algebraic vector bundles and relate them to Spanier-Whitehead duality in the
both the motivic and \'etale framework. Throughout the following discussion we will let
$\rmS$ denote a Noetherian affine smooth scheme defined and of finite type over a given perfect field $\k$: we will restrict to  smooth schemes of finite type over $\rmS$. 
\subsection{Basic results on Thom-spaces} \index{Thom space}
We begin with the following basic observation on vector bundles over affine schemes. $\Spc_*(\rmS_{\rm mot})$ will denote the category of pointed simplicial
presheaves on the Nisnevich site of $\rmS$ defined as in  ~\ref{Nis.presh}.
\vskip .2cm
\begin{proposition}
	\label{vect.bundles.over.affines}
	\begin{enumerate}[\rm(i)]
	\item Let $\rmX$ denote any affine scheme. Then any vector bundle ${\mathcal E}$ on $\rmX$ has a complement, i.e. there exists another vector bundle
	${\mathcal E}^{\perp}$ so that ${\mathcal E} \oplus {\mathcal E}^{\perp} $ is a trivial vector bundle.
	\item Assume $\rmX$ is again an affine scheme. Then, if ${\mathcal E}$ and ${\mathcal F}$ are two vector bundles on $\rmX$, then they represent the same
	class in the Grothendieck group $\rmK^0(\rmX)$ if and only if they are stably isomorphic, i.e., isomorphic after the addition of some trivial vector bundles.
	\item Let $\rmX$ denote a quasi-projective scheme, i.e., locally closed in some projective space over an affine base scheme $\rmS$. Then there exists an
	affine scheme $\tilde \rmX$ together with a surjective map $\tilde \rmX \ra \rmX$ so that $\tilde \rmX$ is an affine-space bundle  over $\rmX$. In particular,
	the map $\tilde \rmX \ra \rmX$ is an ${\mathbb A}^1$-equivalence.
	\end{enumerate}
\end{proposition}
\begin{proof} (i) is clear from the fact that vector bundles on affine schemes correspond to projective modules over the corresponding
	coordinate ring. (ii) is discussed in \cite[Lemma 2.9]{Voev}.  (iii) is the construction discussed in \cite[Lemme 1.5]{Joun} and often referred
	to as the Jouanolou trick.
\end{proof}
We will next summarize some well-known facts about Thom-spaces in $\Spc_*({\rmS}_{\rm mot})$.
If $\alpha$ is a vector bundle over a smooth scheme $\rmX$ over the (base) scheme $\rmS$, then one needs to define the Thom-space of $\alpha$ to be the 
following canonical homotopy  pushout:
\be \begin{equation}
\label{Thom.space.over.S}
\xymatrix{{\rm{\rm E}(\alpha)-\rmX} \ar@<1ex>[d]\ar@<1ex>[r] & {\rmE(\alpha)} \ar@<1ex>[d]\\
	{\rmS} \ar@<1ex>[r] & {{\rm Th}(\alpha)}}
\end{equation} \ee
where $\rmE(\alpha)$ denotes the total space of the vector bundle $\alpha$. Since $\rmE(\alpha)$ and $\rmE(\alpha)-\rmX$ map to $\rmX$ and then to $\rmS$, ${\rm Th}(\alpha)$ maps to $\rmS$. The map $\rmS \ra {\rm Th}(\alpha)$ provides a section to the induced map ${{\rm Th}(\alpha)} \ra \rmS$, so that ${\rm Th}(\alpha)$ is pointed over $\rmS$ and hence is an object in $\Spc_*({\rmS}_{\rm mot})$.  
(We may often assume that injective maps are cofibrations, in which case the map in the top row is a cofibration, and therefore, it suffices to
take the ordinary pushout, in the place of the homotopy pushout.) 
\vskip .1cm
When we view $\rmE(\alpha)$ and $\rmE(\alpha)-\rmX$ as sheaves on the big \'etale site, the corresponding pushout of sheaves on the big \'etale site of $\rmS$ will be denoted
${\rm Th}(\alpha)_{et}$. 
\vskip .1cm
\begin{proposition}
	\label{props.Thom.spaces} Let $\alpha$ denote a vector bundle over the scheme $\rmX$.
	\begin{enumerate}[\rm(i)]
	\item Viewing $\rmP(\alpha \oplus \epsilon^1) ={\rm {Proj}}_{\rmX}(\alpha \oplus \epsilon^1)$ and $ \rmP(\alpha) ={\rm {Proj}}_{\rmX}(\alpha)$ as simplicial presheaves over the base scheme $\rmS$ and taking the quotient presheaf, 
	$\rmP(\alpha \oplus \epsilon^1) /\rmP(\alpha) \simeq {\rm Th}(\alpha)$ where $\rmP(\beta)$ denotes the projective space bundle associated to a vector bundle
	$\beta$ and $\epsilon^1$ denotes a trivial bundle of rank $1$. Viewing $\rmP(\alpha \oplus \epsilon^1)$ and $\rmP(\alpha)$ as simplicial presheaves over 
	$\rmX$ and taking the quotient presheaf over $\rmX$, 
	$\rmP(\alpha \oplus \epsilon^1) /_{\rmX}\rmP(\alpha) \simeq S(\alpha \oplus \epsilon^1)$, a {\it sphere bundle over $\rmX$}. (This may be called the ``{\it one-point compactification
		of the vector bundle $\alpha$}''.)  The obvious projection $\rmS(\alpha \oplus \epsilon^1) \ra X$ has a section s that sends a point in $\rmX$ to the 
	point at $\infty$ in the fiber over that point. Now $\rmS(\alpha \oplus \epsilon ^1)/s(\rmX) \cong {\rm Th}(\alpha)$.
	\item If $ \rmX \ra \rmY$ is a closed immersion of smooth schemes with $\cN$ denoting the corresponding normal bundle, then ${\rm Th}(\cN) \simeq \rmX/\rmX-\rmY$.
	\item Let $g:\rmS' \ra \rmS$ denote a map of smooth schemes and let $g^*(\alpha)$ denote the induced vector bundle on $\rmX'= \rmX{\underset {\rmS} \times}\rmS'$. Then $g$ induces a map ${\rm Th}(g^*(\alpha)) \ra {\rm Th}(\alpha)$
	compatible with the given map $g:\rmS' \ra \rmS$. Moreover, the induced map ${\rm Th}(g^*(\alpha)) \ra {\rm Th}(\alpha)$ is natural in $g$ and $\alpha$.
	\end{enumerate}
	\end{proposition}
\begin{proof} We skip the proof as the above statements are rather well-known. 
\end{proof}
\begin{notation}
	When we view $\rmP(\alpha)$ and $\rmP(\alpha \oplus \epsilon^1)$ as sheaves on the big \'etale site of $\rmX$ the corresponding
	quotient $\rmP(\alpha \oplus \epsilon^1)_{et}/_{\rmX}\rmP(\alpha)_{et}$ will be denoted $\rmS(\alpha \oplus 1)_{et}$.
\end{notation}
\subsection{The fiber-wise join of simplicial presheaves (spectra) fibering over another simplicial presheaf (spectrum)}
\label{AppA.fiberwise.join}
Given maps of simplicial presheaves $\rmY \ra \rmX $ and $\rmZ \ra \rmX$, the {\it fiber-wise join} $\rmY*_{\rmX}\rmZ$ is the simplicial presheaf defined as the (canonical) homotopy pushout
\be \begin{equation}
\label{fib.join}
\xymatrix{{\rmY{\underset {\rmX} \times} \rmZ} \ar@<1ex>[r] \ar@<1ex>[d] & \rmZ \ar@<1ex>[d]\\
	\rmY \ar@<1ex>[r] & {\rmY*_{\rmX}\rmZ}}
\end{equation} \ee
One may readily verify that the above construction extends readily to spectra.
We elaborate a bit on the above construction and its application to Thom-spaces. First we show that the fiber-wise join indeed does
what it is supposed to do.
\begin{lemma}
	\label{AppA.fiberwise.join.1}
	Assume $\rmX$, $\rmY$ and $\rmZ$ are simplicial presheaves as above. Then: 
	\begin{enumerate}[\rm(i)]
	\item there is an induced map $\rmY*_{\rmX}\rmZ \ra \rmX$. 
	\item If $\rmY_x, \rmZ_x$ denote the fibers over ${\rm x} \in \rmX$, $(\rmY*_{\rmX}\rmZ)_x \simeq {\rm S}^1 \wedge (\rmY_x \wedge \rmZ_x)$.
	\item Therefore, if $\rmX$ denotes the simplicial presheaf represented by a smooth scheme, and $\rmY$, $\rmZ$ are pointed simplicial presheaves over $\rmX$, $\rmY*_{\rmX}\rmZ \simeq ({\rm S}^1\times \rmX)\wedge^{\rmX} (\rmY\wedge ^{\rmX} \rmZ) \simeq 
	\rmY \wedge ^{\rmX} ((\rmS^1 \times \rmX) \wedge ^{\rmX} Z)$.
	\end{enumerate}
\end{lemma}
\begin{proof} This is a well-known result. See for example, \cite[Lemma 2.1]{CS}. 
\end{proof}
\vskip .2cm
If $\beta$ is a vector
bundle over the smooth scheme $\rmX$, we let ${\overset o \beta} \ra X$ denote the associated bundle $\beta -0 \ra X$. Now with 
$\obeta_+ = \obeta \sqcup X$, one obtains $(\rmS^1 \times \rmX) \wedge^{\rmX} \obeta_+ \simeq (\beta {\underset  {\obeta} {\sqcup} } X ) \cong \rmS(\beta \oplus 1)$.
The last $\cong$ is an isomorphism as simplicial presheaves over $\rmX$ while the $\simeq$ is a weak-equivalence
of such simplicial presheaves. The last isomorphism may be seen by working locally on $\rmX$, so that $\beta$
is trivial. The $\simeq$ follows from the observation that the fibers of $\beta $ are acyclic so that
$(\rmS^1 \times \rmX) \wedge^{\rmX} \obeta_+ \simeq ( \beta {\underset {\obeta}{\sqcup}} \rmX)$ as simplicial presheaves over $\rmX$.

\begin{lemma}
	\label{AppA.fiberwise.join.2}
	Let $\alpha$ and $\beta$ denote two vector bundles over the scheme $\rmX$. Then we obtain the identifications:
	\begin{enumerate}[\rm(i)]
	\item $\rmS(\alpha \oplus \epsilon^1) *_{\rmX} {\overset o \beta}_+ \simeq \rmS(\alpha \oplus \epsilon^1) \wedge^{\rmX} S(\beta \oplus \epsilon^1) \simeq \rmS(\alpha \oplus \beta \oplus \epsilon^1)$ where ${\overset o \beta}$ denotes $\beta -0 \ra X$, the
	associated sphere bundle.
	\item The map $\rmS(\alpha \oplus \beta \oplus \epsilon^1) \ra X$ has a section $s$ sending each point of $\rmX$ to the point at $\infty$ in the
	fiber over that point. Then the quotient $\rmS(\alpha \oplus \beta \oplus \epsilon^1)/s(\rmX) = {\rm Th}(\alpha \oplus \beta)$, which is the 
	Thom-space of $\alpha \oplus \beta$.
	\end{enumerate}
\end{lemma}
\begin{proof} Since (ii) is rather straightforward, we will discuss only (i).
	In view of the weak-equivalences above between $(\rmS^1 \times \rmX) \wedge^{\rmX} \obeta_+$ and $\rmS(\beta \oplus \epsilon^1)$, and the 
	observation that $\wedge ^{\rmX}$ is a homotopy pushout of simplicial presheaves over $\rmX$ (in the injective model
	structure), it follows that
	$\rmS(\alpha \oplus \epsilon^1) *_{\rmX} \obeta_+ \simeq S(\alpha \oplus \epsilon^1) \wedge^{\rmX} ((S^1 \times \rmX) \wedge^{\rmX} \obeta_+) \simeq
	\rmS(\alpha \oplus \epsilon^1) \wedge ^{\rmX} \rmS(\beta \oplus \epsilon^1)$. Since the last fibers over $\rmX$, one may work locally on
	$\rmX$ and show readily that it identifies with $\rmS(\alpha \oplus \beta \oplus \epsilon^1)$.
\end{proof}
\begin{lemma}
	\label{rel.smash} 
	Let $\rmB = Spec\, \k$ denote the base field.  Let $\group$ denote a linear algebraic group defined over $\rmB$ and acting on the simplicial presheaves $\rmE$ and $\rmX$ over the base scheme $\rmB$. Assume that
	$\rmE$ is in fact a smooth scheme of finite type over $\rmB$ so that the (geometric) quotient $\rmE/G$ exists and is in fact a scheme of finite type over $\rmB$. Let $\rmP$ denote
	a pointed simplicial presheaf in $\Spc_*(\rmB)$ together with a $\group$-action that leaves the base point of $\rmP$ fixed. Then
	$\rmE\times_{\rmG}(\rmP \wedge \rmX_+) \cong \rmP_{\rmE/\rmG}\wedge^{\rmE/\rmG} (\rmE \times _{\rmG} \rmX_+)$, where $\rmP_{\rm E/\rmG} = \rmP\times_{\rmG}\rmE$.
\end{lemma}
\begin{proof} The proof is skipped as one may readily verify the above conclusions.
\end{proof}
\subsection{Motivic Atiyah duality} \index{Atiyah duality}
\label{motivic.Atiyah}
The rest of this section will be devoted to summarizing  a version of Atiyah-duality (see \cite{At}) that applies to the motivic and also the \'etale context: accordingly, we will assume 
 that for any smooth projective 
scheme $\rmX$ over a perfect field $\k$, there exists a vector bundle over the scheme $\rmX$ (which we call the {\it virtual normal bundle})
so that the $\T$-suspension spectrum of its Thom-space is a Spanier-Whitehead dual of the  suspension spectrum  ${\Sigma^{\infty}_{\T}}\rmX_+$  in the category $\Spt({\k}_{\rm mot})$.
 The idea of the proof may be summarized as follows: the Voevodsky collapse considered in Definition ~\ref{V.collapse} (see below), provides a 
 co-evaluation map $c:\T^{\rm n} \ra \rm\rmX_+ \wedge {\rm Th}(\nu_{\rmX})$. Here $\nu_{\rmX}$ is the {\it virtual normal bundle} considered 
 in Definition ~\ref{V.collapse}. One defines an evaluation map dual to this and with these, one shows that ${\rm Th}(\nu_{\rmX})$ is
 in fact a Spanier-Whitehead dual of $\Sigma_{\T}\rmX_+$ modulo certain shifts.
\vskip .2cm
Under the assumption that the base scheme $\rmS =\k$ is a perfect field satisfying the finiteness hypothesis as in ~\eqref{etale.finiteness.hyps}, we may readily 
observe that the pullback functors
$\epsilon^*:\Spt(\k_{\rm mot}) \ra \Spt(\k_{et}),  \bar \epsilon^*:\Spt(\bar \k_{\rm mot}) \ra \Spt(\bar \k_{et}), \mbox{ and } \eta^*: \Spt(\k_{et}) \ra \Spt(\bar \k_{et})$ considered in 
 ~\eqref{maps.topoi.2} as well as  the functors $\theta$ and $\phi _{\cE}$ (discussed in the paragraph below
 ~\eqref{maps.topoi.2}) send suspension spectra of the Thom-spaces in the framework of the source,  to suspension spectra of the corresponding Thom-spaces in the  framework of the target, are compatible
with the smash-products and internal Homs in these categories and also
send maps that are homotopic to the identity to maps that are homotopic to the identity. Therefore, the discussion below carries over from
the framework of $\Spt(\k_{\rm mot})$ to all of the other frameworks (at least after inverting ${\mathbb A}^1$ in all these frameworks). Thus,
the construction of a Spanier-Whitehead dual from the Thom-space of a vector bundle worked out below in the motivic framework 
carries over to the \'etale setting after smashing with an $\ell$-complete spectrum, $\ell$ being prime to the characteristic.
\vskip .2cm
Over  algebraically closed fields of arbitrary characteristic, there is already a different construction
valid in the \'etale setting and making strong of use of \'etale tubular neighborhoods: see 
\cite{J86} and \cite{J87}.
\vskip .2cm
\begin{definition} (The diagonal map.)
 \label{Delta}
Next we consider the following {\it diagonal map}. Let $ \alpha, \beta$ denote two vector bundles on the scheme $\rmX$. Then there is a diagonal map
${\rm Th}(\alpha \oplus \beta) \ra {\rm Th}(\alpha) \wedge {\rm Th}(\beta)$. This map is induced by the map ${\rm E}(\alpha \oplus \beta) \ra {\rm E}(\alpha) \times {\rm E}(\beta)$
lying over the diagonal map $\rmX \ra X \times X$. In this case, one may verify that $\rmE(\alpha \oplus \beta)-\{0\}$ maps
to  $ (\rmE(\alpha)-\{0\}) \times \rmE(\beta) \cup \rmE(\alpha) \times (\rmE(\beta)-\{0\})$.  Taking $\alpha$ to be a zero-dimensional bundle, one obtains the diagonal map
\be \begin{equation}
   \label{diag.map.0}
\Delta: {\rm Th}(\beta) \ra  \rmX_+ \wedge {\rm Th}(\beta) .
\end{equation} \ee
\vskip .2cm
One may interpret the above diagonal map in terms of the associated disk and sphere bundles as follows:
\vskip .2cm
$\Delta':{\rm Th}(\beta) = \rmP(\beta \oplus \epsilon ^1) /\rmP(\beta) \ra \rmP(\beta \oplus \epsilon^1)_+ \wedge \rmP(\beta \oplus \epsilon ^1)/\rmP(\beta)  
= (\rmP(\beta \oplus \epsilon^1) \times \rmP(\beta \oplus \epsilon ^1))/\rmP(\beta \oplus \epsilon^1) \times \rmP(\beta )$
\vskip .2cm \noindent
Now one composes with the projection $\rmP(\beta \oplus \epsilon ^1) \ra \rmX$ to define the diagonal map in ~\eqref{diag.map.0}.
\end{definition}

\subsection{Basic framework: the projective case}
\label{projective.case}
Assume next that $\rmX$ and $\rmY$ are smooth projective schemes with $\rmX$ provided with a closed immersion into $\rmY$ over $\k$. 
$\rmY$ will usually denote a projective space over $\k$, but we denote it by $\rmY$ for simplicity of notation.
Let $\tau_{\rmX}$ ($\tau_Y$, $\cN$) denote the tangent bundle to $\rmX$ (the tangent bundle to $\rmY$ and the normal bundle associated to the imbedding of $\rmX$ in $\rmY$, \resp).
Then one 
 obtains the short exact sequence 
\be \begin{equation}
     \label{tangent.normal}
0 \ra \tau_{\rmX} \ra \tau_{\rmY|\rmX} \ra \cN \ra 0.
\end{equation} \ee
 Let $\pi_{\rmY}: \tilde \rmY \ra \rmY$ denote the affine replacement provided by Jouanolou's construction. 
 Let $\tilde \rmX = \rmX \times_{\rmY}\tilde \rmY$ and let  $\pi_{\rmX}: \tilde \rmX \ra \rmX$ denote the induced map. Then the following are proven in
\cite[Proposition 2.7 through Theorem 2.11]{Voev}:
 \begin{enumerate}
     \item There exists a vector bundle $\rmV$ on $\rmY$ so that $\pi_{\rmY}^*(\rmV) \oplus \pi_{\rmY}^*(\tau_{\rmY})$ is stably isomorphic to a trivial vector bundle.
So we will assume that $\pi_{\rmY}^*(\rmV) \oplus \pi_{\rmY}^*(\tau_{\rmY}) \oplus \epsilon^m \cong \epsilon^{n}$ for some ${\rm m}$ and ${\rm n}$. We will replace $\rmV$ by $\rmV \oplus 
\epsilon ^{\rm m}$ so that $\pi_{\rmY}^*(\rmV) \oplus \pi_{\rmY}^*(\tau_{\rmY})$ is the trivial bundle $\epsilon^{\rm n}$.
\item There exists a collapse map $\rmV:\T^{n} \ra {\rm Th}(\rmV)$. See \cite[Lemma 2.10 and Theorem 2.11]{Voev}.
    \end{enumerate}
\vskip .1cm
One may observe that  $\pi_{\rmX}^*(\cN \oplus \rmV_{|X}) \oplus \pi_{\rmX}^*(\tau_{\rmX})$ is also stably trivial. If 
$\pi_{\rmX}^*(\cN \oplus \rmV_{|X}) \oplus \pi_{\rmX}^*(\tau_{\rmX}) \oplus \epsilon^m \cong \epsilon^{\rm N}$ for some $m$ and ${\rm N}$, we will replace $\rmV$ by $\rmV \oplus \epsilon^{\rm m}$ and
we will make the following definition.
\begin{definition} 
	\label{V.collapse}
	({\it Virtual normal bundle in the projective case, the Voevodsky collapse and the corresponding co-evaluation map})
\label{virt.norm.bundle.proj} \index{Virtual normal bundle} \index{Voevodsky collapse}
We let  $\nu_{\rmX}= \cN \oplus \rmV_{|\rmX}$ and call it {\it the virtual normal bundle}
to $\rmX$ in $\rmY$. Taking $\rmY=\rmX$, we see that $\nu_{Y}$ has the property that $\pi_{\rmY}^*(\nu_{\rmY})$ is a complement to $\pi_{\rmY}^*(\tau_{\rmY})$ in some
trivial bundle over $\tilde \rmY$. 
\vskip .2cm
Clearly ${\rm Th}(\nu_{\rmX}) \simeq \rmV/\rmV-\rmX$ where $\rmX$ is imbedded in $\rmV$ by the composite imbedding $\rmX \ra \rmY {\overset {0-section} \ra} {\rm E}(\rmV)$. Therefore, one obtains
a collapse map ${\rm Th}(\rmV) = \rmV/\rmV-\rmY \ra \rmV/\rmV-\rmX \simeq {\rm Th}(\nu_{\rmX})$. Composing with the collapse $\rmV: \T^{\rm n} \ra {\rm Th}(\rmV)$ one obtains the collapse $\rmV_{\rmX}:\T^{\rm n} \ra {\rm Th}(\nu_{\rmX})$. Composing with the diagonal map $\Delta$ (considered above), one obtains a map $c:\T^{\rm n} \ra \rm\rmX_+ \wedge {\rm Th}(\nu_{\rmX})$. The main result we need is that this map is
indeed a co-evaluation map in the sense of \cite[1.3 Theorem]{DP}, so that ${\Sigma^{\infty} \wedge {\T}}^{\rm -n} \wedge{\rm Th}(\nu_{\rmX})$ is indeed a Spanier-Whitehead dual of 
 ${\Sigma^{\infty}_{\T}}\rm\rmX_+$. This is rather well-known
 by now, as discussed for example in  \cite{Hu-Kr}. 
\end{definition}




\vskip .4cm
\begin{thebibliography}{MMMMM}

\bibitem[At]{At} M. F. Atiyah, {\it Thom complexes},   Proc. London Math. Soc.
(3),  11,  (1961),  291--310.




\bibitem[BG75]{BG75}J. Becker and D. Gottlieb, {\it The transfer map and fiber
bundles}, Topology, {\bf 14}, (1975), 1-12. 



 
 \bibitem[Bor]{Bor} F. Borceux. \textit{Handbook of categorical algebra. II: Categories and structures},  Cambridge University
Press, Cambridge, (1994).



\bibitem[CE]{CE} H. Cartan and S. Eilenberg, \emph{Homological Algebra}, Princeton University Press, (1980).

\bibitem[Ch]{Ch} S\'eminaire  C. Chevalley, 2 ann\'ee, {\it Anneaux de Chow et applications},  Paris: Secretariat
math\'ematique, (1958).



\bibitem[CJ23-T2]{CJ23-T2} G. Carlsson and R. Joshua, \textit{The motivic and \'etale Becker-Gottlieb transfer and splittings},
Preprint, (2023).



\bibitem[CS]{CS} D. Shatur and J. Scherer, \textit{Fiberwise nullification and the cube theorem},
arXiv:math/303062v1 [math.AT], March, 2003.


\bibitem[Day]{Day} B. Day, \emph{On closed categories of functors}, In {\it Reports of the
Midwest category seminar}, IV, 1-38, Springer, Berlin, (1970).




\bibitem[DHI]{DHI} D. Dugger S. Hollander and D. Isaksen, \textit{Hypercovers and simplicial presheaves},
Math. Proc. Camb. Phil. Soc, \textbf{136},  (2004), 9-51.



\bibitem[DP]{DP} A. Dold and V. Puppe, \textit{Duality, Traces and Transfer},
Proc. Steklov Institute, (1984), 85-102.


\bibitem[DRO1]{DRO1} B. Dundas, O. Rondigs, P. $\varnothing$stv\ae{}r, \textit{Enriched functors
and stable homotopy theory}, Documenta. Math, \textbf{8}, (2003), 409-488.

\bibitem[DRO2]{DRO2} B. Dundas, O. Rondigs, P. $\varnothing$stv\ae{}r, \textit{Motivic functors}, 
Documenta. Math, \textbf{8}, (2003), 489-525.


\bibitem[Dug]{Dug}D. Dugger, \textit{Universal homotopy theories}, Adv. in Math,
\textbf{164}, (2001), 144-176.




\bibitem[Guill]{Guill}B. Gillou, \textit{A short note on models for equivariant
homotopy theory}, Preprint, (2009).
\bibitem[Hirsch]{Hirsch} P. Hirschhorn, \textit{ Localization of Model categories},
AMS, (2002).

\bibitem[Hirsch15]{Hirsch15} P. Hirschhorn, \textit{Overcategories and Undercategories of Model categories}, Preprint, (2015), arXiv:1507.01624v1 [Math.AT] 6 Jul 2105.
 
\bibitem[Hirsch14]{Hirsch14} P. Hirschhorn, \textit{Notes on Homotopy colimits and limits}, Preprint,  (2014).

\bibitem[Ho]{Ho} M. Hoyois, \textit{The six operations in equivariant motivic homotopy theory}, Advances in Math, \textbf{305}, (2017), 197-279.


\bibitem[HKO]{HKO}M. Hoyois, S. Kelley and P. Ostv\ae{}r, \textit{The motivic Steenrod algebra over perfect fields}, Preprint  (2012).



\bibitem[Hov01]{Hov01} M. Hovey, \textit{Spectra and symmetric spectra in
general model categories},  J. Pure Appl. Algebra 165 (2001), no. 1, 63–127.

\bibitem[Hov99]{Hov99} M. Hovey, \textit{Model categories}, AMS Math surveys and
monographs, {\bf 63}, AMS, (1999).




\bibitem[Hov03]{Hov03} M. Hovey, \textit{Monoidal model categories}, Math
AT/9803002, (2003).

\bibitem[HSS]{HSS} M. Hovey, B. Shipley and J. Smith, \textit{Symmetric spectra}, JAMS, \textbf{13}, (2000), no.1, 149-208.





\bibitem[Hu-Kr]{Hu-Kr} P. Hu and I. Kriz, \textit{Appendix A: On the Picard group of the stable ${\mathbb A}^1$ homotopy category}, Topology, \textbf{44} (2005),
609-640.

\bibitem[Isak]{Isak} D. C. Isaksen, \emph{\'Etale realization on the ${\mathbb A}^1$-homotopy theory of schemes}, Advances in Math, \textbf{184}, (2004), 37-63.
 


\bibitem[Joun]{Joun} J. P. Jouanolou, \textit{Une suite exacte de Mayer-Vietoris en K-Theorie algebrique},  Lect. Notes
in Math., \textbf{341}, 293-316, Springer, (1973).

\bibitem[J22]{J22} R. ~Joshua, \textit{Equivariant Derived Categories for Toroidal Group Imbeddings}, Transformation Groups,
issue 1,( 113-162),  (2022). 

\bibitem[JP23]{JP23} R. ~Joshua and P. ~Pelaez, \textit{Additivity of motivic trace and the motivic Euler-characteristic}, Advances In Math, 
\textbf{429}, (2023), 109184.

\bibitem[J01]{J01}R. Joshua, \textit{Mod-$l$ algebraic K-theory and Higher Chow
groups of linear varieties}, Camb. Phil. Soc. Proc., 130, (2001), 37-60.


\bibitem[J02]{J02} R. ~Joshua, \textit{Derived functors for maps of simplicial
spaces}, JPAA, {171},  (2002), 219-248.

\bibitem[JT]{JT} R. ~Joshua, \textit{Mod-$l$ Spanier-Whitehead duality and the
Becker-Gottlieb transfer in \'etale homotopy
and applications}, Ph. D thesis, Northwestern University, (1984).

\bibitem[J86]{J86} R. ~Joshua, \textit{Mod-$l$ Spanier-Whitehead duality in
\'etale homotopy}, Transactions
 of the AMS, \textbf{296}, 151-166, (1986).

\bibitem[J87]{J87} R. ~Joshua, \textit{Becker-Gottlieb transfer in \'etale
homotopy theory}, Amer. J. Math., \textbf{107}, 453-498, (1987).


\bibitem[K]{K} S. Kelley, {\it Triangulated categories of motives in positive characteristics}, Ph. D Thesis,
Universit\'e de Parix XIII, arXiv:1305.5349v1, (2013).


\bibitem[Lev]{Lev} M. Levine, {\it The homotopy coniveau tower}, Journal of  Topology, (2008), \textbf{1},  217-267.

\bibitem[Lev18]{Lev18} M. Levine, {\it Motivic Euler Characteristics and Witt-valued Characteristic classes},  Nagoya Math Journal,(2019). DOI:https://doi.org/10.1017/nmj.2019.6.

\bibitem[LMS]{LMS}L. G. Lewis, J. P. May, and M. Steinberger, \textit{Equivariant Stable Homotopy Theory},
 Lect. Notes in Mathematics, \textbf{1213}, Springer, (1985).


\bibitem [Lur]{Lur} J. Lurie, \textit{Higher Topos Theory}, Annals of Math Study, \textbf{170}, Princeton University Press, 
(2009).


\bibitem[Mil]{Mil} J. Milne, \emph{\'Etale Cohomology}, Princeton University Press, (1980).


\bibitem [MV] {MV} F. Morel and V. Voevodsky, \textit{${\mathbb A}^1$-homotopy
theory of schemes}, I. H. E. S Publ. Math., {90}, (1999), 45--143 (2001).


\bibitem[ncatlab] {ncatlab} https://www.ncatlab.org/nlab/show/monoidal+functor.



\bibitem[RO]{RO} O. Rondigs and P. Ost\ae{}ver, \textit{Modules over motivic cohomology},
Advances in Math., \textbf{219}, (2008), 689-727.


\bibitem[Ri05]{Ri05} J. Riou, \textit{Dualit\'e de Spanier-Whitehead en G\'eometrie Algébriques},
 C. R.  Math. Acad. Sci. Paris, \textbf{340}, (2005), no. 6, 431-436.

 \bibitem[Ri13]{Ri13} J. Riou, \textit{$\ell'$ Alterations and Dualizability}, Appendix B to
 Algebraic Elliptic Cohomology Theory and Flops I, by M. Levine, Y. Yang and G. Zhao, Preprint, (2013).


\bibitem[SSch]{SSch} B. Shipley and S. Schwede, \textit{Algebras and modules in 
monoidal model categories}, Proc. London Math Soc, \textbf{80}, 491-511, (1998).


\bibitem[SpWh58]{SpWh58}E. Spanier and J. H. C. Whitehead, {\it Duality in
relative homotopy theory},   Ann. of Math. (2),  {\bf 67} ,  (1958),  203--238.


\bibitem[Tot]{Tot} B. Totaro, \textit{ The Chow ring of a classifying space},
Algebraic K-theory (Seattle, WA, 1997), 249-281, Proc. Symposia in Pure Math,
{\bf 67}, AMS, Providence, (1999).




\bibitem[Voev]{Voev} Vladimir Voevodsky: {\it Motivic cohomology with ${\mathbb
Z}/2$-coefficients},  Publ. Math. Inst. Hautes Études Sci.  No. 98  (2003),
59-104.
\end{thebibliography}
\end{document}